\documentclass[preprint]{imsart}

\startlocaldefs
\usepackage{t1enc}

\usepackage[latin2]{inputenc}
\usepackage{amssymb}
\usepackage{amsmath,amsthm}
\usepackage{color}
\usepackage{mathrsfs}
\usepackage{latexsym}
\usepackage{epsfig}
\usepackage[a4paper,left=4cm,right=3cm,top=3cm,bottom=3cm]{geometry}
\usepackage{subfigure}
\usepackage{hyperref}

\newcommand{\sumin}{\sum_{i=1}^n}
\def\g{\gamma}

\long\def\beginskip#1\endskip{}
\def\endskip{}
\newcommand{\diamm}{\text{diam$_n$}}
\newcommand{\diam}{\mathop{\rm diam}\nolimits}
\newcommand{\ra}{\rightarrow}
\newcommand{\da}{\downarrow}

\renewcommand{\k}{\kappa}

\newcommand{\E}{\mathord{\rm E}}
\newcommand{\Var}{\mathop{\rm var}\nolimits}
\let\var\Var
\newcommand{\sd}{\mathop{\rm sd}\nolimits}
\newcommand{\eps}{\varepsilon}

\newcommand{\RR}{\mathbb{R}}

\renewcommand{\phi}{\varphi}

\newcommand{\given}{\,|\,}
\renewcommand{\l}{\lambda}

\renewcommand{\th}{\theta}
\renewcommand{\t}{\tau}
\newcommand{\z}{\zeta}
\newcommand{\e}{\varepsilon}

\renewcommand{\d}{\delta}

\renewcommand{\a}{\alpha}
\newcommand{\argmax}{\mathop{\rm argmax}}

\newcommand{\mmle}{\widehat\tau_M}
\newcommand{\1}{\textbf{1}}

\newcommand{\cS}{\mathcal{S}}
\newcommand{\cL}{\mathcal{L}}
\newcommand{\cM}{\mathcal{M}}

\newtheorem{theorem}{Theorem}[section]

\newtheorem{proposition}[theorem]{Proposition}
\newtheorem{lemma}[theorem]{Lemma}

\theoremstyle{definition}
\newtheorem{remark}[theorem]{Remark}
\newtheorem{ex}[theorem]{Example}

\newtheorem{cond}{Condition}
\newtheorem{assumption}{Assumption}
\numberwithin{equation}{section}

\graphicspath{{./Figures/}}

\endlocaldefs

\begin{document}

\begin{frontmatter}

\title{Uncertainty quantification for the horseshoe}
\runtitle{Uncertainty quantification for the horseshoe}


\author{\fnms{St\'ephanie} \snm{van der Pas}\ead[label=e1]{svdpas@math.leidenuniv.nl}\thanksref{t1,m1}},
\author{\fnms{Botond} \snm{Szab\'o}\ead[label=e2]{b.t.szabo@math.leidenuniv.nl}\thanksref{t1,t2,m1,m2}}
\and
\author{\fnms{Aad} \snm{van der Vaart}\ead[label=e3]{avdvaart@math.leidenuniv.nl}\thanksref{t2,m1}}
\address{Leiden University\thanksmark{m1} 
and Budapest University of Technology and Economics\thanksmark{m2}\\
\printead{e1}\\ \printead{e2}\\\printead{e3}}

\thankstext{t1}{Research supported by the Netherlands Organization for Scientific Research.}
\thankstext{t2}{The research leading to these results has received funding from the European
  Research Council under ERC Grant Agreement 320637.}

\runauthor{van der Pas et al.}

\begin{abstract}
We investigate the credible sets and marginal credible intervals resulting from the horseshoe prior in the sparse multivariate normal means model. We do so in an adaptive setting without assuming knowledge of the sparsity level (number of signals). We consider both the hierarchical Bayes method of putting a
prior on the unknown sparsity level and the empirical Bayes method with the sparsity level estimated
by maximum marginal likelihood. We show that credible balls and marginal credible intervals have
good frequentist coverage and optimal size if the sparsity level of the prior is 
set correctly. By general theory honest confidence sets cannot adapt in size to an unknown
sparsity level. Accordingly the hierarchical and empirical Bayes credible sets based on the
horseshoe prior are not honest over the full parameter space. We show that this
is due to over-shrinkage for certain parameters and characterise the set
of parameters for which credible balls and marginal credible intervals
do give correct uncertainty quantification. In particular we show that
the fraction of false discoveries by the marginal Bayesian procedure is controlled
by a correct choice of cut-off.
\end{abstract}



\end{frontmatter}

\section{Introduction}
Despite the ubiquity of problems with sparse structures, and the large amount of research effort into finding consistent and minimax optimal estimators for the underlying sparse structures \cite{Tibshirani1996, Johnstone2004, Castillo2012,
  Castillo2015, Jiang2009,Griffin2010, Johnson2010, Ghosh2015, Caron2008, Bhattacharya2014, Bhadra2015,
  Rockova2015}, the number of options for uncertainty quantification in the sparse normal means problem is very limited. In this paper, we show that the horseshoe credible sets and intervals are effective tools for uncertainty quantification, unless the underlying signals are too close to the universal threshold in a sense that is made precise in this work. We first introduce the sparse normal means problem, and our measures of quality of credible sets.

The sparse normal means problem, also known as the sequence model, is frequently studied and considered as a test case for sparsity methods, and has some applications in, for example, image processing (\cite{Johnstone2004}). A random vector $Y^n = (Y_1, \ldots, Y_n)$ of observations, taking  values in $\mathbb{R}^n$, is modelled as the sum of  fixed means and noise:
\begin{equation}
\label{EqObservation}
Y_i = \th_{0,i} + \varepsilon_i, \quad i = 1, \ldots, n,
\end{equation}
where the $\varepsilon_i$ follow independent standard normal distributions. The sparsity assumption made on the mean vector $\th_0 =(\th_{0,1},\ldots,\th_{0,n})$ is that it is nearly black, which stipulates that most of the means are zero, except for $p_n = \sum_{i=1}^n\1\{\th_{0,i} \neq 0\}$ of them. The sparsity level $p_n$ is unknown, and assumed to go to infinity as $n$ goes to infinity, but at a slower rate than $n$: $p_n \to \infty$ and $p_n = o(n)$.

This paper studies the Bayesian approach based on the \emph{horseshoe prior}
\cite{Carvalho2010, Carvalho2009, Scott2011, Polson2012, Polson2012-2}. The horseshoe prior is popular due to its good performance in simulations and under theoretical study (e.g. \cite{Carvalho2010, Carvalho2009, Polson2012,
  Polson2010, Bhattacharya2014, Armagan2013, vdPas, Datta2013}). The horseshoe prior is a scale mixture of normals, with a half-Cauchy prior on the variance. It is given by
\begin{equation}
\begin{aligned}
\th_i \given \l_i, \t &\sim \mathcal{N}(0,\l_i^2\t^2),\\
\l_i &\sim C^+(0,1), \qquad i = 1, \ldots, n.
\end{aligned}
\label{EqHorseShoePrior}
\end{equation}
The ``global hyperparameter" $\t$ was determined to be important towards the minimax optimality of the horseshoe posterior mean as an estimator of $\theta_0$ (\cite{vdPas}). The results in \cite{vdPas} show that $\t$ can be interpreted as the proportion of nonzero parameters, up to a logarithmic factor. If it is set at a value of the order $(p_n/n)\sqrt{\log{n/p_n}}$, then the horseshoe posterior contracts around the true $\th_0$ at the (near) minimax estimation rate for quadratic loss.  Adaptive posterior contraction, where the number $p_n$ is not assumed known but estimated by empirical Bayes or hierarchical Bayes as in this paper,  was proven for estimators of $\tau$ that are bounded above by $(p_n/n)\sqrt{\log{n/p_n}}$ with high probability in \cite{contractionpaper}.

The adaptive concentration of the horseshoe posterior is encouraging towards the usefulness of the horseshoe credible balls for uncertainty quantification, as in the Bayesian framework the spread of the posterior distribution
over the parameter space is used as an indication of the error in estimation. It follows from general results of \cite{Li, RobinsvdV,Nickl2013} that honest uncertainty quantification 
is irreconcilable with adaptation to sparsity. Here \emph{honesty} of confidence sets $\hat{C}_n=\hat{C}_n(Y^n)$ relative to
a parameter space $\tilde \Theta\subset\RR^n$ means that
$$\liminf_{n\rightarrow\infty}\inf_{\th_0\in \tilde\Theta}P_{\th_0}(\th_0\in\hat{C}_n)\geq 1-\alpha,$$
for some prescribed confidence level $1-\alpha$. Furthermore, \emph{adaptation} to a partition 
$\tilde\Theta=\cup_{p\in P}\Theta_{p}$ of the parameter space into submodels $\Theta_p$ 
indexed by a hyper-parameter $p\in P$, means that, for every $p\in P$ and
for $r_{n,p}$ the (near) minimax rate of estimation relative to $\Theta_p$,
$$\liminf_{n\rightarrow\infty}\inf_{\th_0\in \Theta_{p}}P_{\th_0}(\diam(\hat C_n)\leq r_{n,p})=1.$$
This second property ensures that the good coverage is not achieved by taking conservative, 
overly large confidence sets, but that these sets have ``optimal'' diameter.
In our present situation we may choose the models $\Theta_p$ equal to nearly black bodies with $p$ nonzero
coordinates, in which case $r_{n,p}^2\asymp p\log (n/p)$, if $p\ll n$.
Now it is shown in \cite{Li} that confidence regions that are honest
over all parameters in $\tilde\Theta=\RR^n$ cannot be of square diameter smaller than $n^{1/2}$,
which can be (much) bigger than $p\log (n/p)$, if $p\ll n^{1/2}$. Similar restrictions are
valid for honesty over subsets of $\RR^n$, as follows from testing arguments (see the appendix in \cite{RobinsvdV}). 
Specifically, in \cite{Nickl2013} it is shown that confidence regions that adapt in size to nearly black bodies
of two different dimensions $p_{n,1}\ll p_{n,2}$ cannot be honest over the union of these two bodies, but only over the union
of the smallest body and the vectors in the bigger body that are at some distance from the smaller body.
As both the full Bayes and empirical Bayes horseshoe posteriors contract at the near square minimax rate $r_{n,p}$,
adaptively over every nearly black body, it follows that their credible balls cannot be honest in the full
parameter space. 

In Bayesian practice credible balls are nevertheless used as if they were confidence sets.
A main contribution of the present paper is to investigate for which parameters $\th_0$ this practice
is justified. We characterise the parameters for which the
credible sets of the horseshoe posterior distribution give good coverage, and the ones for which they do not. 
We investigate this both for the empirical and hierarchical Bayes approaches, both when $\t$ is
set deterministically, and in adaptive settings where the number of nonzero means is unknown. 
In the case of deterministically chosen $\t$, uncertainty quantification is essentially correct provided
$\t$ is chosen not smaller than $(p_n/n)\sqrt{\log{n/p_n}}$. For the more interesting full and
empirical Bayes approaches, the correctness depends on the sizes of the nonzero coordinates 
in $\th_0$. If a fraction of the nonzero coordinates is  detectable, meaning that they exceed the 
``threshold'' $\sqrt{2\log(n/p_n)}$, then uncertainty quantification by a credible ball
is correct up to a multiplicative factor in the radius. More generally, this is true if the
sum of squares of the non-detectable nonzero coordinates is suitably dominated,
as in \cite{BelitserNurushev}.


We show in this work that the uncertainty quantification given by the horseshoe posterior distribution
is ``honest'' only conditionally on certain prior assumptions on the parameters.
In contrast, interesting recent work within the context of the sparse linear regression model
is directed at obtaining confidence sets that are honest in the full parameter set \cite{ZhangZhang,Geeretal,LiuYu}. 
The resulting methodology, appropriately 
referred to as  ``de-sparsification'', might in our present very special case of the regression model
reduce to confidence sets for $\th_0$ based on the trivial pivot $Y^n-\th_0$, or functions thereof,
such as marginals. These confidence sets would have uniformly correct coverage, but
be very wide, and not employ the presumed sparsity of the parameter. This seems a high price to
pay;  sacrificing some coverage so as to retain some shrinkage may not be unreasonable.
Our contribution here is to investigate in what way the horseshoe prior makes this trade-off. In addition, we provide a specific example of an estimator that meets our conditions for adaptive coverage: the maximum marginal likelihood estimator (MMLE).  The MMLE is introduced in detail in  \cite{contractionpaper}. In this paper, we expand on the MMLE results in \cite{contractionpaper} by showing that it meets the imposed conditions for adaptive coverage as well.

Uncertainty quantification in the case of the sparse normal mean model was 
addressed also in the recent paper \cite{BelitserNurushev}. These authors consider a mixed Bayesian-frequentist
procedure, which leads to a mixture over sets $I\subset\{1,2,\ldots, n\}$ of projection estimators $(Y_i\1_{i\in I})$,
where the weights over $I$ have a Bayesian interpretation and each projection estimator comes with
a distribution. Treating this as a posterior distribution, the authors obtain credible balls for the parameter, which they show to be honest
over parameter vectors $\th_0$ that satisfy an ``excessive-bias  restriction''. This interesting procedure
has similar properties as the horseshoe posterior distribution studied in the present paper. While initially
we had derived our results under a stronger ``self-similarity'' condition, we present here the results
under a slight weakening of the ``excessive-bias restriction'' introduced in  \cite{BelitserNurushev}.

The performance of adaptive Bayesian methods for uncertainty quantification for the estimation of functions 
has been  previously considered in 
\cite{SzVZ2,szabo:vaart:zanten:2015c,serra:krivobokova:2014,castillo:nickl:2014,ray:2015,Sniekers1,Sniekers2,Sniekers3, belitser:2015, rousseau:sz:2016}.
These papers focus on adaptation to functions of varying regularity. This runs into similar problems
of honesty of credible sets, but the ordering by regularity sets the results
apart from the adaptation to sparsity in the present paper.

For single coordinates $\th_{0,i}$ uncertainty quantification by  marginal credible intervals is quite natural. 
Credible intervals can be easily visualised by plotting them versus the index (cf.\ Figure~\ref{fig:marginal}). 
They may also be used as a testing device, for instance by declaring coordinates $i$ for
which the credible interval does not contain 0 to be \emph{discoveries}.
We show that the validity of these intervals depends on the value of the true coordinate. 
On the positive side we show that marginal credible intervals for coordinates
$\th_{0,i}$ that are either close to zero or above the detection boundary are essentially correct. 
In particular, the fraction of false discoveries tends to zero. 
On the negative side the horseshoe posteriors
shrink intervals for intermediate values too much to zero for coverage.
Different from the case of credible balls, these conclusions are
hardly affected by whether the sparseness level $\t$ is set by an oracle or adaptively, based on the data.

The paper is organized as follows. The results for the marginal credible intervals are given in Section \ref{sec:intervals}, and the consequences for the false and true discoveries resulting from testing with the marginal credible intervals are explored in Section \ref{sec:fdr}. The results for the credible balls are given in Section \ref{sec:coverage}.
In all cases, the results are given for  deterministic and general empirical and hierarchical Bayes approaches. We illustrate the coverage properties of the marginal credible sets
computed by empirical and hierarchical Bayes methods, as well as the model selection properties in a simulation study in Section~\ref{sec:simulation}. 
We conclude with appendices containing all proofs not given in the main text.

\subsection{Notation} 
The posterior distribution of $\th$ relative to the
prior \eqref{EqHorseShoePrior} given fixed $\t$ is denoted by  $\Pi(\cdot\given Y^n,\t)$, and the posterior distribution
in the hierarchical setup where $\t$ has received a prior is denoted by $\Pi(\cdot\given Y^n)$.
We use $\Pi(\cdot\given Y^n,\hat\t)$ for the empirical Bayes ``plug-in posterior'' , which is  $\Pi(\cdot\given Y^n,\t)$  with a data-based variable $\hat\t$ substituted for $\t$. 
To emphasize that $\hat\t$ is not conditioned on, we alternatively use 
$\Pi_\t(\cdot\given Y^n)$ for $\Pi(\cdot\given Y^n,\t)$, and  $\Pi_{\hat \t}(\cdot\given Y^n)$ for $\Pi(\cdot\given Y^n,\hat\t)$.

The function $\phi$ denotes the density of the standard normal distribution. The class of nearly black vectors is given by
$\ell_0[p] = \{\th \in \mathbb{R}^n : \sum_{i=1}^n\1\{\th_{i} \neq 0\} \leq p\}$,
and we abbreviate
$$\z_\t = \sqrt{2\log(1/\t)}, \qquad \t_n(p) = (p/n)\sqrt{\log (n/p)}, \qquad \t_n= \t_n(p_n).$$

\section{\label{sec:intervals}Credible intervals}
We study the coverage properties of credible intervals for the individual coordinates $\th_{0,i}$.  We show that the marginal credible intervals fall into three categories, dependent on $\t$. We show that coordinates $\th_{0,i}$ that are either ``small" or ``large"  will be covered, in the sense that within both categories the fraction of correct intervals is arbitrarily close to 1. On the other hand, none of  the ``intermediate" coordinates $\th_{0,i}$ are covered. We show this first for the deterministic case, where the boundaries between the categories are at  multiples of $\tau$ and $\z_\t$ respectively. Furthermore, we show that the results for deterministic marginal credible intervals extend to the adaptive situation
for \emph{any} true parameter $\th_0$, with slight modification of the boundaries
between the three cases of small, intermediate and large coordinates. We elaborate on the implications for model selection in Section \ref{sec:fdr}.

\subsection{Definitions}
Non-adaptive  marginal credible intervals  can be  constructed from the \emph{marginal posterior distributions}
$\Pi(\th: \th_i\in \cdot\given Y^n,\t)$. By the independence of the pairs $(\th_i,Y_i)$ given $\t$, the
$i$th marginal  depends only on the $i$th observation $Y_i$.
We consider intervals of the form
\begin{equation}
\hat{C}_{ni}(L,\t)=\bigl\{\th_i: |\th_i-\hat\th_i(\t)|\leq L\hat r_i(\a,\t)\bigr\},
\label{EqMarginalCredibleInterval}
\end{equation}
where $\hat\th_{i}(\t)=\E(\th_i\given Y_i,\t)$ 
is the marginal posterior mean, $L$ a positive constant, and $\hat r_i(\a,\t)$ is determined so that,
for a given $0<\a\le 1/2$,
$$\Pi\bigl(\th_i:  |\th_i-\hat\th_i(\t)|\leq \hat r_i(\a,\t) \given Y_i,\t\bigr)=1-\alpha.$$

Adaptive empirical Bayes marginal credible intervals are defined by plugging in an estimator
$\widehat\t_n$ for $\t$ in the intervals $\hat{C}_{ni}(L,\t)$ defined by \eqref{EqMarginalCredibleInterval}.
Similarly full Bayes credible intervals $\hat C_{ni}(L)$ are defined from the full Bayes marginal
posterior distributions, centered around the posterior mean.

\subsection{Credible intervals for deterministic $\t$}

The coverage of the marginal credible intervals depends crucially on the value of the true coordinate $\th_{0,i}$.
For given $\t \ra 0$, positive constants $k_S$, $k_M$, $k_L$ and numbers $f_\t\uparrow\infty$ as $\t\ra0$, 
we distinguish three regions (small, medium and large) of signal parameters:
\begin{align*}
\cS&:=\bigl\{1\le i\le n: |\th_{0,i}|\leq k_S \t\bigr\},\\
\cM&:=\bigl\{1\le i\le n:  f_\t\t\leq |\th_{0,i}|\leq k_M \z_\t\bigr\},\\
\cL&:=\bigl\{1\le i\le n: k_L \z_\t\leq |\th_{0,i}|\bigr\}.
\end{align*}
The conditions on the constants and $f_\t$ in the following theorem make that these
three sets may not cover all coordinates $\th_{0,i}$, but their boundaries are almost contiguous.
The following theorem shows that the fractions of coordinates contained in $\cS$ and in $\cL$ that are covered by
the credible intervals are close to 1, whereas no coordinate in $\cM$ is covered.

Let $|\cdot|$ denote the cardinality of a set.

\begin{theorem}\label{thm:marginal}
Suppose that  $k_S>0$, $k_M<1$, $k_L>1$, and $f_\t\uparrow\infty$,  as $\t\ra0$. Then for $\t\ra 0$ and 
any sequence $\g_n\rightarrow c$ for some $0\le c\leq 1/2$, satisfying $\z_{\g_n}\ll \z_\t$,
\begin{align}
P_{\th_0} \Bigl( \frac1{|\cS|}{|\{i\in \cS: \th_{0,i}\in \hat{C}_{ni}(L_{S},\t)\}|}\geq 1-\g_n\Bigr)&\rightarrow 1,\label{eq: margcovS}\\
P_{\th_0} \bigl( \th_{0,i}\notin \hat{C}_{ni}(L,\t) \bigr)\rightarrow 1,\quad\text{for any } L>0 &\text{ and $i\in \cM$},\label{eq: margcovM}\\
P_{\th_0} \Bigl( \frac1{|\cL|}{|\{i\in \cL: \th_{0,i}\in \hat{C}_{ni}(L_L,\t)\}|}\geq 1-\g_n \Bigr)&\rightarrow 1,\label{eq: margcovL}
\end{align}
where $L_S=(2.1/z_\a)\bigl[k_S+(2/\g_n) \z_{\g_n/2}\bigr]$ and 
$L_L=(1.1/z_\a) \z_{\g_n/2}$.
\end{theorem}

\begin{proof}
See Section~\ref{sec:marginal}.
\end{proof}

\subsection{\label{sec:adaptive_intervals}Adaptive credible intervals}
We show that the adaptive credible intervals mimic the behaviour of the intervals for deterministic
$\t$ given in Theorem~\ref{thm:marginal_adapt}. The adaptive results require some conditions on either the empirical Bayes estimator of $\tau$, or the hyperprior on $\tau$. In the empirical Bayes case, one condition on the estimator of $\tau$ suffices, stated below. It is the same condition under which  adaptive contraction of the empirical Bayes horseshoe posterior was proven in \cite{contractionpaper}.

\begin{cond}\label{cond.eb}
There exists a constant $C > 0$ such that $\widehat\t_n\in[1/n, C\t_n(p_n)]$, with $P_{\th_0}$-probability tending to one,
uniformly in $\th_0\in\ell_0[p_n]$. 
\end{cond}

As proven in \cite{contractionpaper}, Condition \ref{cond.eb} is met by the marginal maximum likelihood estimator (MMLE). The MMLE is the maximum likelihood estimator of $\tau$ in the model where we assume that the data are distributed according to the convolution of the standard normal density and the horseshoe density on $\theta$. It is given by

\begin{align}\label{def: MMLE_HS}
\mmle =  \argmax_{\t\in \left[{1}/{n},1\right]} \prod_{i=1}^n\int_{-\infty}^{\infty} \phi(y_i-\th)g_{\t}(\th)\,d\th,
\end{align}
where $g_{\t}(\th)=\int_0^{\infty}\phi\left(\frac{\th}{\l\t}\right)\frac{1}{\l\t}\frac{2}{\pi(1+\l^2)}\, d\l$.

The restriction of the MMLE to the interval $[1/n, 1]$ corresponds to an assumption that the number of signals is between 1 and $n$, following the interpretation of $\tau$ as (approximately) the proportion of signals. In \cite{contractionpaper}, and in the simulation study in Section \ref{sec:simulation}, the MMLE is compared to the ``simple" estimator' of \cite{vdPas}, which estimates $p_n$ by counting the number of observations that are larger than  (a constant multiple of) the universal threshold $\sqrt{2\log{n}}$. and its computation is discussed. It is proven that the MMLE meets Condition \ref{cond.eb}, and thus that   the empirical Bayes procedure with the MMLE as a plug-in estimate of $\tau$ leads to adaptive posterior concentration results.

In the hierarchical Bayes procedure, we  impose the same conditions on the hyperprior $\pi_n$ as for adaptive posterior concentration in  \cite{contractionpaper}. We recall them below.

\begin{cond}\label{cond.hyper.3}
The prior density $\pi_n$ is supported inside $[1/n, 1]$. 
\end{cond}

\begin{cond}\label{cond.hyper.1}
Let $t_n= C_u\pi^{3/2}\, \t_n(p_n)$, with the constant $C_u$ as in 
Lemma~\ref{lem: E_nonzero_m(Y)}(i). The prior density $\pi_n$ satisfies
\begin{equation*}
\int_{t_n/2}^{t_n}\pi_n(\t)\, d\t \gtrsim e^{-cp_n},\quad \text{for some $c<C_u/2$},
\end{equation*}
\end{cond}
where $t_n= C_u\pi^{3/2}\, \t_n(p_n)$. 

Condition \ref{cond.hyper.1} may be replaced by the weaker Condition \ref{cond.hyper.1a}, at the price of suboptimal rates.

\begin{cond}\label{cond.hyper.1a}
For $t_n$ as in Condition~\ref{cond.hyper.1} the prior density $\pi_n$ satisfies,
\begin{equation*}
\int_{t_n/2}^{t_n}\pi_n(\t)\, d\t \gtrsim t_n.
\end{equation*}
\end{cond}

Examples of priors meeting Conditions \ref{cond.hyper.3} and \ref{cond.hyper.1a} are the Cauchy prior on the positive reals, or the uniform prior, both truncated to $[1/n, 1]$. They satisfy the stronger Condition \ref{cond.hyper.1} if $p_n \geq C\log{n}$, for a sufficiently large $C > 0.$

In the adaptive case, the three regions (small, medium and large) of signal parameters are defined as, for  given positive constants  $k_S$, $k_M$, $k_L$, and $f_n$:
\begin{align*}
\cS_{a}&:=\bigl\{1\le i\le n: |\th_{0,i}|\leq k_S/n\bigr\},\\
\cM_{a}&:=\bigl\{1\le i\le n: f_n\t_n(p_n)\leq |\th_{0,i}|\leq k_M \sqrt{2\log(1/\t_n(p_n))}\bigr\},\\
\cL_{a}&:=\bigl\{1\le i\le n: k_L \sqrt{2\log{n}}\leq |\th_{0,i}|\bigr\}.
\end{align*}

\begin{theorem}\label{thm:marginal_adapt}
Suppose that $k_S>0$, $k_M<1$, $k_L>1$, and $f_n\uparrow\infty$.
If $\hat\t_n$ satisfies  Condition~\ref{cond.eb}, then 
for any sequence $\g_n\rightarrow c$ for some $0\le c\le 1/2$ such that $\z_{\g_n}^2\ll \log (1/\t_n(p_n))$, we have that 
\begin{align}
P_{\th_0} \Bigl( \frac1{|\cS_a|}|\{i\in \cS_a: \th_{0,i}\in \hat{C}_{ni}(L_S, \hat\t_n)\}|\geq 1-\g_n \Bigr)&\rightarrow 1,\label{eq: margcovS_a}\\
P_{\th_0} \bigl( \th_{0,i}\notin \hat{C}_{ni}(L, \hat\t_n) )\rightarrow 1,\quad\text{for any }L>0 &\text{ and $i\in \cM_a$},\label{eq: margcovM_a}\\
P_{\th_0} \Bigl( \frac1{|\cL_a|}|\{i\in \cL_a: \th_{0,i}\in \hat{C}_{ni}(L_L,\hat\t_n)\}|/l\geq 1-\g_n \Bigr)&\rightarrow 1,\label{eq: margcovL_a}
\end{align}
with $L_S$ and $L_L$ given in Theorem~\ref{thm:marginal}.
Under Conditions~\ref{cond.hyper.3} and \ref{cond.hyper.1} and in addition $p_n\gtrsim \log n$ the same statements hold
for the hierarchical Bayes marginal credible sets. This is also true 
under Conditions~\ref{cond.hyper.3} and \ref{cond.hyper.1a} if $f_n\gg \log n$, with different constants  $L_S$ and $L_L$.
\end{theorem}

\begin{proof}
See Section \ref{sec:marginal_adapt}.
\end{proof}

\begin{remark}
Under the self-similarity assumption \eqref{def: selfsim} discussed in Section \ref{sec:coverage.bayes}, the statements of Theorem~\ref{thm:marginal_adapt} 
hold for the sets $\cS$, $\cM$ and $\cL$ given preceding Theorem~\ref{thm:marginal} with $\t=\t_n(p_n)$.
\end{remark}

\begin{figure}[h!]
\begin{center}
\includegraphics[width = \textwidth]{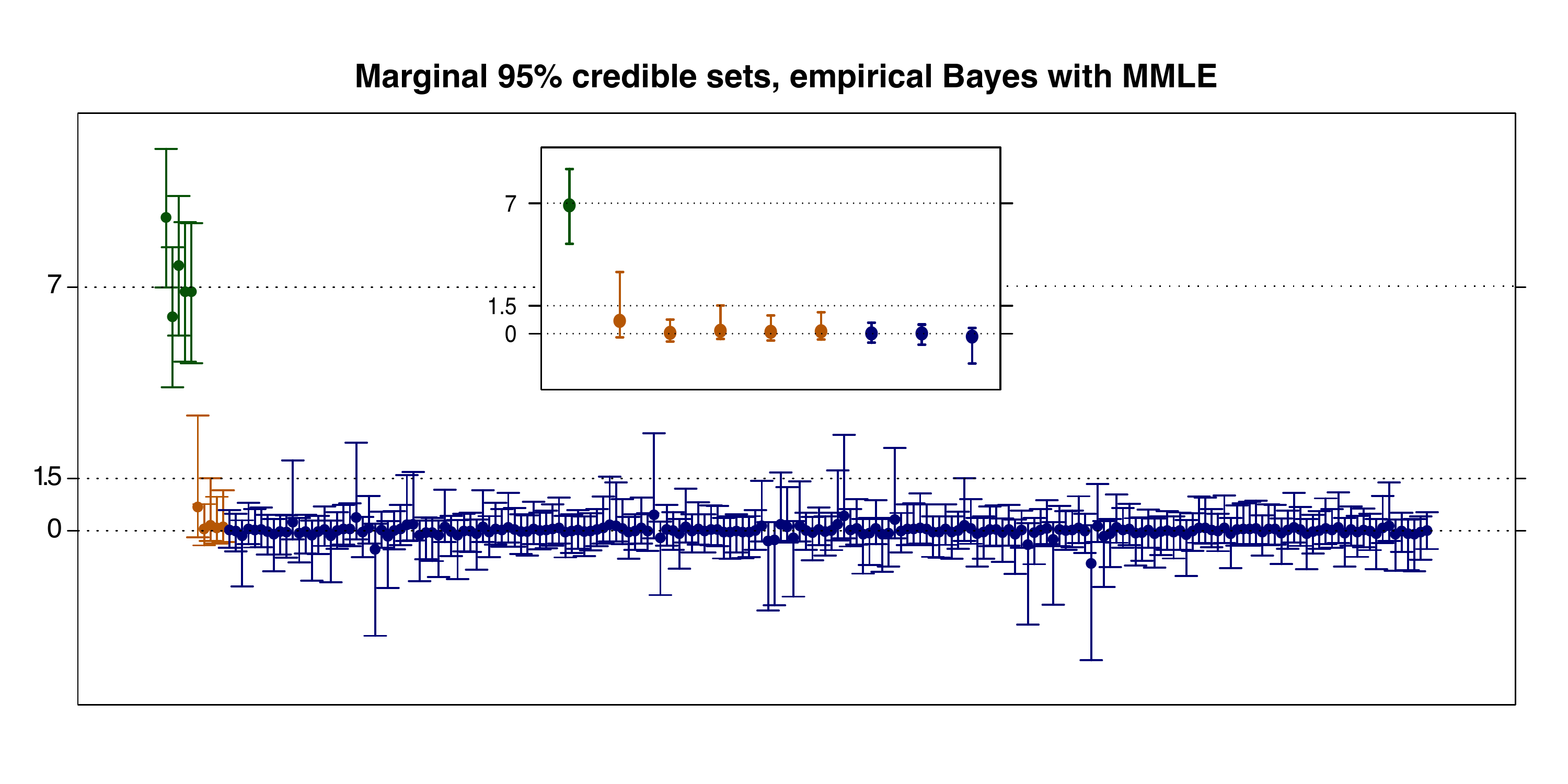}
\caption{95\% marginal credible intervals based on the MMLE empirical Bayes method,
for a single observation $Y^n$ of length $n = 200$ with $p_n = 10$ nonzero parameters,
the first 5  (from the left) being 7 (green), the next 5 equal to 1.5 (orange);  the remaining 190  parameters are coded (blue). 
The inserted plot zooms in on credible intervals 5 to 13, thus showing one large mean and all intermediate means.}
\label{fig:marginal}	
\end{center}
\end{figure}

Figure~\ref{fig:marginal} illustrates Theorem~\ref{thm:marginal_adapt}
by showing the marginal credible sets for just a single draw of the observation,
in a setting with  $n = 200$, and $p_n = 10$ nonzero coordinates.
The value $\t$ was chosen equal to the MMLE, which realised as approximately 0.11.
The means were taken equal to 7, 1.5 or 0, corresponding to the three regions $\cL, \cM,\cS$ listed in the theorem
($\sqrt{2\log{n}} \approx 3.3$).  All the large means (equal to 7)
were covered; only 2 out of 5 of the medium means (equal to 1.5) were covered; and all small (zero) means
were covered, in agreement with Theorem \ref{thm:marginal_adapt}.
It may be noted that intervals for zero coordinates are not necessarily narrow.

\section{\label{sec:fdr}Model selection}
The marginal credible sets give rise to a natural model selection procedure: a parameter is selected
as a signal if and only if the corresponding credible interval does not contain zero. We study this
procedure again both in the case that a value of $\tau$ is available and in the adaptive case where
$\tau$ is estimated from the data or receives a hyperprior.

In light of the results of Theorems~\ref{thm:marginal} and~\ref{thm:marginal_adapt}, and the fact
that the number of nonzero parameters is a vanishing fraction of the total set of coordinates by
assumption, we consider three quantities to describe the accuracy of a model selection
procedure. The first is the fraction of parameters exactly equal to zero that is falsely considered
a signal. The second is the fraction of small and medium signals in $\cS$ and $\cM$, or $\cS_a$ and
$\cM_a$, that is correctly considered a signal. The third is the fraction of large signals in $\cL$
or $\cL_a$ that is correctly selected as signals. A quick summary of the results is that, with probability tending
to one, only vanishing fractions of zeroes and large signals are incorrectly selected or not
selected, while the fraction of small and medium signals that are not discovered tends to one.

We state the result for the adaptive case, with empirical Bayes with the MMLE or hierarchical Bayes. 
A similar assertion for non-adaptive case is stated in Appendix \ref{sec:proofs_model_selection} as Theorem \ref{thm: modelselect:credible}.

\begin{theorem}\label{thm: modelselect:adapt} Suppose that $k_S>0$, $k_M<1$, $k_L>1$, and $f_n\uparrow\infty$.
For any sequence $\g_n\rightarrow c$ for some $0\le c\le 1/2$ such that $\z_{\g_n}^2\ll \log (1/\t_n(p_n))$, the following statements hold.

\begin{enumerate}
\item[(i)] The false discovery rate of the MMLE empirical Bayes and hierarchical Bayes credible intervals based model selection procedure is bounded from above by $\gamma_n$.

\item[(ii)] With probability tending to one, at least a $1-\gamma_n$ fraction of the signals  belonging to the set $\cL_a$ will be covered, i.e. 
\begin{align*}
P_{\th_0} \Bigl( \frac1{|\cL_a|}{|\{i\in \cL_a: 0\notin \hat{C}_{ni}(C,\hat\t)\}|}\geq 1-\g_n \Bigr)&\rightarrow 1,
\end{align*}
for any $C>0$ both for the hierarchical and empirical Bayes method.

\item[(iii)] At most a $\gamma_n$ fraction of the nonzero parameters $\theta_{0,i}\in \cS_a\cup \cM_a$ 
will be selected by the credible set method (with any blow up factor $C\geq 1$), with probability tending to one.
\end{enumerate}

\end{theorem}

\begin{proof} See Appendix \ref{sec:proofs_model_selection}. \end{proof}

\begin{remark}
The model selection theorem above is valid for arbitrary blow up constant $C\geq 1$, in contrast
  to the marginal coverage Theorem \ref{thm:marginal_adapt}, where the blow up factors must be
  chosen large enough. This is because selection is an easier problem. Marginal coverage
  requires that the size of the credible set compares appropriately to the distance between the true parameter and the
  marginal posterior mean. This comparison could be done only up to constant multipliers. Model selection
  depends only on whether zero is inside the marginal credible set, and this 
  requires only a lower bound for the marginal posterior probability that the signal is negative (or
  positive) if the marginal posterior mean is positive (or negative).
\end{remark}

An alternative method for model selection using the horseshoe was proposed by \cite{Carvalho2010}.  
They proposed to select as nonzero coordinates the indices such that the ratio $\kappa_i(\tau)=\hat\theta_i(\tau)/Y_i$ 
exceeds a threshold (to be precise $\kappa_i(\tau)>1/2$). This method has similar behaviour to the
credible set based model selection approach, as proven in Theorem \ref{thm: modelselect:threshold}
in Appendix \ref{sec:proofs_model_selection}. We refer to \cite{Datta2013} for theoretical properties of
this procedure, and compare the credible interval and thresholding methods further through
simulation in Section \ref{sec:simulation}.

\section{Credible balls}
\label{sec:coverage}
By their definition, credible sets contain a fixed fraction, e.g.\ 95 \%, of the posterior mass.
The diameter of  such sets will be at most of the order of the posterior contraction rate. 
The upper bounds on the contraction rates of the horseshoe posterior distributions
given in \cite{contractionpaper} imply that the horseshoe credible sets are narrow enough to be informative. However, these bounds
do not guarantee that the credible sets will \emph{cover} the truth. The latter is dependent on the \emph{spread}
of the posterior mass relative to its distance to the true parameter. For instance, the bulk of the 
posterior mass may be highly concentrated inside a ball of radius the contraction rate, but within a narrow area 
of diameter much smaller than its distance to the true parameter.

In this section we study coverage of credible balls, that is, credible sets for the full parameter vector $\th_0\in\RR^n$ relative to the Euclidean distance. 
We do so first in the case of deterministic $\t$ and next for the empirical and full Bayes posterior
distributions.


\subsection{Definitions}

Given a deterministic hyperparameter $\t$, possibly depending on $n$ and $p_n$,
we consider  a \emph{credible ball} of the form
\begin{align}
\hat{C}_n(L,\t)=\bigl\{\th: \|\th-\hat\th(\t)\|_2\leq L \hat r(\a,\t)\bigr\},\label{def:credible}
\end{align}
where $\hat\th(\t)=\E(\th\given Y^n ,\t)$ is the posterior mean, $L$ a positive constant, and 
for a given $\a\in (0,1)$ the number $\hat r(\a,\t)$ is determined such that
$$\Pi\big(\th: \|\th-\hat\th(\t)\|_2\le \hat r(\a,\t) \given Y^n,\t\big)=1-\a.$$
Thus $\hat r(\a,\t)$ is the natural radius of a set of ``Bayesian credible level'' $1-\a$, and
$L$ is a constant, introduced to make up for a difference between credible and confidence
levels, similarly as in \cite{SzVZ2}. Unlike in the latter paper the  radii $\hat r(\a,\t)$ do depend on the observation $Y^n$,
as indicated by the hat in the notation.

In the empirical Bayes approach we define a credible  set by plugging in an estimator
$\widehat\t_n$ of $\t$ into the non-adaptive credible ball  $\hat{C}_n(L,\t)$ given in \eqref{def:credible}:
\begin{align}
\hat{C}_n(L,\widehat\t_n)=\bigl\{\th: \|\th-\hat\th({\widehat\t_n})\|_2\leq L \hat r(\a,\widehat\t_n)\bigr\}.\label{def:ebcredible}
\end{align}
In the hierarchical Bayes case we use a ball around the full posterior mean $\hat\th=\int\th\,\Pi(d\th\given Y^n)$, given by 
\begin{align}
\hat{C}_n(L)=\bigl\{\th: \|\th-\hat\th\|_2\leq L \hat r(\a)\bigr\},\label{def:hbcredible}
\end{align}
where $L$ is a positive constant and $\hat r(\a)$ is defined from the full posterior distribution by 
$$\Pi\bigl(\th: \|\th-\hat\th\|_2\leq \hat r(\a) \given Y^n\bigr)=1-\a.$$
The question is whether these Bayesian credible sets are appropriate for uncertainty
quantification from a frequentist point of view.

\subsection{Credible balls for deterministic $\t$}
\label{sec:deterministic}
The following lower bound for $\hat r(\a,\t)$ in the case that $n\t\ra\infty$ is the key to the frequentist coverage.
The assumption $n\t/\z_\t \ra \infty$ is satisfied for $\t$ of the order the ``optimal'' rate $\t_n(p_n)$
provided $p_n\ra\infty$ (as we assume). 

\begin{lemma}\label{lem: LBradius}
If  $n\t/\z_\t\rightarrow\infty$, then with $P_{\th_0}$-probability tending to one,
$$\hat r(\a,\t)\ge 0.5\sqrt{n\t\z_\t}.$$
\end{lemma}

\begin{proof}
See Section \ref{Sec: lem: LBradius}.
\end{proof}

\begin{theorem}\label{thm: coverage_nonadapt}
If  $\t\geq \t_n$ and $\t\ra0$ and $p_n\ra\infty$ with $p_n=o(n)$, then, there exists a large enough $L>0$ such that
$$\liminf_{n\ra\infty}\inf_{\th_0\in \ell_0[p_n]}P_{\th_0}\big(\th_0\in \hat{C}_n(L, \t)\big)\geq 1-\a.$$
\end{theorem}

\begin{proof}
The probability of the complement of the event in the display 
is equal to $P_{\th_0}\bigl(\|\th_0-\hat\th(\t)\|_2>L\, \hat r(\a,\t)\bigr)$.
In view of Lemma~\ref{lem: LBradius}  this is bounded by $o(1)$ plus 
$$P_{\th_0}\bigl(\|\th_0-\hat\th(\t)\|_2>0.5L\sqrt{n\t\z_\t}\bigr)
\lesssim \frac{\E_{\th_0}\|\hat\th(\t)-\th_0\|_2^2}{L^2n\t\z_\t}.$$
By Theorem 3.2 of \cite{vdPas} 
the numerator on the right is bounded by a multiple
of $p_n\log (1/\t)+n\t\sqrt{\log 1/\t}$. By the assumption $\t\ge \t_n\ge 1/n$ the
quotient is smaller than $\a$ for appropriately large choice of $L$.
\end{proof}

Theorem \ref{thm: coverage_nonadapt} combined with the upper bound on the posterior contraction rate  in \cite{vdPas} show that a (slightly enlarged)  credible ball centered at the posterior mean is of rate-adaptive size and covers the truth provided $\t$ is chosen of the order of the ``optimal'' value $\t_n(p_n)$. This is not possible in general, as it requires knowing the number of signals. In the next sections, we will show that if empirical Bayes estimators are ``close" to $\t_n(p_n)$, or if a hyperprior on $\t$ places ``enough'' mass on a neighborhood of a quantity of order $\t_n(p_n)$, then adaptation to the unknown number of signals is possible. 

\subsection{Adaptive credible balls}
\label{sec:coverage.bayes} 
We now turn to credible sets in the more realistic scenario that the sparsity parameter $p_n$ is not available. We investigate both the empirical Bayes and the hierarchical Bayes credible balls. We show that  both empirical and hierarchical credible balls cover the true parameter $\theta_0$, if $\th_0$ satisfies the ``excessive-bias restriction'', given below, under some conditions on the empirical Bayes plug-in estimate or the hierarchical Bayes hyperprior on $\t$.

\subsubsection{The excessive-bias restriction}

Unfortunately, coverage can be guaranteed only for a selection of true parameters $\th_0$.  
The problem is that a data-based estimate of sparsity may lead to \emph{over-shrinkage}, due
to a too small value of the plug-in estimator or concentration of the posterior distribution of $\t$ too close to zero.
Such over-shrinkage makes the credible sets too small and close to zero. A simple condition preventing 
over-shrinkage is that a sufficient number of nonzero parameters $\th_{0,i}$ is above
the ``detection boundary''. It turns out  that the correct threshold for detection is
given by $\sqrt{2\log (n/p_n)}$. This leads to the following condition.

\begin{assumption}[self-similarity] 
A vector $\th_0\in\ell_0[p]$ is called \emph{self-similar} if 
\begin{align}
\#\bigl(i: |\th_{0,i}|\geq A \sqrt{2\log (n/p)}\bigr)\ge \frac p{C_s}.\label{def: selfsim}
\end{align}
The two constants $C_s$ and $A$ will be fixed to universal values, where necessarily $C_s\ge1$
and it is required that $A>1$.
\end{assumption}

The problem of over-shrinkage is comparable to the problem of over-smoothing in the
context of nonparametric density estimation or regression, due to the choice of a too
large bandwidth or smoothness level. The preceding self-similarity condition plays the same
role as the assumptions of ``self-similarity'' or ``polished tail'' used by 
\cite{PicTri,GineNickl,Bull,nickl:szabo:2014, SzVZ2,Sniekers2,rousseau:sz:2016} in their investigations of confidence sets
in nonparametric density estimation and regression, or the ``excessive-bias'' restriction 
in \cite{belitser:2015} employed in the context of Besov-regularity classes in the normal mean model. 

The self-similarity condition is also reminiscent of the \emph{beta-min condition} for the adaptive Lasso
\cite{vdGeer2011, Buhlman2011}, which imposes a lower bound on the nonzero signals in order to
achieve consistent selection of the set of nonzero coordinates of $\th_0$. However, the present condition is
different in spirit both by the size of the cut-off and by requiring only that a fraction of the nonzero means is above
the threshold.

For ensuring coverage of credible balls the condition can be weakened to 
the following more technical condition.

\begin{assumption}[excessive-bias restriction] 
\label{AssumptionEB}
A vector $\th_0\in\ell_0[p]$ satisfies the \emph{excessive-bias restriction} 
 for constants $A>1$ and  $C_s,C>0$, if there exists an integer $q\ge1$ with
\begin{align}\label{condition: EB}
\sum_{i: |\th_{0,i}|< A\sqrt{2\log (n/q)}}\th_{0,i}^2\le C q\log (n/q),
\qquad \# \bigl(i: |\th_{0,i}|\ge A\sqrt{2\log (n/q)}\bigr)\ge \frac{q}{C_s}.
\end{align}
The set of all such vectors $\th_0$ (for fixed constants $A, C_s,C$) is denoted by $\Theta[p]$,
and $\tilde p=\tilde p (\th_0)$ denotes $\# \bigl(i: |\th_{0,i}|\ge A\sqrt{2\log (n/q)}\bigr)$, for the
smallest possible $q$.
\end{assumption}

If $\th_0\in \ell_0[p]$ is self-similar, then it satisfies the excessive-bias restriction with $q=p$,
$C=2A^2$ and the same constants $A$ and $C_s$. This follows, because the sum 
in \eqref{condition: EB} is trivially bounded by $\#(i: \th_{0,i}\not=0)\, A^2  2\log (n/q)$. 

In the following example we show that the excessive-bias restriction
is also implied by a condition with the same name introduced in \cite{BelitserNurushev}.
The latter condition motivated Assumption~\ref{AssumptionEB}, which is more
suited to our investigation of the horseshoe credible sets.

\begin{ex}
For a given $\th_0$ and any subset $I\subset\{1,2,\ldots, n\}$ let 
$$G(I)=\sum_{i\in I^c}\th_{0,i}^2+2A^2 |I| \log \frac{ne}{|I|}.$$
In \cite{BelitserNurushev} $\th_0$ is defined to satisfy the \emph{excessive-bias restriction}
if $G$ takes its minimum at a nonempty set $\tilde I$ such that 
$G(\tilde I)\le C |\tilde I| \log ({ne}/{|\tilde I|})$.

We now show that in this case $\th_0$ also satisfies Assumption~\ref{AssumptionEB}, with $q=|\tilde I|$.
Let $\th_{0,i}$ be a coordinate with $i\in \tilde I$ of minimal 
absolute value $|\th_{0,i}|=\min\{|\th_{0,j}|: j\in \tilde I\}$. From $G(\tilde I)\le G(\tilde I-\{i\})$
we obtain that $\th_{0,i}^2\ge 2A^2 |\tilde I|\log (ne/|\tilde I|)-2A^2 (|\tilde I|-1)\log (ne/(|\tilde I|-1))\ge 2A^2\log(n/|\tilde I|)$,
since the derivative of $x\mapsto x\log (ne/x)$ is $\log (n/x)$.
Consequently, first $\#(j: \th_{0,j}^2\ge 2A^2\log (n/|\tilde I|))\ge \#(j: \th_{0,j}^2\ge \th_{0,i}^2)\ge|\tilde I|$, 
by the minimising property of $\th_{0,i}$, verifying the second inequality in \eqref{condition: EB}. Second 
$\{j: \th_{0,j}^2<2A^2 \log (n/q)\}\subset \{j: \th_{0,j}^2<\th_{0,i}^2\}\subset \tilde I^c$,
again by the minimising property of $\th_{0,i}$. Thus the first inequality of \eqref{condition: EB}
follows by the fact that $G(\tilde I)\le C |\tilde I| \log ({ne}/{|\tilde I|})$.
\end{ex}

\subsubsection{Empirical Bayes condition and the MMLE }
To obtain coverage in the empirical Bayes setting, we replace Condition~\ref{cond.eb} by the following.

\begin{cond}\label{cond.adapt.coverage.eb}
The estimator $\widehat\t_n$ satisfies,  for a given sequence $p_n$ and some constant $C>1$,
with $\tilde p=\tilde p(\th_0)$,
\begin{equation*}
\inf_{\th_0 \in \Theta[p_n]} P_{\th_0}\bigl(C^{-1}\t_n(\tilde{p})\leq\widehat\t_n
	\leq C\t_n(\tilde{p})\bigr)\rightarrow 1.
\end{equation*}
\end{cond}

The lower bound of order $\t_n(\tilde p)$ instead of $1/n$ prevents over-shrinkage. Although this condition may appear more restrictive than Condition~\ref{cond.eb}, 
Condition~\ref{cond.adapt.coverage.eb} may not be more stringent than Condition~\ref{cond.eb}, because it only
needs to hold for vectors $\th_0$ that meet the excessive-bias restriction.

For the coverage results in this paper, we need the additional result that the MMLE is of the order $\t_n(\tilde p(\theta_0)$ for all vectors $\theta_0$ satisfying the excessive-bias restriction.

\begin{lemma}\label{thm: LB_tau}
For $p_n\ra\infty$ such that $p_n=o(n)$, the MMLE $\widehat\t_n$ satisfies Condition~\ref{cond.adapt.coverage.eb}.
\end{lemma}

\begin{proof}
See Section \ref{sec: thm: LB_tau}.
\end{proof}

The relative performances of the empirical Bayes procedures with the MMLE or the ``simple" estimator are studied further in Section \ref{sec:simulation}.

\subsubsection{Main result on adaptive credible balls}
Under the excessive-bias restriction, both the empirical and hierarchical Bayes credible balls are honest and adaptive. In the hierarchical Bayes setting, the hyperprior is assumed to be supported on $[1/n, 1]$, similar to the MMLE. 

\begin{theorem}\label{thm:coverage}
Let $\tilde p_n\le p_n$ be given sequences with $\tilde p_n\ra\infty$ and $p_n=o(n)$.
If the estimator $\widehat\t_n$ of $\t$ satisfies Condition~\ref{cond.adapt.coverage.eb}, then
for a sufficiently large constant $L$ the empirical Bayes credible ball $\hat{C}_n(L,\hat\t_n)$ 
has honest coverage and rate adaptive (oracle) size:
\begin{align*}
\liminf_{n\ra\infty}\inf_{\th_0\in\Theta[p_n], \tilde p(\th_0)\ge \tilde p_n} P_{\th_0}\Big(\th_0\in \hat{C}_n(L,\hat\t_n) \Big)\geq 1-\a,\\
\inf_{\th_0\in \Theta[p_n]}P_{\th_0}\Bigl(\diam\bigl(\hat{C}_n(L,\widehat\t_n) \bigr)
\lesssim \sqrt{\tilde{p}\log(n/\tilde p)}\Bigr)\rightarrow 1.
\end{align*}
In particular, these assertions are true for the MMLE.  Furthermore,
if $\tilde{p}_n\geq C\log n$ for a sufficiently large constant $C$, then the hierarchical Bayes method
with $\t\sim\pi_n$ for $\pi_n$ probability densities on $[1/n,1]$ that are bounded away from zero
also yields adaptive and honest confidence sets: for sufficiently large $L$,
\begin{align*}
\liminf_{n\ra\infty}\inf_{\th_0\in\Theta[p_n], \tilde p(\th_0)\ge \tilde p_n} P_{\th_0}\Big(\th_0\in \hat{C}_n(L) \Big)\geq 1-\a,\\
\inf_{\th_0\in\Theta[p_n], \tilde p(\th_0)\ge \tilde p_n}
P_{\th_0}\Big(\diam\bigl(\hat{C}_n(L)\bigr)\lesssim \sqrt{\tilde{p}\log (n/\tilde p)}\Big)\rightarrow 1.
\end{align*}
\end{theorem}

\begin{proof}
See Section \ref{Sec:coverage_adaptive}.
\end{proof}

It may be noted that for self-similar $\th_0$ the square diameter of the credible balls
is of the order $p\log (n/p)$, improving on the square contraction rate $p\log n$ obtained in 
 \cite{contractionpaper}. For parameters satisfying the excessive-bias restriction,
this may further improve to $\tilde p\log (n/\tilde p)$.

\section{\label{sec:simulation}
Simulation study}
In the first simulation study in Section \ref{sec:sim_coverage}, we compare four versions of the horseshoe (empirical Bayes with two different estimators and hierarchical Bayes with two different priors) and evaluate the coverage properties and interval lengths of the resulting credible intervals. In addition, we include an approximation to the credible intervals based on the normal distribution.

In the  simulation study in Section \ref{sec:sim_selection}, we compare the model selection properties of the method based on credible intervals resulting from the horseshoe with the MMLE, as discussed in Section \ref{sec:fdr}, to the thresholding method introduced by \cite{Carvalho2010}, with the MMLE of $\tau$ plugged in. We use the MMLE because the best results are obtained for the horseshoe with MMLE in the first simulation in Section \ref{sec:sim_coverage}.
All simulations were carried out using the R package `horseshoe' \cite{horseshoepackage}.

\subsection{Coverage, interval length, and $\tau$}\label{sec:sim_coverage}
Several MCMC samplers and software packages are available for 
computation of the posterior distribution \cite{Scott2010-2, Makalic2015,  Gramacy2014, horseshoepackage, fasthorseshoe}.

We study the relative performances of the empirical Bayes and hierarchical Bayes approaches further through simulation studies, 
extending the simulation study in \cite{vdPas}.
We consider  empirical Bayes combined with either (i) the simple estimator (with $c_1 = 2, c_2 = 1$)
or (ii) the MMLE, and for hierarchical Bayes with either (iii) a Cauchy prior on $\t$, or (iv) a Cauchy prior truncated to $[1/n, 1]$ on $\t$. We  study the coverage and average lengths of the marginal credible intervals resulting from these four methods, as well as intervals based solely on the posterior mean and variance. In addition, we study intervals of the form $\hat\th_i(y_i, \mmle) \pm 1.96\sqrt{\var(\th_i \given y_i, \mmle)}$,  based on a normal approximation to the posterior, where $\hat\th_i(y_i, \mmle)$ is the posterior mean and $\var(\th_i \given y_i, \mmle)$ refers to the posterior variance, both with the MMLE plugged in. We include the approximation because it offers a computational advantage over the other methods, as no MCMC is required. 

We  consider a mean vector of length $n = 400$, with $p_n \in \{20, 200\}$. We draw the nonzero means from a $\mathcal{N}(A, 1)$-distribution, with $A = c\sqrt{2\log{n}}$ for $c \in \{1/2, 1, 2\}$, corresponding to most nonzero means being below the universal threshold, close to the universal threshold, or well past the universal threshold, respectively. In each of the $N = 500$ iterations, we created the 95\% marginal credible sets for the hierarchical and empirical Bayes methods by taking the 2.5\%- and 97.5\%-quantiles of the MCMC samples as the endpoints. We did not include a blow-up factor.

Figure \ref{fig:coverage} gives the coverage results averaged over the 500 iterations, for all parameters, and separately for the $p_n$ nonzero means and the $(n-p_n)$ zero means. The average lengths of the credible sets, again for all signals and separately for the nonzero and zero means, are displayed in Figure \ref{fig:length}. Figure \ref{fig:tau} gives the mean value of $\t$ - in the hierarchical Bayes settings, the posterior mean of $\t$ was recorded for each iteration. No value is given for the normal approximation, as it uses the MMLE as a plug-in value for $\t$. 

\begin{figure}[h!]
\begin{center}
\includegraphics[width = 0.9\textwidth]{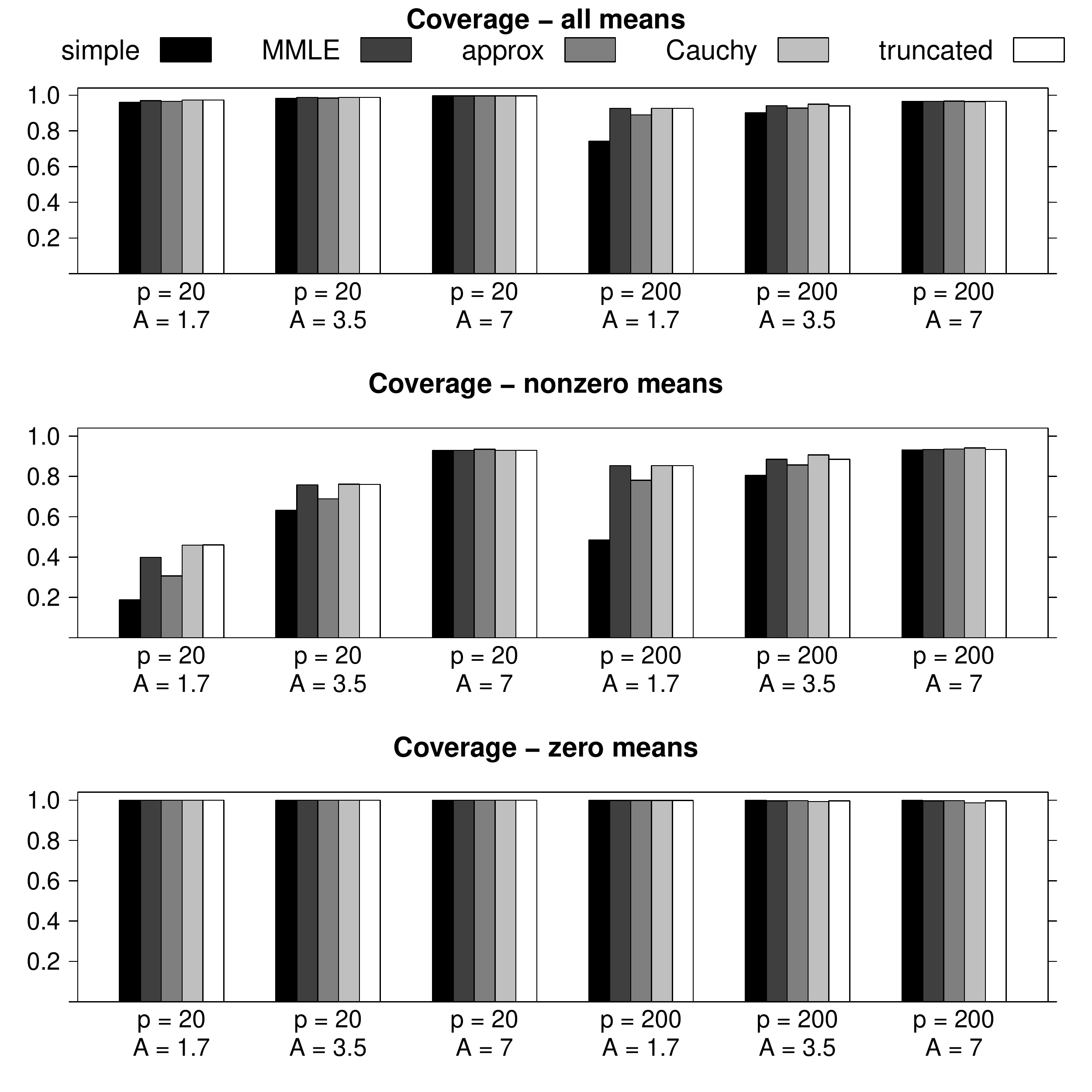} 
\caption{Average coverage of all parameters (top), the nonzero means (middle) and the zero means
  (bottom) for the five methods, from left to right: empirical Bayes with simple estimator
  ($c_1 = 2, c_2 = 1$) and MMLE, normal approximation, hierarchical Bayes with Cauchy prior on $\t$
  and with Cauchy prior truncated to $[1/n, 1]$. The $p_n$ nonzero means were drawn from a
  $\mathcal{N}(A, 1)$ distribution. Results are based on averaging over 500 iterations.}
\label{fig:coverage}
\end{center}
\end{figure}

\begin{figure}[h!]
\begin{center}
\includegraphics[width = 0.9\textwidth]{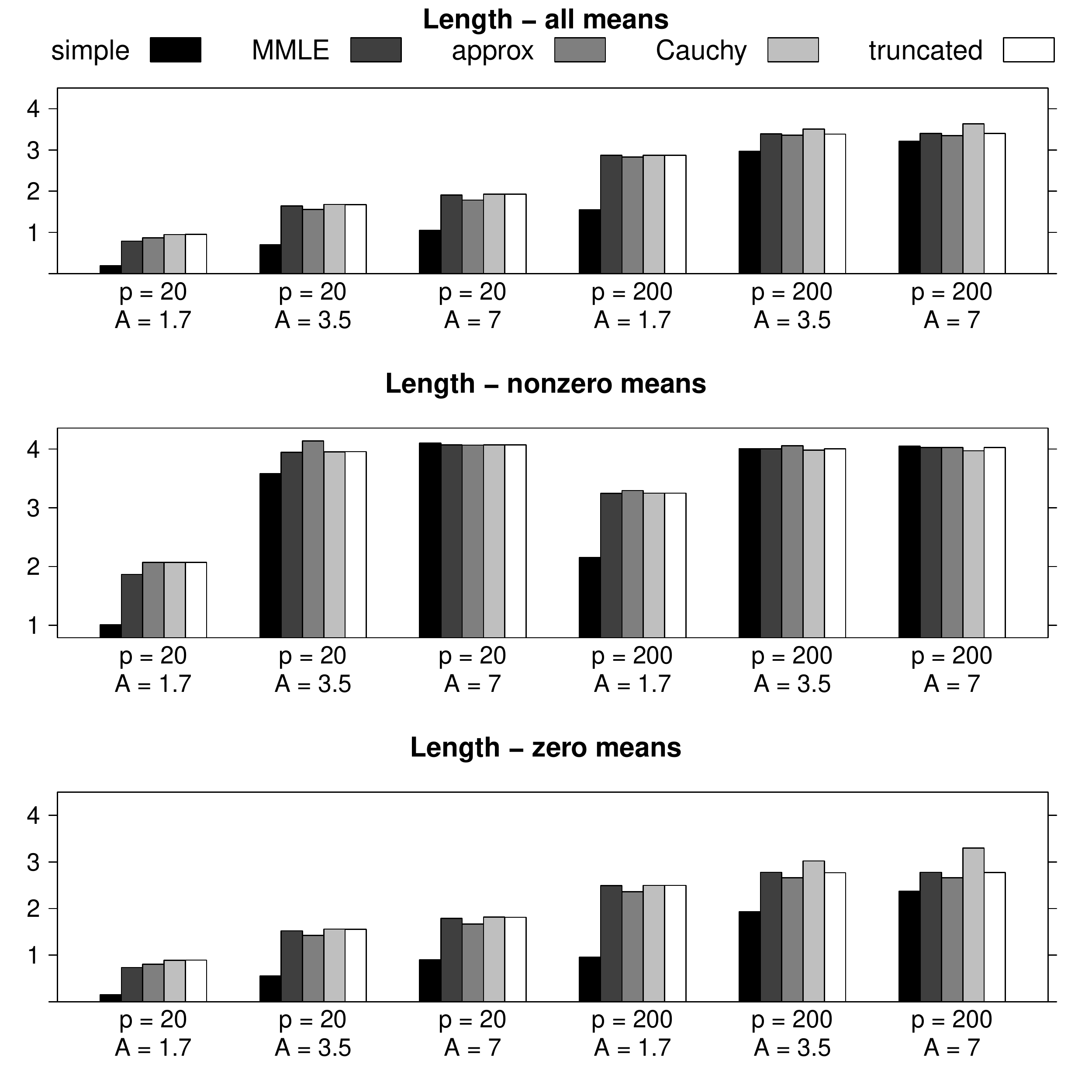} 
\caption{Average length of the credible sets of all parameters (top), the nonzero means (middle) and
  the zero means (bottom) for the five methods, from left to right: empirical Bayes with simple
  estimator ($c_1 = 2, c_2 = 1$) and MMLE, normal approximation, hierarchical Bayes with Cauchy
  prior on $\t$ and with Cauchy prior truncated to $[1/n, 1]$. The $p_n$ nonzero means were drawn
  from a $\mathcal{N}(A, 1)$ distribution. Results are based on averaging over 500 iterations. }
\label{fig:length}
\end{center}
\end{figure}

\begin{figure}[h!]
\begin{center}
\includegraphics[width = 0.9\textwidth]{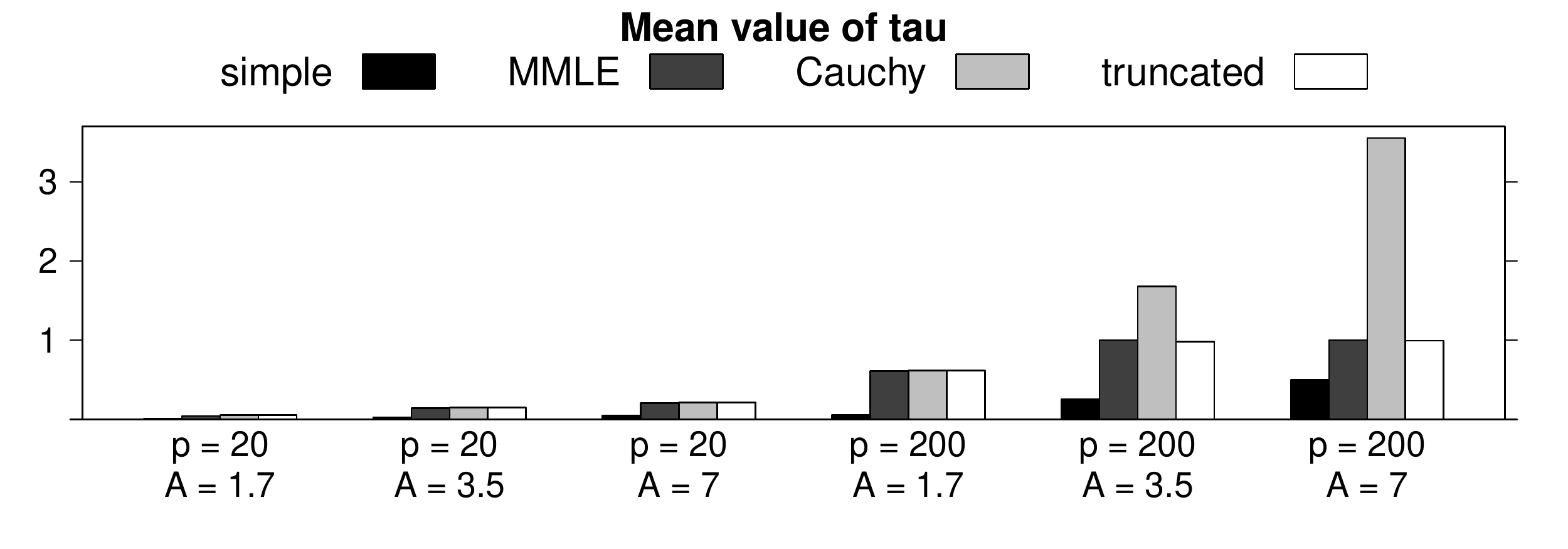} 
\caption{Average value of $\t$ for four methods, from left to right: empirical Bayes with simple
  estimator ($c_1 = 2, c_2 = 1$) and MMLE, hierarchical Bayes with Cauchy prior on $\t$ and with
  Cauchy prior truncated to $[1/n, 1]$. For the hierarchical Bayes approaches, the posterior mean of
  $\t$ was recorded for each iteration. The $p_n$ nonzero means were drawn from a
  $\mathcal{N}(A, 1)$ distribution. Results are based on averaging over 500 iterations. .}
\label{fig:tau}
\end{center}
\end{figure}

We remark on some aspects of the results. First, we see that the zero means are nearly perfectly covered by all methods in all settings, and the main differences lie in the nonzero means. Secondly, coverage of the nonzero means improves as their values increase. Thirdly, the lengths of the credible intervals adapt to the signal size. They are smaller for the zero means than for the nonzero means, and smaller for the nonzero means corresponding to $A = (1/2)\sqrt{2\log{n}}$ than for the nonzero means corresponding to $A = \sqrt{2\log{n}}$ and $A = 2\sqrt{2\log{n}}$, while there is not much difference between the interval lengths in those latter two settings, suggesting that the interval length does not increase indefinitely with the size of the nonzero mean.

Furthermore, empirical Bayes with the simple estimator achieves the lowest overall coverage, and especially bad coverage of the nonzero means. This appears to be due to smaller interval lengths caused by lower estimates of $\t$ compared to the other methods. The normal approximation leads to better coverage than the simple estimator, and has the highest coverage of the nonzero means, even though the corresponding intervals are slightly shorter than those of empirical Bayes with the MMLE and the hierarchical Bayes approaches. However, its coverage of nonzero means is worse than that of those three methods, while the corresponding intervals are longer, except in the case where $A$ is largest. The normal approximation appears to be reasonable for very large signals only.

The hierarchical Bayes approach with a non-truncated Cauchy on $\t$ leads to the highest overall coverage and coverage of the nonzero means, albeit by a small margin. The price is slightly larger intervals compared to the other methods, mostly for the zero means. These larger intervals are most likely due to the larger values of $\t$ that are employed, this being the only approach that allows for estimates of $\t$ larger than one, and it avails itself of the opportunity in the non-sparse setting. Finally, we again observe that the results for empirical Bayes with the MMLE and  hierarchical Bayes with a truncated Cauchy lead to highly similar results. Their coverage is comparable to that of hierarchical Bayes with a non-truncated Cauchy in all settings except when $p_n = 200$ and $A$ is at least at the threshold, in which case the non-truncated Cauchy has slightly better coverage. Their intervals are shorter on average, because $\t$ is not allowed to be larger than one.

In conclusion, empirical Bayes with the simple estimator should not be used for uncertainty quantification. The normal approximation is faster to compute than the marginal credible sets, but leads to worse coverage of the nonzero compared to the empirical Bayes with the MMLE and the hierarchical Bayes approaches, unless the nonzero means are very large. The results of those latter three methods are very similar to each other.  All these results can be understood in terms of the behaviour of the estimate of $\t$: larger values lead to larger intervals and better coverage, which may lead to worse estimates however (as seen in the previous section). Empirical Bayes with the MMLE, or hierarchical Bayes with a truncated Cauchy, appear to be the best choices when considering both estimation and coverage. Those two approaches yield highly similar results and the choice for one over the other may be based on other considerations such as computational ones.

\subsection{\label{sec:sim_selection}Model selection}
We compare the procedure based on credible intervals studied in Section \ref{sec:fdr} to the thresholding method introduced in \cite{Carvalho2010}. Two scenarios are considered. In the first, the signals are either ``small'', ``intermediate'' or ``large'', as defined in Section \ref{sec:adaptive_intervals}. In the second, all signals are drawn from a distribution. 

In the credible interval method, a parameter is selected as a signal if zero is not contained in the corresponding credible interval. For the thresholding method of \cite{Carvalho2010}, the posterior mean is divided by the observation. The result is a number between zero and one, which indicates the amount of shrinkage of that particular observation. If this number is larger than 0.5, the corresponding parameter is considered a signal. For both methods, we estimate $\tau$ by the MMLE. 

In the first scenario, we have $n$ observations, with $p_n$ signals. The $p_n$ signals are divided into three groups, corresponding to the three intervals of Section \ref{sec:adaptive_intervals}. The small ones are equal to $1/n$, the intermediate ones are $0.5\sqrt{2\log(1/\tau_n(p_n))}$, and the large ones are equal to $1.5\sqrt{2\log{n}}$. We study four combinations of $n$ and $p_n$: $n = 400, p_n = 60$; $n = 800, p_n = 60$; $n = 800, p_n = 120$ and $n = 1600, p_n = 120$. We count the number of false positives, that is the noise signals that are incorrectly selected as signals, and the number of correctly selected signals in each group. The number of true discoveries, averaged over $N = 500$ iterations, are in Figure \ref{fig:scen1}, and the FDR is in the upper left panel of Figure \ref{fig:fdr}.

In the second scenario, all signals are drawn from a distribution:  the Laplace  distribution with dispersion parameter equal to 3, the Gamma distribution with shape and scale equal to 2, or the Cauchy distribution with scale equal to 5. The number of false positives and the number of correctly selected variables are counted. The number of true discoveries, averaged over $N = 100$ iterations, are in Figure \ref{fig:scen2}, and the FDRs are in Figure \ref{fig:fdr}.

Both simulation scenarios tell a consistent story: the thresholding method results in more discoveries, both true and false, than the credible interval method. The findings of Figures \ref{fig:scen1} are as expected based on the theoretical results of Section \ref{sec:fdr}: almost none of the small and medium signals are  detected, while the large signals are nearly perfectly detected by both methods. In scenario 2, where the signals are drawn from a distribution, the thresholding method finds more of the signals. Comparing the left and right columns of Figure \ref{fig:scen2}, we see that both methods detect more of the signals when the truth is less sparse. This may be due to the behaviour of the MMLE, which is likely to be larger in the less sparse settings, leading to less shrinkage of the true signals.

The FDR of the credible interval method remains well below 0.05 in all settings (Figure \ref{fig:fdr}). In contrast, the FDR of the thresholding method exceeds 0.10 in all cases, and is much larger still when the observations are drawn from a Cauchy distribution. The FDR of the thresholding method can of course be lowered by taking a different cut-off than 0.5, but no guidelines are available at the moment, and a decrease of the FDR will come at the cost of the number of true discoveries. The credible intervals have low FDR, but fail to detect small and medium observations. We speculate that improvement might be possible by combining the information contained in the posterior mean and variance.

\begin{figure}[h!]
\begin{center}
\includegraphics[width = 0.45\textwidth]{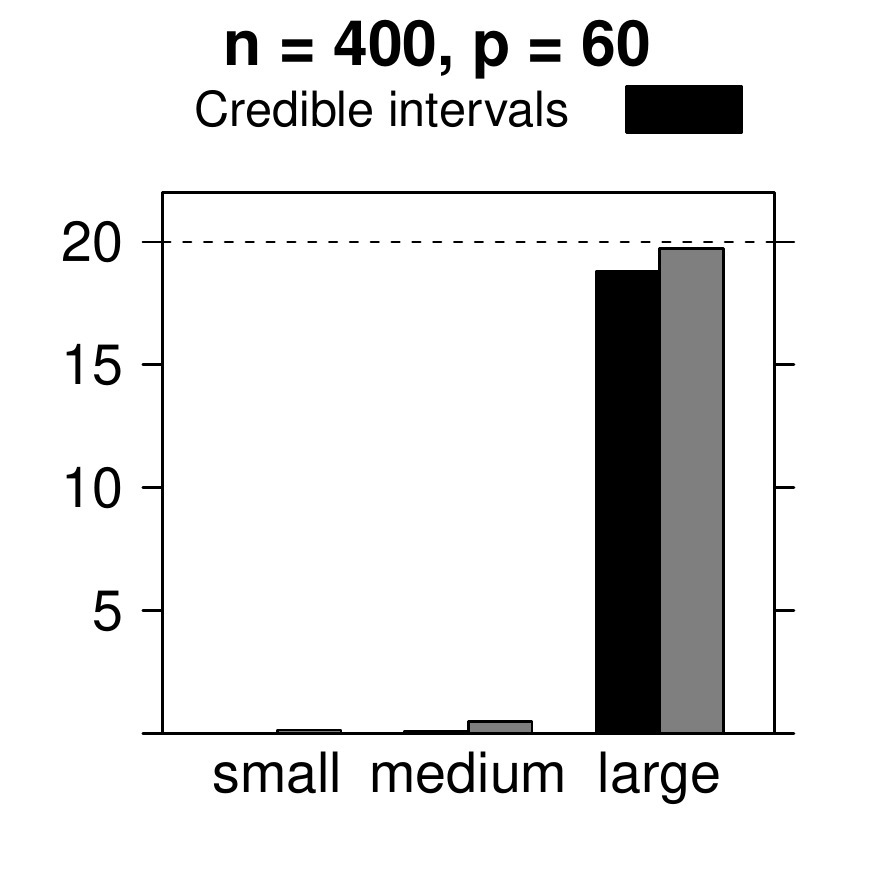} 
\includegraphics[width = 0.45\textwidth]{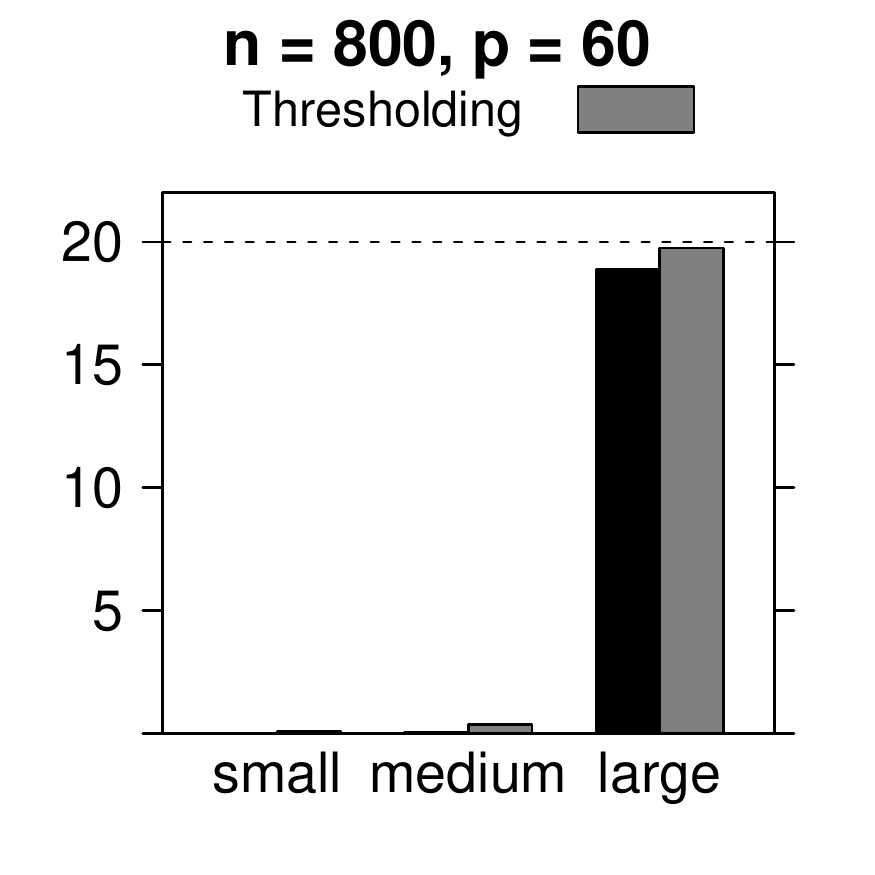} 
\includegraphics[width = 0.45\textwidth]{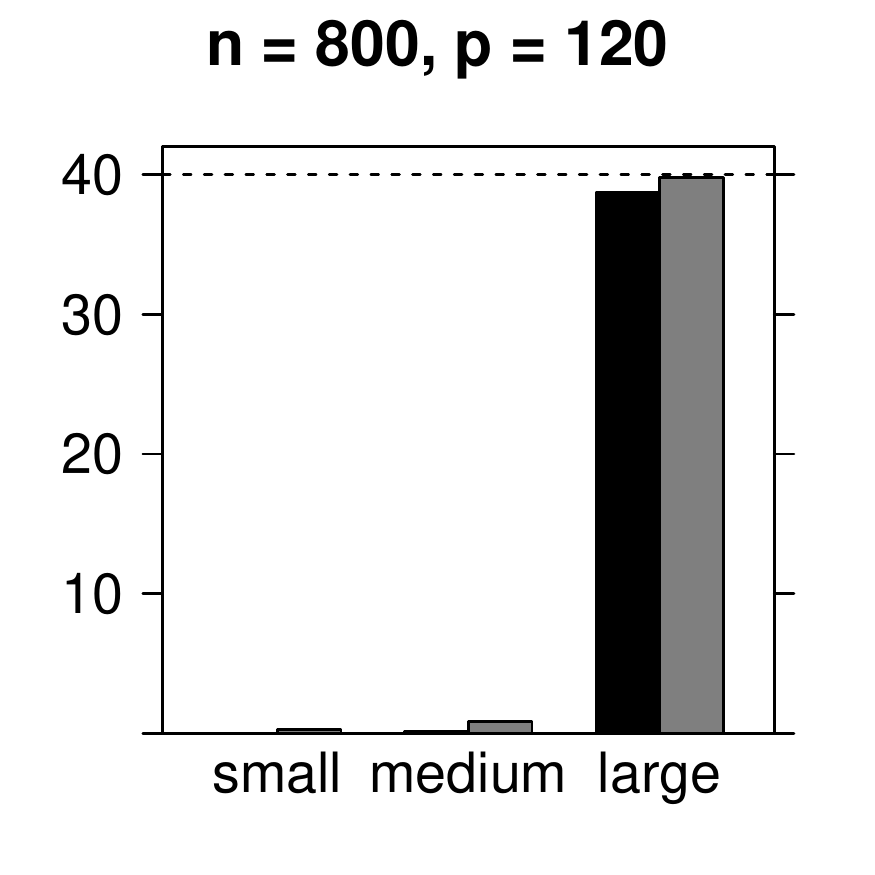} 
\includegraphics[width = 0.45\textwidth]{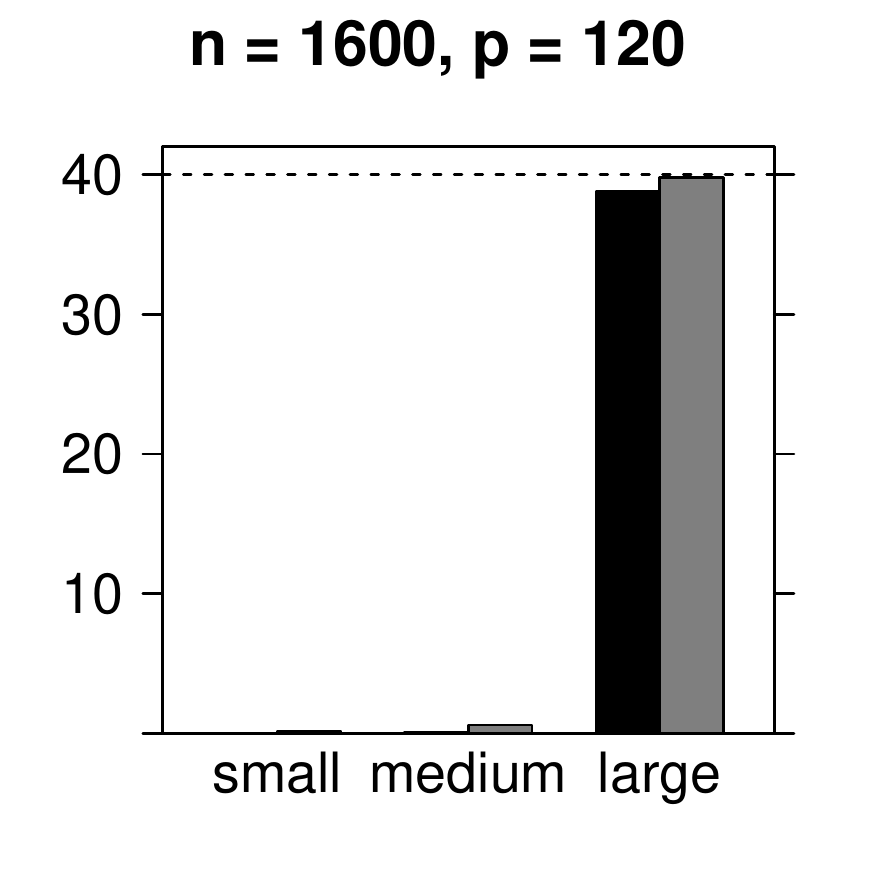} 
\caption{Number of true discoveries, split up by signal size, in scenario 1. The true number of signals in each category is indicated by the dotted line. }
\label{fig:scen1}
\end{center}
\end{figure}

\begin{figure}[h!]
\begin{center}
\includegraphics[width = 0.45\textwidth]{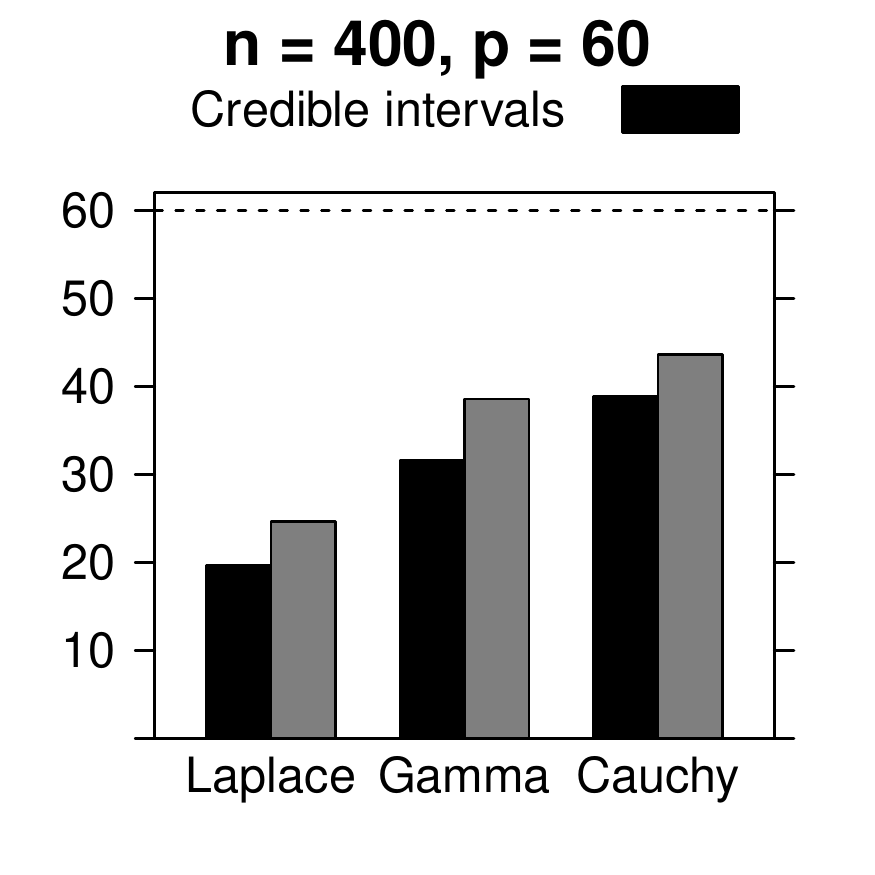} 
\includegraphics[width = 0.45\textwidth]{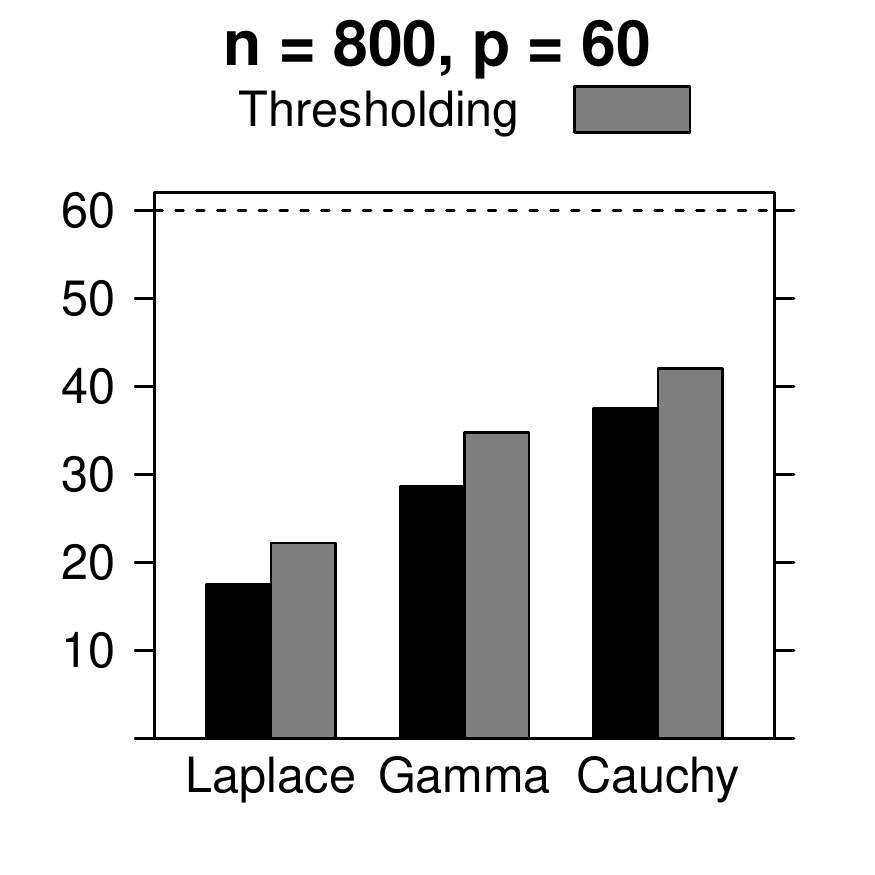} 
\includegraphics[width = 0.45\textwidth]{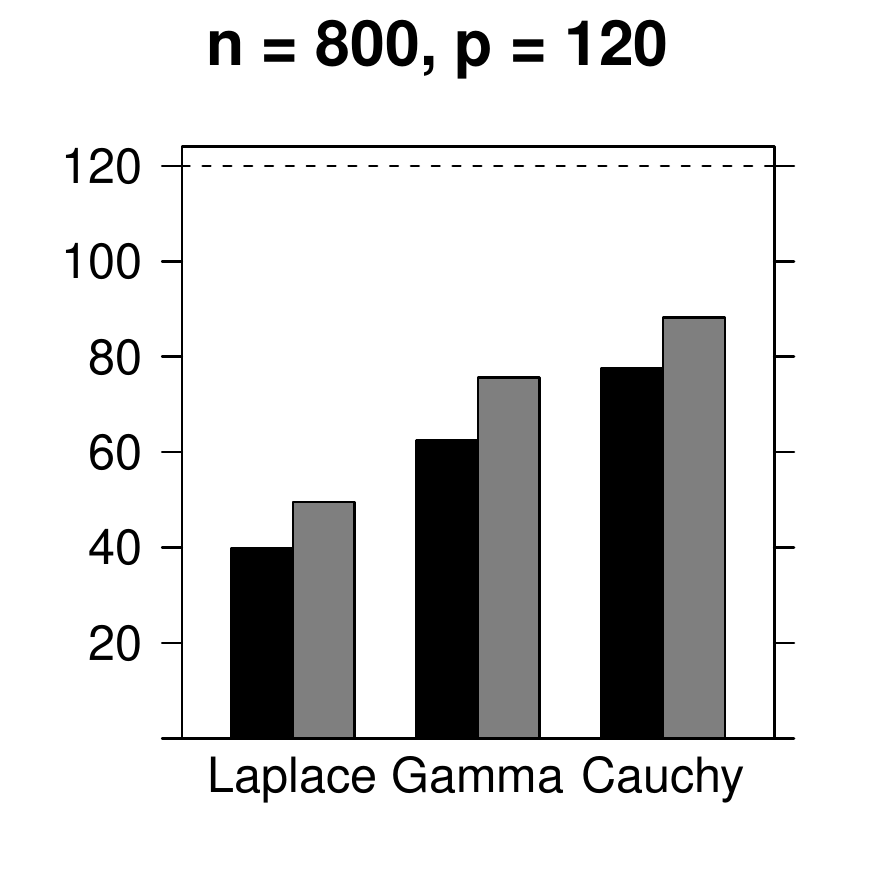} 
\includegraphics[width = 0.45\textwidth]{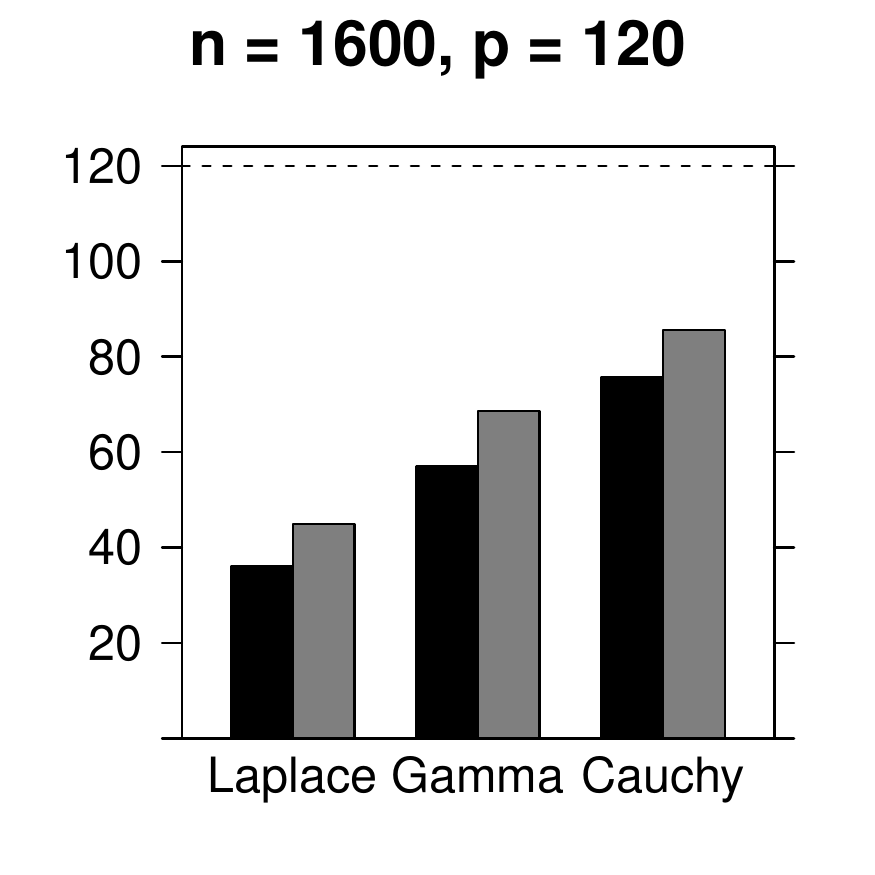} 
\caption{Number of true discoveries in scenario 2. The true number of signals in each category is indicated by the dotted line.}
\label{fig:scen2}
\end{center}
\end{figure}

\begin{figure}[h!]
\begin{center}
\includegraphics[width = 0.45\textwidth]{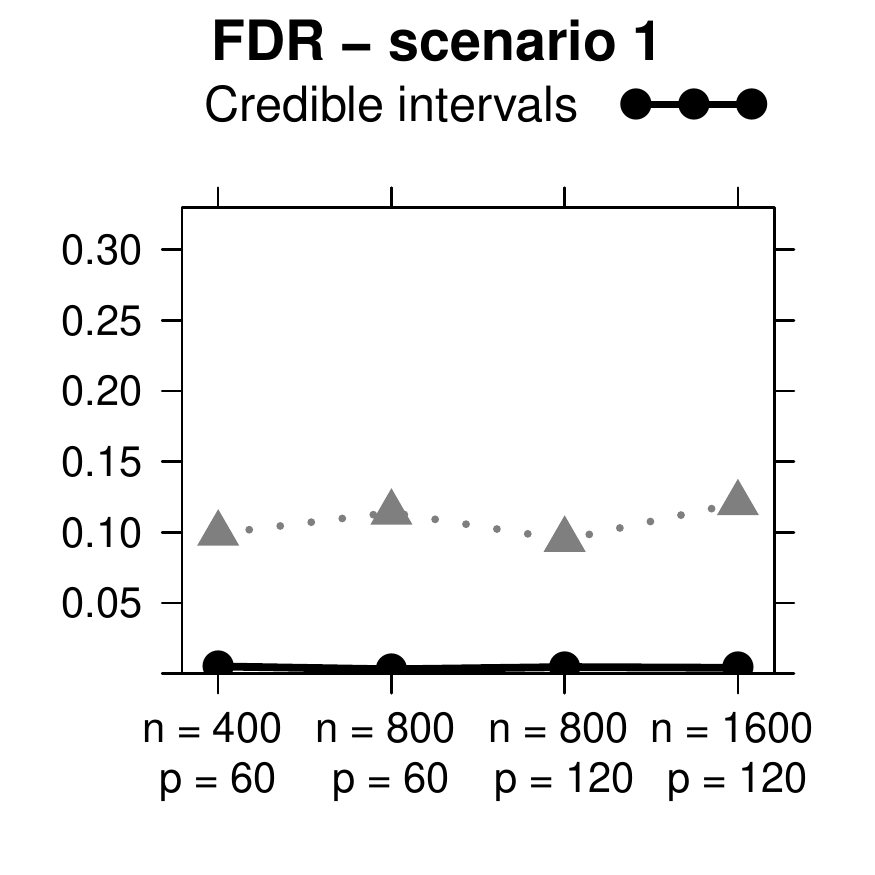} 
\includegraphics[width = 0.45\textwidth]{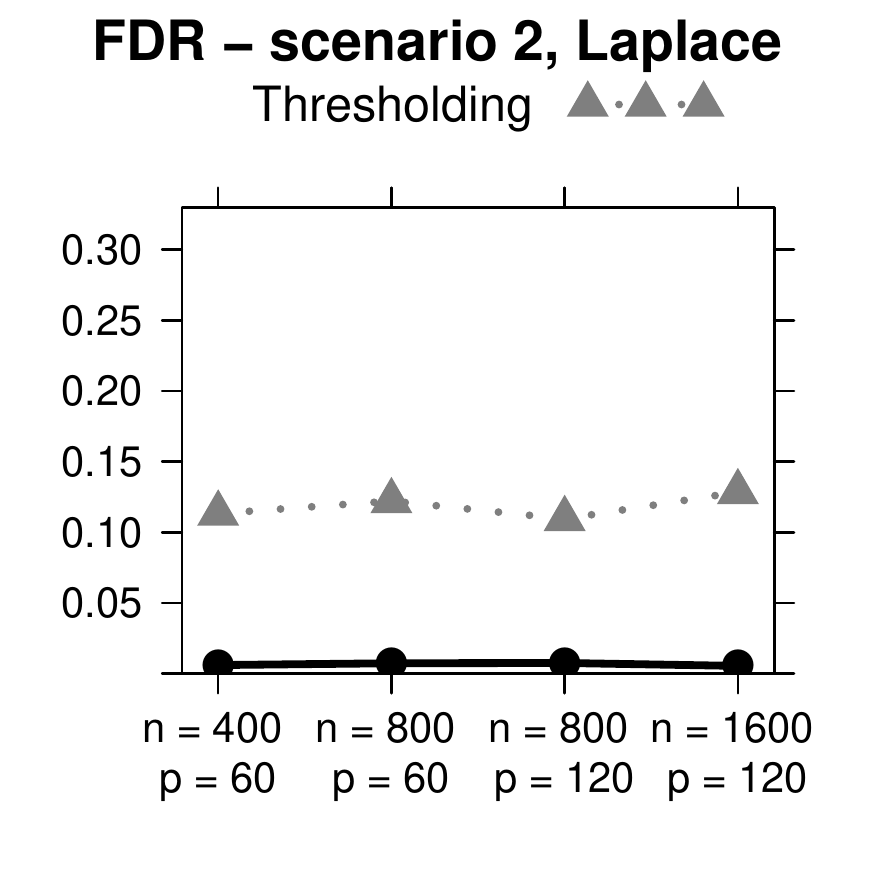} 
\includegraphics[width = 0.45\textwidth]{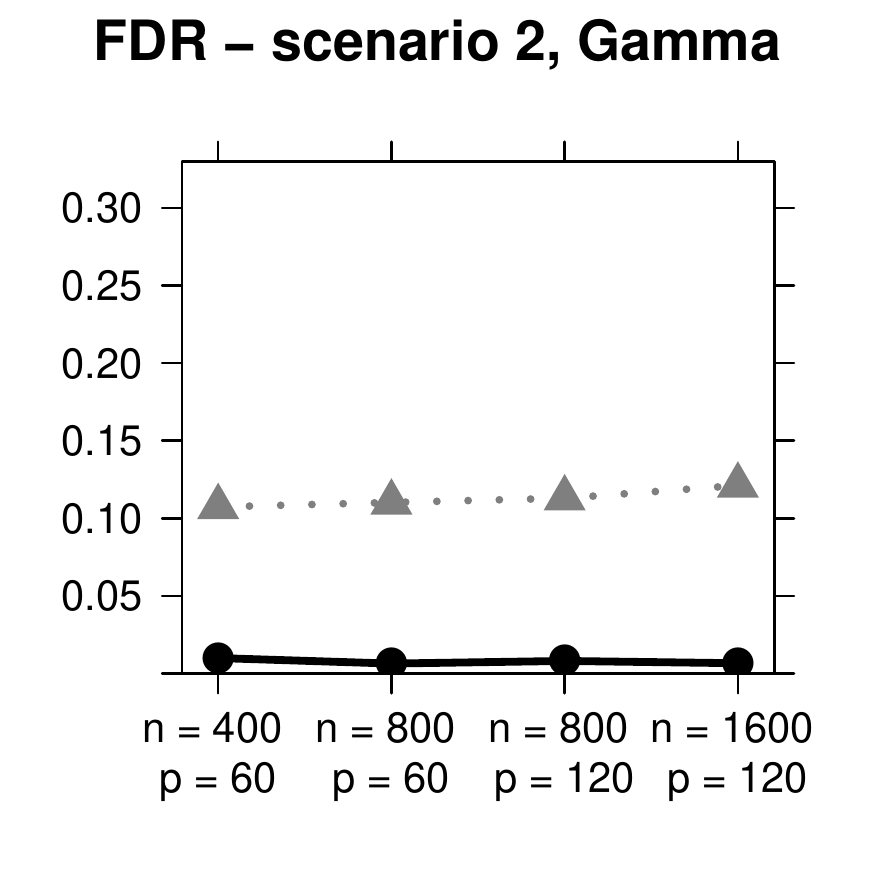} 
\includegraphics[width = 0.45\textwidth]{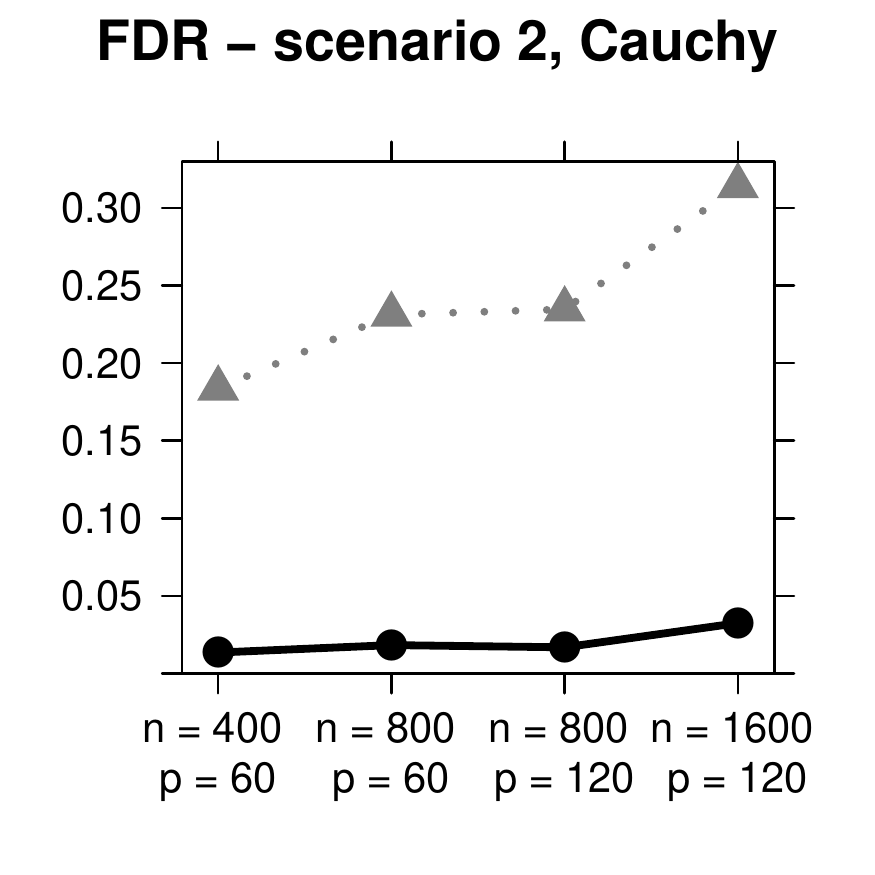} 
\caption{False discovery rate in scenarios 1 and 2.}
\label{fig:fdr}
\end{center}
\end{figure}


\appendix
\section{Notation}
For $k\in\{-1/2,1/2,3/2\}$ define a function $I_k: \RR\to\RR$ by
\begin{align}\label{eq:def.Ik}
I_k(y): =\int_0^1 z^k \frac{1}{\t^2+(1-\t^2)z}e^{y^2z/2}\,dz.
\end{align}
The Bayesian marginal density of $Y_i$, $\psi_\t$, was shown in \cite{contractionpaper} to be given by 
\begin{align}
\psi_\t(y)
&=\frac{\t}{\pi}I_{-1/2}(y)\phi(y).\label{Eqpsitau} 
\end{align}
Note that $I_{-1/2}$ depends on $\t$, but this has been suppressed  from the notation $I_k$. Set
\begin{align}
m_\t(y)=y^2\frac{I_{1/2}(y)-I_{3/2}(y)}{I_{-1/2}(y)}-\frac{I_{1/2}(y)}{I_{-1/2}(y)}.\label{def: m}
\end{align}
The posterior mean for a fixed value of $\tau$ is denoted by $\hat\theta(\tau) = (\hat\theta_1(\tau), \ldots, \hat\theta_n(\tau))$. The posterior mean for fixed $\tau$ and $\lambda$ is denoted by $\hat\theta(\tau, \lambda) = (\hat\theta_1(\tau, \lambda_1), \ldots, \hat\theta_n(\tau, \lambda_n))$, with each entry given by $\hat\th_i(\t,\l_i): = \lambda_i^2\t^2/(1+\lambda_i^2\t^2)Y_i$.

\section{Proofs for the marginal credible intervals}
\subsection{Proof of Theorem \ref{thm:marginal}}\label{sec:marginal}
The posterior distribution of $\th_i$ given $(Y_i, \t,\l_i)$ is normal with mean and variance
\begin{align*}
\hat\th_i(\t,\l_i):=&\E (\th_i\given Y_i,\t,\l_i)=\frac{\l_i^2\t^2}{1+\l_i^2\t^2}Y_i,\\
r_i^2(\t,\l_i):=&\var(\th_i\given Y_i,\t,\l_i)=\frac{\l_i^2\t^2}{1+\l_i^2\t^2}.
\end{align*}
Furthermore, the posterior distribution of $\l_i$ given $(Y_i,\t)$ possesses a density function given by 
$$\pi(\l_i\given Y_i,\t)\propto e^{-\frac{Y_i^2}{2(1+\l_i^2\t^2)}}(1+\t^2\l_i^2)^{-1/2}(1+\l_i^2)^{-1}.$$
The parameter $\th_{0,i}$ is contained in $C_{ni}(L,\t)$ if and only if 
$|\th_{0,i}-\hat\th_i(\t)|\le L \hat r_i(\a,\t)$. We show that this is true, or not, for
$\th_{0,i}$ belonging to the three regions separately for $\cS$, $\cL$ and $\cM$. 

Case $\cS$: proof of \eqref{eq: margcovS}. 
If $i\in \cS$, then $|\th_{0,i}-\hat\th_i(\t)|\le k_S\t+\t |Y_i|e^{Y_i^2/2}$, by the triangle inequality
and Lemma~\ref{LemmaBoundsPostMeanVariance}(iii).
Below we  show that $\hat r_i(\a,\t)\ge \t z_\a c$, with probability tending to one, for $z_\a$
the standard normal upper $\a$-quantile and every $c<1/2$.
Hence  $\th_{0,i}\in C_{ni}(L,\t)$ as soon as $ |Y_i|e^{Y_i^2/2}\le L z_\a c-k_S$.

For $i\in \cS$ the variable $|Y_i|$ is stochastically bounded by $|\th_{0,i}|+|\varepsilon_i|\le
k_S\t +|\varepsilon_i|$. Since the variables $|\varepsilon_i|$ are i.i.d.\ with quantile
function $u\mapsto \Phi^{-1}((u+1)/2)\le \sqrt{2\log (2/(1-u))}$, a fraction $1-\g$ of the variables $Y_i$ with $i\in \cS$
is bounded above by $k_S\t+\sqrt{2\log (2/\g)}+\d=k_S\t+\z_{\g/2}+\d$, with probability tending to 1, for any $\d>0$. 
Then the corresponding fraction of parameters $\th_{0,i}$ is contained in their credible interval if $L$ is chosen big enough that
$$Lz_\a c-k_S \ge (k_S\t+\z_{\g/2}+\d )e^{(k_S\t+\z_{\g/2}+\d)^2/2},$$
where $\eps\ra 0$ if $\g\ra 0$ and can be chosen arbitrarily small if $\d$ is chosen small and $\t\ra0$. As the right hand side of the above inequality is bounded above by $\frac 2\g\z_{\g/2}(1+\eps)$, 
this is certainly true for $L_S$ as in the theorem.

We finish by proving the lower bound for the radius $\hat r_i(\a,\t)$. 
Because the conditional distribution of $\th_i$ given $(Y_i, \t,\l_i)$ is normal with mean $\hat \th_i(\t,\l_i)$
it follows by Anderson's lemma that 
$\Pi\bigl(\th_i: |\th_i-\hat\th_i(\t)| > r \given Y_i,\t,\l_i\bigr)\ge
\Pi\bigl(\th_i: |\th_i-\hat\th_i(\t,\l_i)| >  r \given Y_i,\t,\l_i\bigr)$, for any $r>0$. Furthermore,
by the monotonicity of the variance in $\l_i$ of this conditional distribution,
the last function is increasing in $\l_i$. If $\tilde\pi(\cdot\given \t)$ is the probability density given by 
$$\tilde{\pi}(\l_i\given \t)\propto (\l_i^2\t^2+1)^{-1/2}(1+\l_i^2)^{-1},$$
then $\l_i\mapsto \pi(\l_i\given Y_t,\t)/\tilde \pi(\l_i\given \t)$ is increasing. Combining the preceding
observations with Lemma~\ref{Lem: DominatingDensity}, we see that 
\begin{align}
\alpha&=\int_0^{\infty} \Pi(\th_i: |\th_i-\hat\th_i(\t)| > \hat r_i(\a,\t) \given Y_i,\t,\l_i)\pi(\l_i \given Y_i,\t)\,d\l_i\nonumber\\
&\geq \int_0^{\infty} \Pi(\th_i: |\th_i-\hat\th_i(\t,\l_i)| > \hat r_i(\a,\t) \given Y_i,\t,\l_i) \tilde{\pi}(\l_i\given \t)\,d\l_i.\label{eq: hulp0}
\end{align}
On the other hand, since $\sd(\th_i\given Y_i,\t,\l_i)\ge \t/2(1+o(1))$, for $\l_i\ge 1/2$, the normality of the conditional
distribution of $\th_i$ given $(Y_i, \t,\l_i)$ gives that
\begin{align}
&\int_0^\infty \Pi\bigl(\th_i: |\th_i-\hat\th_i(\t,\l_i)|>z_{\a}\t/2(1+o(1))) \given Y_i,\t,\l_i) \tilde{\pi}(\l_i \given \t\bigr)\,d\l_i \label{eq: margcovS0}\\
&\qquad\geq 2\a\, \tilde{\Pi}(\l_i\ge 1/2 \given \t)\nonumber
\geq 2\a\times 2/3>\a.
\end{align}
Here the second last inequality follows from
$$\frac{\int_{0}^{1/2} (\l_i^2\t^2+1)^{-1/2}\,(1+\l_i^2)^{-1}\, d\l}
{\int_{0}^{\infty}(\l_i^2\t^2+1)^{-1/2}\,(1+\l_i^2)^{-1}\,d\l_i }
\ra \frac{\int_{0}^{1/2}(1+\l_i^2)^{-1}\,d\l_i}{\int_{0}^{\infty}(1+\l_i^2)^{-1}\,d\l_i}<\frac 13,$$
as $\t\ra0$, by two applications of the dominated convergence theorem.
Combination of \eqref{eq: hulp0} and \eqref{eq: margcovS0} shows that $\hat r_i(\a,\t)\ge z_\a\t/2(1+o(1))$.

Case $\cL$: proof of \eqref{eq: margcovL}. 
If $i\in \cL$, then 
\begin{align}
|\th_{0,i}-\hat\th_i(\t)|\le |\th_{0,i}-Y_i|+|Y_i-\hat\th_i(\t)|\le |\varepsilon_i|+2\z_\t^{-1},\label{eq: hulp01}
\end{align}
 eventually, provided
$|Y_i|\ge A\z_\t$ for some constant $A>1$, by the triangle inequality and Lemma~\ref{LemmaBoundsPostMeanVariance}(i).
Below we show that $\hat r_i(\a,\t)\ge z_\a+o(1)$, with probability tending to one.
It then follows that $\th_{0,i}\in C_{ni}(L,\t)$ as soon as $ |Y_i|\ge A\z_\t$ and $|\varepsilon_i|\le L z_\a+o(1)-2\z_\t^{-1}=L z_\a+o(1)$.

For $i\in \cL$ the variable $|Y_i|$ is lower bounded by $|\th_{0,i}|-|\varepsilon_i|\ge k_L\z_\t-|\varepsilon_i|$ and hence
 $|Y_i|\ge A\z_\t$ if $|\varepsilon_i|\le (k_L-A)\z_\t$. This is automatically satisfied if $|\varepsilon_i|\le L z_\a+o(1)$, for constants
$L$ with $L\ll \z_\t$. As for the proof of Case $\cS$ we have that 
$|\varepsilon_i|\le L z_\a+o(1)$ with probability tending to one for a fraction $1-\g$ of the indices $i\in \cS$
if $L\ge z_\a^{-1}\z_{\g/2}+\d$, for some $\d>0$. 

The proof that $\hat r_i(\a,\t)\ge z_\a+o(1)$ follows the same lines as the proof of the corresponding result
in Case $\cS$, expressed in \eqref{eq: hulp0} and \eqref{eq: margcovS0}, but with the true density $\pi$ instead of $\tilde \pi$.
Inequality \eqref{eq: hulp0} with $\pi$ instead of $\tilde\pi$ is valid by Anderson's lemma, while in
 \eqref{eq: margcovS0} we replace $z_\a\t/2(1+o(1))$ by $z_\a+o(1)$.
Since $\var(\th_i\given Y_i,\t,\l_i)\ge g_\t/(1+g_\t)=1+o(1)$ for every $\l_i\ge g_\t/\t$ and $g_\t\ra\infty$, the desired
result follows if $\Pi\bigl(\l_i\ge g_\t/\t\given Y_i,\t\bigr)$ is eventually bigger than $2/3$, for every $i$ such that
$|Y_i|\ge A\z_\t$. Now by the form of $\pi(\l_i\given Y_i,\t)$, for any $c, d>0$,
\begin{align*}
\Pi(\l_{i}\leq g_\t/\t \given Y_i,\t)
&\le\frac{e^{-\frac{Y_i^2}{2(1+c^2)}}\int_{0}^{c/\t} (1+\l^2)^{-1}\,d\l+ e^{-\frac{Y_i^2}{2(1+g_\t^2)}}\int_{c/\t}^{g_\t/\t} (1+c^2/\t^2)^{-1}\,d\l}
{e^{-\frac{Y_i^2}{2(1+d^2g_\t^2)}}\int_{d g_\t/\t}^{2d g_\t/\t} (1+4d^2g_\t^2)^{-1/2}(1+4d^2g_\t^2/\t^2)^{-1}\,d\l}\\
&\lesssim \frac{\exp\Bigl[-\frac{Y_i^2}{2}\Bigl(\frac{1}{1+c^2}-\frac1{1+d^2g_\t^2}\Bigr)\Bigr] 
+\exp\Bigl[-\frac{Y_i^2}{2}\Bigl(\frac{1}{1+g_\t^2}-\frac1{1+d^2g_\t^2}\Bigr)\Bigr] g_\t \t}
{(g_\t/\t) (1/g_\t) (\t^2/g_\t^2)}.
\end{align*}
For $|Y_i|>A\z_\t$ and $A>1$ we can choose $c$ sufficiently close to zero so that the 
first exponential is of order $\t^{A'}$ for some $A'>1$. Then it is much smaller than the denominator, which is
of order $\t/g_\t^{2}$, provided $g_\t$ tends to infinity slowly.
If we choose $d>1$, then the term involving the second exponential will also tend to zero for $|Y_i|>A\z_\t$
as soon as $e^{-c\z_\t^2/g_\t^2}g_\t^3\ra 0$, for a sufficiently small constant $c$.
This is  true (for any $c>0$) for instance if $g_\t= \sqrt{\z_\t}$. Then the quotient tends to zero,
and is certainly smaller than $1/3$.

Case $\cM$: proof of \eqref{eq: margcovM}.
We show below that $\hat r_i(\a,\t)\lesssim U_\t:=\t(1\vee |Y_i|e^{Y_i^2/2})$, with probability tending to one,
whenever $i\in \cM$. By Lemma~\ref{LemmaBoundsPostMeanVariance}(iii) exactly the same
bound is valid for $|\hat\th_i(\t)|$. If $|\hat\th_i(\t)|+\hat r_i(\a,\t)\lesssim U_\t$, but $|\th_{0,i}|\gg U_\t$
then $\th_{0,i}\notin C_{ni}(L,\t)$ eventually, and hence it suffices to 
prove that the probability of the event that $|\th_{0,i}|\gg U_\t$ tends to one whenever $i\in \cM$. Consider two cases.
If $|\th_{0,i}|\leq 1$, then $|Y_i|\le 1+|\varepsilon_i|=O_P(1)$ and hence $U_\t=O_P(\t)$. For $i\in \cM$, we have
$|\th_{0,i}|\gg \t$ and hence $|\th_{0,i}|\gg U_\t$ with probability tending to one.
On the other hand, if $|\th_{0,i}|\ge 1$ but $|\th_{0,i}|\le k_M\z_\t$, then $|Y_i|\le k\z_\t$ with probability tending
to one for any $k>k_M$, and hence $U_\t\lesssim \t\z_\t e^{k^2\z_\t^2/2}=\t^{1-k^2}\z_\t$. Since $k_M<1$ we can choose $k<1$, so that 
$\t^{1-k^2}\z_\t \ra0$, and again we have $|\th_{0,i}|\gg U_\t$ with probability tending to one. 

We finish by proving that $\hat r_i(\a,\t)\lesssim U_\t$, with probability tending to one.
As a first step we show that, for $k<1$,
\begin{align}
\lim_{M\ra\infty}\sup_{|y|\le k\z_\t}\Pi(\l_i\geq M \given Y_i=y,\t)\ra 0.\label{eq: hyper_post_mass}
\end{align}
By the explicit form of the posterior density of $\l_i$ we have
\begin{align*}
\Pi(\l_i\geq M \given Y_i=y,\t)
&\leq \frac{\int_{M}^{\infty} e^{-\frac{y^2}{2(1+\l_i^2\t^2)}}(1+\l_i^2\t^2)^{-1/2}(1+\l_i^2)^{-1}\,d\l_i}
{\int_{1}^2e^{-\frac{y^2}{2(1+\l_i^2\t^2)}}(1+\l_i^2\t^2)^{-1/2}(1+\l_i^2)^{-1}\,d\l_i}\\
&\le e^{y^2/2}5\sqrt 2\int_{M}^{\infty} e^{-\frac{y^2}{2(1+\l_i^2\t^2)}} (1+\l_i^2\t^2)^{-1/2}(1+\l_i^2)^{-1}\,d\l_i.
\end{align*}
We split the remaining integral over the intervals $[M,\t^{-a})$ and $[\t^{-a},\infty)$, for some $a<1$. 
On the first interval we use that $y^2/(1+\l_i^2\t^2)=y^2+o(1)$, uniformly in $|y|\lesssim \z_\t$ and $\l_i\le \t^{-a}$,
while on the second we simply bound the factor $e^{-{y^2}/{(2(1+\l_i^2\t^2))}}$ by 1, to see that the preceding
display is bounded above by 
$$e^{y^2/2}5\sqrt 2 \Bigl[e^{-y^2/2}e^{o(1)}\int_{M}^{\t^{-a}} (1+\l_i^2)^{-1}\,d\l_i+ \int_{\t^{-a}}^\infty (1+\l_i^2)^{-1}\,d\l_i\Bigr].$$
The first term in square brackets (times the leading term) contributes less than a multiple of $\int_M^\infty\l^{-2}\,d\l=1/M$, while the
second term contributes less than $e^{y^2/2}\t^a\le \t^{-k^2+a}$, for $|y|\le k\z_\t$, which tends to zero if $a>k^2$. 
This concludes the proof of \eqref{eq: hyper_post_mass}.

By the reverse triangle inequality, for any $M>0$,
\begin{align*}
&\int_0^{\infty}\Pi(\th_i:|\th_i-\hat\th_i(\t)|\geq r+|\hat\th_i(\t,\l_i)-\hat\th_i(\t)|\ \given Y_i,\l_i,\t)\pi(\l_i \given Y_i,\t)\,d\l_i\\
&\quad\leq \int_0^{M}\Pi(\th_i: |\th_i-\hat\th_i(\t,\l_i)|\ge r\given Y_i,\l_i,\t)\pi(\l_i \given Y_i,\t)\,d\l_i+\Pi(\l_i\ge M \given Y_i,\t).
\end{align*}
For sufficiently large $M$ the second term on the far right is smaller than $\a/2$ by the preceding
paragraph and for $r= z_{\a/4}\sup_{\l\le M}r_i(\t, \l)$ the first term on the right is smaller than $\a/2$ as well,
by the normality of $\th_i$ given $(Y_i,\l_i,\t)$ and the definition of $r_i(\t, \l_i)$. The inequality remains valid if
$|\hat\th_i(\t,\l_i)-\hat\th_i(\t)|$ in the first line is replaced by $\sup_{\l\le M}|\hat\th_i(\t,\l_i)|+|\hat\th_i(\t)|$. It follows that
$$\hat r_i(\a,\t)\le z_{\a/4}\sup_{\l\le M}r_i(\t, \l) +\sup_{\l\le M}|\hat\th_i(\t,\l_i)|+|\hat\th_i(\t)|.$$
The first term is bounded above by $M\t$,  and the second by $M\t |Y_i|$, by the definitions
of $r_i(\t,\l)$ and $\hat\th_i(\t,\l)$, while $|\hat\th_i(\t)|\le\t |Y_i|e^{Y_i^2/2}$, by 
Lemma~\ref{LemmaBoundsPostMeanVariance}(iii). This concludes the proof that
$\hat r_i(\a,\t)\lesssim U_\t$.

\subsection{Proof of Theorem \ref{thm:marginal_adapt}}\label{sec:marginal_adapt}
The proof for the empirical Bayes procedure closely follows the proof of Theorem~\ref{thm:marginal}.
The lower bounds $\hat r_i(\a,\t)\ge \t z_\a (1+o(1))$ and $\hat r_i(\a,\t)\ge z_\a+o(1)$ in the cases $\cS$ and $\cL$,
and the upper bound $\hat r_i(\a,\t)\le \t(1\vee |Y_i|e^{Y_i^2/2})$ in case $M$, with probability tending to one,
remain valid when $\t$ is replaced by $\hat\t_n$. The remainders of the arguments then go
through with minor changes, where it is used that $\hat\t_n\ge 1/n$,  $\z_{\hat\t_n}\le \sqrt{2\log n}$ and $\hat\t_n\le \t_n(p)$
with probability tending to one by Condition~\ref{cond.eb}. Note the slightly changed right boundary of the set
$\cS_a$ and left boundary of the set $\cL_a$, which refer to ``extreme'' cases.

In the proof for the hierarchical Bayes  method, we denote by
$\hat{\th}_i$ the $i$th coordinate of the hierarchical posterior mean $\hat \th$ and by $\hat{r}_i(\a)$ the (Bayesian) 
radius of the marginal hierarchical Bayes credible interval. Hence $\th_{0,i}$ is contained in this credible interval if
$|\th_{0,i}-\hat\th_i|\le L \hat r_i(\a)$.

By Lemma~\ref{lem:HBhyper} we have that $\Pi(1/n<\t<5 t_n \given Y^n)\ra 1$ under Condition~\ref{cond.hyper.1}, 
or $\Pi(1/n<\t< (\log n)t_n \given Y^n)\ra 1$ under the weaker Condition~\ref{cond.hyper.1a}.  

Case $\cS_a$: proof of the hierarchical Bayes version of \eqref{eq: margcovS_a}. 
For $i\in \cS_a$ we have $|Y_i|\le k_S/n+|\varepsilon_i|$.
Because the $1-\g$-quantile of the absolute errors $|\varepsilon_i|$ is bounded above by
$\z_{\g/2}$, the set $\cS_\g$ of coordinates $i\in \cS_a$ such that $|Y_i|\le \z_{\g/2}+\d$ contains at least a fraction
$1-\g$ of the elements of $\cS_a$, with probability tending to one. We show below that 
with probability tending to one
both $\hat r_i(\a)\ge c|\hat\th_i|z_{\a/2}\z_{\g/2}$ and $\hat r_i(\a)\ge z_\a/(2n)$ for $i\in \cS_\g$, and any $c<1/2$. 
Then $|\hat\th_i-\th_{0,i}|\leq |\hat\th_i|+k_S/n\leq [(cz_{\a/2}\z_{\g/2})^{-1}+(2/z_\a) k_S]\hat r_i(\a)$,
and hence $\th_{0,i}$ is contained in its credible interval for every $i\in \cS_\g$ if $L\ge (cz_{\a/2}\z_{\g/2})^{-1}+(2/z_\a) k_S$.

To show that $\hat r_i(\a)\ge c|\hat\th_i|z_{\a/2}\z_{\g/2}$ for $i\in \cS_\g$, we assume $Y_i>0$ for simplicity.
Then $\hat\th_i(\t,\l_i)>0$ for every $(\t,\l_i)$ and hence so is $\hat\th_i$.
By its definition $\hat\th_i(\t,\l_i)=r_i^2(\t,\l_i)Y_i$. Since $r_i(\t,\l_i)\le 1$, it follows that 
$\hat\th_i(\t,\l_i)\le r_i(\t,\l_i) (\z_{\g/2}+\d)$, for every $i\in \cS_\g$. If $\hat\th_i(\t,\l_i)\ge \hat\th_i/2$, then
$r_i(\t,\l_i)\ge \hat\t_i (2\z_{\g/2}+2\d)$ and
we can conclude, using Anderson's lemma and the conditional normal distribution of $\th_i$ given
$(Y_i,\l_i,\t)$ with variance $r_i^2(\t,\l_i)$, that $\Pi\bigl(\th_i: |\th_i-\hat\th_i|\ge z_{\a/2}\hat\th_i/(2 \z_{\g/2}+2\d) \given Y_i,\t, \l_i\bigr)\ge \a$.
If $\hat\th_i(\t,\l_i)\le \hat\th_i/2$, then $\th_i\le \hat\th_i(\t,\l_i)$ implies that $|\th_i-\hat\th_i|\ge \hat\th_i/2$,
and hence  $\Pi\bigl(\th_i: |\th_i-\hat\th_i|\ge \hat\th_i/2 \given Y_i,\t, \l_i\bigr)\ge 
\Pi\bigl(\th_i: \th_i\le \hat\th_i(\t,\l_i) \given Y_i,\t, \l_i\bigr)= 1/2$, since $ \hat\th_i(\t,\l_i)$ is the median
of the conditional normal distribution of $\th_i$. For $c_0=(1/2)\wedge (z_{\a/2}/(2\z_{\g/2}+2\d))$ and $\a\le 1/2$,
we have that
$\Pi\bigl(\th_i: |\th_i-\hat\th_i|\ge c\hat\th_i\bigr)\ge \a$ in both cases, and hence
\begin{align*}
& \Pi\bigl(\th_i: |\th_i-\hat\th_i|\geq c_0\hat\th_i \given Y^n\bigr)\\
&\quad\int_{1/n}^{1}\int_0^{\infty} \Pi\bigl(\th_i: |\th_i-\hat\th_i|\geq c_0\hat\th_i \given Y_i,\t,\l_i\bigr) \pi(\l_i \given Y_i,\t) \pi(\t \given Y^n)\,d\l_i\, d\t
\geq \a.
\end{align*}
Thus $\hat r_i(\a)\ge c_0\hat\th_i$ by the definition of $\hat r_i(\a)$. 

For the proof that $\hat r_i(\a)\geq z_{\a}/(2n)$, we first note that, similarly to \eqref{eq: hulp0},
\begin{align*}
\a&=\Pi\bigl(\th_i: |\th_i-\hat\th_i|\geq \hat r_i(\a) \given Y^n\bigr)\\
&\ge \int_{1/n}^{1}\int_{0}^{\infty}
\Pi\bigl(\th_i: |\th_i-\hat\th_i(\t,\l_i)|\geq \hat r_i(\a) \given Y_i,\t,\l_i)\pi(\l_i \given Y_i,\t\bigr)\pi(\t \given Y^n)\,d\l_i\, d\t
\end{align*}
On the other hand, since $r_i(\t,\l_i)\geq 1/(2n)(1+o(1))$, whenever  $\t\in[1/n,5t_n]$ and $\l_i>1/2$,
we have similarly to \eqref{eq: margcovS0},
\begin{align*}
&\int_{1/n}^{1}\int_{0}^{\infty}\Pi\bigl(\th_i: |\th_i-\hat\th_i|\geq z_{\a}/(2n) \given Y_i,\t,\l_i\bigr)\pi(\l_i \given Y_i,\t)\pi(\t \given Y^n)\,d\l_i\, d\t\\
&\quad\geq \int_{1/n}^{5t_n}\int_{1/2}^{\infty}
\Pi\bigl(\th_i: |\th_i-\hat\th_i|\geq z_{\a} r_i(\t,\l_i) \given Y_i,\t,\l_i\bigr)\pi(\l_i \given Y_i,\t)\pi(\t \given Y^n)\,d\l_i\, d\t\\
&\quad\ge \int_{1/n}^{5t_n} (4\a/3) \pi(\t \given Y^n)d\t>\a,
\end{align*}
where the lower bound $4\a/3$ follows as in \eqref{eq: margcovS0}.
Together the two preceding displays imply that $\hat r_i(\a)\geq z_{\a}/(2n)$.

Case $\cL_a$: proof of the hierarchical Bayes version of \eqref{eq: margcovL_a}. 
If $i\in \cL_a$, then $|\th_{0,i}|\ge k_L\sqrt{2\log n}=k_L\z_{1/n}$ and hence 
$|Y_i|\ge k_L\z_{1/n}-|\varepsilon_i|$. The subset $\cL_\g$ of $i$ with $|\varepsilon_i|\le \z_{\g/2}+\d$ contains a fraction of at least $1-\g$ of the
elements of $\cL_a$ eventually with probability tending to one, and $|Y_i|\ge k \z_{1/n}$ for every $i\in \cL_\g$ and some constant $k>1$. 
Then $|Y_i-\th_i(\t)|\lesssim \z_\t^{-1}$ for $\t\ra0$ and 
$|Y_i-\th_i(\t)|\lesssim (\log \z_{1/n})/\z_{1/n}$ for $\t$ bounded away from zero, by Lemma~\ref{LemmaBoundsPostMeanVariance} (i) and (vii), respectively,
and hence $|Y_i-\hat\th_i|$ tends to zero, by Jensen's inequality. It follows that
$|\th_{0,i}-\hat \th_i|\le |\th_{0,i}-Y_i|+|Y_i-\hat\th_i|\le \z_{\g/2}+\d'$ for ever $i\in \cL_\g$ with probability tending to one.
We can prove that $\hat r_i(\a)\geq z_{\a}(1+o(1))$ similarly as in the proof for Case $\cL$
in the proof of Theorem~\ref{thm:marginal} (adapted similarly as in the proof for case $\cS_a$), but now using that 
$r_i(\t,\l_i)\geq 1+o(1)$, whenever $\t\in[1/n,5t_n]$ and $\l_i\geq g_\t/\t$, for some $g_\t\ra\infty$.
Thus $|\th_{0,i}-\hat \th_i|\le L\hat r_i(\a)$ with probability tending to one, if $Lz_\a \ge \z_{\g/2}+\d'$.

Case $\cM_a$: proof of the hierarchical Bayes version of  \eqref{eq: margcovM_a}. 
First assume that Condition~\ref{cond.hyper.1} holds, so that
$\Pi(\t\le 5 t_n\given Y^n)\ra 1$ in probability, by Lemma~\ref{lem:HBhyper}, and in fact
$\Pi(\t\le 5 t_n\given Y^n)\le e^{-c_0 p_n}$, for some $c_0>0$ by the proof of the lemma.
Since $i\in \cM_a$ we have that $|Y_i|\le |\th_{0,i}|+|\varepsilon_i|\le k\z_{\t_n}$,
with probability tending to one and some $k<1$. We show below that
both $\hat{r}_{i}(\a)$ and $|\hat\th_i|$ are bounded above by $t_n (1\vee |Y_i|) e^{Y_i^2/2}$, with probability tending to one.
The argument as in the proof Theorem~\ref{thm:marginal}, split in the cases that $|\th_{0,i}|$ is smaller or bigger than $1$,
then goes through and shows that $\th_{0,i}$ is not contained in the credible interval, with probability tending to one.

By the triangle inequality, for any $r>0$,
$$\Pi\bigl(\th_i:|\th_i-\hat\th_i|\geq r +|\hat\th_i(\t,\l_i)-\hat\th_i|\given Y_i,\l_i,\t\bigr)
\leq \Pi\bigl(\th_i: |\th_i-\hat\th_i(\t,\l_i)|\geq r\given Y_i,\l_i,\t\bigr).$$
For $r\ge z_{\a/4}r_i(\t,\l_i)$ the right side is at most $\a/2$. For given $M$ define
\begin{align*}
r_i:=z_{\a/4}\sup_{\substack{\t\in[1/n,5t_n]\\ \l_i\leq M}}r_i(\t,\l_i)+\sup_{\substack{\t\in[1/n,5t_n]\\ \l_i\leq M}}|\hat\th_i(\t,\l_i)|+|\hat\th_i|.
\end{align*}
Then it follows that 
\begin{align*}
&\int_{1/n}^1\int_0^{\infty}\Pi(\th_i: |\th_i-\hat\th_i|\geq r_i \given Y_i,\l_i,\t)\pi(\l_i \given Y_i,\t)\pi(\t \given Y^n)\,d\l_i\, d\t\\
&\quad\leq \a/2 +\int_{1/n}^{5t_n}\int_{M}^{\infty}\pi(\l_i \given \t,Y_i)\pi(\t \given Y^n)\,d\l_i\, d\t + \int_{5t_n}^{1}\pi(\t \given Y^n)\,d\t.
\end{align*}
By \eqref{eq: hyper_post_mass} the second term on the right can be made arbitrarily small by choosing large $M$, and the third term 
tends to zero by Lemma~\ref{lem:HBhyper}.
We conclude that the left side is then smaller than $\a$ which implies that $\hat r_i(\a)\le r_i$. 
Now by the definitions of $r_i(\t,\l_i)$ and $\hat\th_i(\t,\l_i)$ the suprema in 
the definition of $r_i$ are bounded by $z_{\a/4}M 5t_n$ and $M5t_n |Y_i|$, respectively. Furthermore,
by Lemma~\ref{LemmaBoundsPostMeanVariance} (iii) and (ii),
\begin{align*}
|\hat\th_i|&\leq \int_{1/n}^{1} |\hat\th_i(\t)| \pi(\t  \given  Y^n)\,d\t
\lesssim t_n|Y_i|e^{Y_i^2/2}+ |Y_i|\Pi(\t\geq 5t_n \given Y^n)\lesssim t_n |Y_i|e^{Y_i^2/2},
\end{align*}
since $\Pi(\t\geq 5t_n \given Y^n)\lesssim e^{-c_0 p_n}\ll t_n$ if $p_n\gtrsim \log n$.

If the weaker Condition~\ref{cond.hyper.1a} is substituted for Condition~\ref{cond.hyper.1}, then
in the preceding we must replace $t_n$ by $(\log n) t_n$. The arguments go through, but 
with an additional $\log n$ factor in the upper bound on the radius $\hat r_i(\a)$. This is compensated by
the stronger assumption $f_n\gg \log n$ on the lower bound of $\cM_a$.

\subsection{Technical Lemmas}
The next lemma extends Lemma~\ref{lem: LBradius} to nondeterministic values of $\t$.

\begin{lemma}\label{lem: LBradiusAdapt}
If $n\t/\z_\t\rightarrow\infty$, then for every constant $C>0$ there exists a constant $D>0$ such that
$$P_{\th_0}\Bigl(  \inf_{t\in [C^{-1}\t,C\t]} \hat r(\a,t)\ge D\sqrt{n\t\z_\t} \Bigr)\rightarrow 1.$$
\end{lemma}

\begin{proof}
Set $T=[C^{-1}\t,C\t]$.
By the arguments in the proof of Lemma~\ref{lem: LBradius} and with the same notation, for $1/c^2\le 1-\a$,
$$\inf_{t\in T}\hat r^2(\a,t)\ge\inf_{t\in T}\E(W\given Y^n,t)-c \sup_{t\in T} \sd(W\given Y^n,t).$$
By the first assertion of Lemma~\ref{LemmaMomentsOfThetaGivenY} we have $\inf_{t\in T}\E_0 \E(W\given Y^n,t)\gtrsim n\t\z_\t$. 
Combination with Lemma~\ref{lem: uniform1} gives that the infimum on the right side of the display is
bounded below by a multiple of $n\t\z_\t$, with probability tending to one.
By the second assertion of Lemma~\ref{LemmaMomentsOfThetaGivenY} we have
$\E_0\sup_{t\in T}\var(W\given Y^n,t)\lesssim n\t\z_\t^3$.
An application of Markov's inequality shows that the supremum on the right side of the display is
bounded above by $o(n\t\z_\t)$,  with probability tending to one, in view of the
assumption that $n\t/\z_\t\rightarrow\infty$.
\end{proof}

\begin{lemma}\label{Lem:hyperLB}
Suppose that the density of $\pi_n$ is  bounded away from zero on $[1/n,1]$.
For every sufficiently large constant $D$ there exists $d>0$ such 
that $\hat r(\a)\ge d \sqrt{n\z_{\tilde{\t}_n}\tilde{\t}_n}$  with $P_{\th_0}$-probability tending to one,
uniformly in $\th_0$ satisfying the excessive-bias restriction \eqref{condition: EB} 
with $\tilde{p}\geq D \log n$, where $\tilde\t_n = \t_n(\tilde p)$.
\end{lemma}

\begin{proof}
Set $T_n=[C^{-1}\tilde\t_n, C\tilde\t_n]$, for $C$ the constant in Lemma~\ref{lem:hyperLB}. Then by the
definition of $\hat r_n(\a)$ and the latter lemma
$\int_{\t\in T_n} \Pi\bigl(\|\th-\hat\th\|_2\leq r(\a) \given \t,Y^n)\pi(\t \given Y^n\bigr)\,d\t$ is equal to $1-\a+o(1)$.
Therefore there exists $\t=\t(Y^n)\in T_n$ such that
\begin{align*}
\Pi\bigl(\|\th-\hat\th\|_2\leq \hat r(\a) \given \t,Y^n\bigr) \geq 1-2\a.
\end{align*}
Introduce the notation $\tilde{W}=\|\th-\hat\th\|_2^2$,  and denote by $\E(\cdot | Y^n,\t)$ and  $\sd(\cdot | Y^n,\t)$ the posterior expected value 
and standard variation for given $\t$. By an application of  Chebyshev's inequality,
as in the proofs of Lemmas~\ref{lem: LBradius} and~\ref{lem: LBradiusAdapt},
we see that $\hat r(\a)\ge\E(\tilde{W}\given  \t,Y^n)- c \sd(\tilde{W}\given  \t,Y^n)$, for a sufficiently small constant $c>0$. Hence
it suffices to show that $\inf_{\t\in T_n} \E(\tilde{W}\given  \t,Y^n)\gtrsim n\tilde\t_n \z_{\tilde\t_n}$ and 
$\sup_{\t\in T_n} \sd(\tilde{W}\given  \t,Y^n)\ll n\tilde\t_n \z_{\tilde\t_n}$, with $P_{\th_0}$-probability tending to one.

Since $\hat\th(\t)$ is the mean of $\th$ given $(Y^n,\t)$ and the coordinates $\th_i$ are conditionally independent,
for $W=\|\th-\hat\th(\t)\|_2^2$,
\begin{align*}
\E(\tilde{W}\given  \t,Y^n)&= \E(W\given \t,Y^n)+\|\hat\th-\hat\th(\t)\|_2^2\geq \E(W\given \t,Y^n),\\
\var(\tilde{W}\given  \t,Y^n)&\lesssim \sum_{i=1}^{n} \E\big((\th_i-\hat\th_i(\t))^4 \given \t,Y^n\big) + \sum_{i=1}^{n} \big(\hat\th_i-\hat\th_i(\t)\big)^4.
\end{align*}
The proof of Lemma~\ref{lem: LBradiusAdapt} shows that $\inf_{\t\in T_n}\E(W\given \t,Y^n)\gtrsim n\tilde\t_n\z_{\tilde\t_n}$,
with $P_{\th_0}$-probability tending to one, and hence the same conclusion holds for 
$\inf_{\t\in T_n}\E(\tilde{W}\given  \t,Y^n)$.

It remains to deal with the variance in the preceding display. 
By Lemma~\ref{LemmaMomentsOfThetaGivenY} the $\E_0$-expected value of the supremum over $\t\in T_n$ of the first term on the right
is bounded above by $n\tilde\t_n\z_{\tilde\t_n}^3$, which shows that this term is suitably bounded in view of Markov's inequality. 
By Jensen's inequality the second can be bounded as
\begin{align}
\|\hat\th(\t)-\hat\th\|_4^4&\leq \int_{1/n}^{1} \|\hat\th(\t)-\hat\th(t)\|_4^4\pi(t \given Y^n)\,dt\nonumber\\
&\leq \sup_{t\in T_n}\|\hat\th(\t)-\hat\th(t)\|_4^4+
\sup_{t\in[1/n,1]}\|\hat\th(\t)-\hat\th(t)\|_4^4\,\Pi(t\notin T_n \given Y^n),\label{eq: hulp01_new}
\end{align}
where $\|\th\|_4^4=\sum_{i=1}^n\th_i^4$. In view of Lemma~\ref{LemmaBoundsPostMeanVariance} (i)+(ii),
\begin{align*}
\sup_{\t_1,\t_2\in[1/n,1]}\|\hat\th(\t_1)-\hat\th(\t_2)\|_4^4
&\lesssim \sup_{\t\in[1/n,1]}\|\hat\th(\t)-Y^{n}\|_4^4\lesssim 8n(\log n)^2.
\end{align*}
Furthermore $\Pi(\t\notin T_n \given Y^n )\leq e^{-c_3 \tilde{p}}$ by Lemmas~\ref{lem:hyperLB} and~\ref{lem:HBhyper},
for a constant $c_3>0$. Hence for $\tilde{p}\geq D \log n$, where $D>c_3^{-1}$,
the second term on the right hand side of \eqref{eq: hulp01_new} tends to zero. 

To bound the first term of \eqref{eq: hulp01_new} we first use the triangle inequality 
to obtain that $\sup_{t\in T_n}\|\hat\th(\t)-\hat\th(t)\|_4\le 2 \sup_{t\in T_n}\|\hat\th(t)-\th_0\|_4$.
We next split the sum in $\|\hat\th(t)-\th_0\|_4^4$ in the terms with $|\th_{0,i}|>\z_{\tilde\t_n}/10$ and
the remaining terms. 

If $|\th_{0,i}|>\z_{\tilde\t_n}/10$, then we use that 
$|\hat\th_i(t)-\th_{0,i}|\le |\hat\th_i(t)-Y_i|+|Y_i-\th_{0,i}|\lesssim \z_{t}+|Y_i-\th_{0,i}|$, so that
$$\E_{\th_{0,i}} \sup_{t\in T_n}|\th_{0,i} - \hat\th_i(t)|^4 
\lesssim \z_{\tilde\t_n}^4+1\lesssim \z_{\tilde\t_n}^4.$$
By an analogous argument as in the proof of Theorem~\ref{thm:coverage}
the number of terms with $|\th_{0,i}|>\z_{\tilde\t_n}/10$ is bounded by a multiple
of $\tilde p$, so that their total contribution is bounded above by $\tilde p \z_{\tilde\t_n}^4$.

For the terms with $|\th_{0,i}|\le \z_{\tilde\t_n}/10$, we first use that 
$|\hat\th_i(t)-\th_{0,i}|\le |\hat\th_i(t)|+|\th_{0,i}| \le |Y_i -\th_{0,i}|+2|\th_{0,i}|$, so that 
$$\E_{\th_{0,i}}\sup_{t\in T_n}|\hat\th_i(t)-\th_{0,i}|^4\1_{|Y_i-\th_{0,i}|>\z_{\tilde\t_n}}
\lesssim \int_{\z_{\tilde\t_n}}^\infty y^4\phi(y)\,dy+\th_{0,i}^4\lesssim \tilde\t_n\z_{\tilde\t_n}^3+\th_{0,i}^4.$$
Second we use that $|\hat\th_i(t)-\th_{0,i}|\lesssim \t |Y_i|e^{Y_i^2/2}+|\th_{0,i}|$, 
by Lemma~\ref{LemmaBoundsPostMeanVariance} (iii), so that 
\begin{align*}\E_{\th_{0,i}}\sup_{t\in T_n}|\hat\th_i(t)-\th_{0,i}|^4\1_{|Y_i-\th_{0,i}|\le \z_{\tilde\t_n}}
&\lesssim \tilde\t_n^4\int_{-\z_{\tilde\t_n}}^{\z_{\tilde\t_n}} (y+\th_{0,i})^4e^{2(y+\th_{0,i})^2}\phi(y)\,dy+\th_{0,i}^4\\
&\lesssim \tilde\t_n\z_{\tilde\t_n}^3 e^{4\z_{\tilde\t_n}|\th_{0,i}|+2\th_{0,i}^2}+\th_{0,i}^4.
\end{align*}
For $|\th_{0,i}|\lesssim \z_{\tilde\t_n}^{-1}$,  the exponential in the first term is bounded,
and the first term is bounded above by $\tilde\t_n\z_{\tilde\t_n}^3$. For $|\th_{0,i}|\gtrsim \z_{\tilde\t_n}^{-1}$,
but still $|\th_{0,i}|\le \z_{\tilde\t_n}/10$, the first term can be seen to be bounded above by
$\tilde\t_n\z_{\tilde\t_n}^3\tilde\t_n^{-21/25}$, which is bounded by $\th_{0,i}^4$ in that case.

Combining all the preceding  computations, we obtain:
\begin{align*}
\E_{\th_0} \sup_{t \in T_n} \|\th_0 - \hat\th(t)\|_4^4 \lesssim
\tilde p \z_{\tilde\t_n}^4 + n\tilde\t_n  \z_{\tilde\t_n}^3+ \sum_{i: |\th_{0,i}| <  \z_{\tilde\t_n}/10} \th_{0,i}^4.
\end{align*}
We see that this is of the desired order $n\tilde\t_n  \z_{\tilde\t_n}^3$ by bounding $\th_{0,i}^4$ by $ \z_{\tilde\t_n}^2 \th_{0,i}^2$,
and next applying the excessive-bias restriction.
\end{proof}

\begin{lemma}\label{lem:hyperLB}
If $\th_0$ satisfies the excessive-bias restriction \eqref{condition: EB} with $\tilde p\ge D\log n$
for a sufficiently large constant $D$, and the density of $\pi_n$ is  bounded away from zero on $[1/n,1]$,
then  there exist constants $C>0$ and $c_3>0$ such that
\begin{align*}
\Pi(\t: \t\leq C^{-1}\tilde\t_n \text{ or }\t\ge C\tilde\t_n \given Y^n)\lesssim e^{-c_3 \tilde{p}}.
\end{align*}
\end{lemma}

\begin{proof}
As seen in the proof of Lemma~\ref{thm: LB_tau} the function $\t\mapsto M_\t(Y^n)$ is
increasing for $\t\le c_5\tilde\t_n$. Inspection of the proof (see \eqref{EqDerivativeM})
shows that its derivative is bounded below by $c_6\tilde p/\t$ for $\t$ in the
interval $[c\tilde\t_n,2c\tilde\t_n]$, for $2c<c_5/2$ and suitably chosen $c_5$. 
This shows that $M_\t(Y^n)-M_{c\tilde\t_n}(Y^n)\ge c_8\tilde p$ in the interval $[2c\tilde\t_n,4c\tilde\t_n]$, whence
\begin{align*}
\Pi(\t: \t\leq c\tilde\t_n \given Y^n)
&\leq \frac{\int_{1/n}^{c \tilde\t_n}e^{M_{\t}(Y^n)}\pi(\t)\,d\t}{\int_{2c\tilde{\t}_n}^{4c\tilde{\t}_n}e^{M_{\t}(Y^n)}\pi(\t)\,d\t}
\lesssim \frac{e^{M_{c \tilde\t_n}(Y^n)}}{e^{M_{c \tilde\t_n}(Y^n)+c_8\tilde{p}}c\tilde{\t}_n}.
\end{align*}
This is bounded by $e^{-c_3\tilde p}$, by the assumption that $\tilde p\gtrsim \log n$.

The same bound on $\Pi(\t: \t\geq c\tilde\t_n \given Y^n)$ can be verified following the same 
reasoning, now using that $\t\mapsto M_\t(Y^n)$ is decreasing for $\t\ge c_6\tilde \t_n$ 
with derivative bounded above by $-c_9\tilde p/\t$ on an interval $[c\tilde\t_n/2,c\tilde\t_n]$
for $c/2>2c_6$ (see  \eqref{eq: UB_eb}).
\end{proof}

\section{\label{sec:proofs_model_selection}Proofs for the model selection results}
The following theorem is the non-adaptive version of Theorem \ref{thm: modelselect:adapt}.

\begin{theorem}\label{thm: modelselect:credible} Suppose that  $k_S>0$, $k_M<1$, $k_L>1$, and $f_\t\uparrow\infty$,  as $\t\ra0$. Then for $\t\ra 0$ and 
any sequence $\g_n\rightarrow c$ for some $0\le c\leq 1/2$, satisfying $\z_{\g_n}\ll \z_\t$, the following statements hold. 
\begin{enumerate}
\item[(i)] The false discovery rate of the credible intervals based model selection procedure is bounded from above by $\gamma_n$.

\item[(ii)] At least a $1-\gamma_n$ fraction of the signals in $\cL$ will be selected, with probability tending to one, i.e. 
\begin{align*}
P_{\th_0} \Bigl( \frac1{|\cL|}{|\{i\in \cL: 0\notin \hat{C}_{ni}(C,\t)\}|}\geq 1-\g_n \Bigr)&\rightarrow 1,
\end{align*}
for any $C>0$.

\item[(iii)] We will select at most a $\gamma_n$ fraction of the nonzero parameters $\theta_{0,i}\in \cS\cup \cM$ using the credible set method (with any blow up factor $C\geq 1$), with probability tending to one.
\end{enumerate}
\end{theorem}

 \begin{proof}
 For statement (ii), we note first that in view of the proof of  \eqref{eq: margcovL} by taking $L_n=z_{\alpha}^{-1}\zeta_{\gamma_n}+\delta$ (for some $\delta>0$) we have that with probability tending to one for $1-\gamma_n$ fraction of the indexes in $\cL$, denoted by $\tilde{\cL}=\tilde{\cL}(Y)$, the inequality $|\eps_i|\leq L_n z_{\alpha}+o(1)=\zeta_{\gamma_n/2}+\delta z_{\alpha}+o(1)=o(\zeta_{\tau})$ holds. Therefore, for all $i\in\tilde{\cL}$ we have $|Y_i|\geq|\theta_{0,i}|-|\eps_i|\geq (k_L+o(1))\zeta_{\tau}$. Hence in view of Lemma \ref{LemmaBoundsPostMeanVariance} (iv) and by Chebyshev's inequality we have that $\hat{r}_{i}(\alpha,\tau)\leq \sqrt{1+\zeta_{\tau}^{-2}}/{\sqrt{\alpha}}$. Furthermore, in view of \eqref{eq: hulp01} we have that again for all $i\in\tilde{\cL}$ the inequality $|\hat\theta_i(\tau)|\geq |\theta_{0,i}|-|\eps_i|-2\zeta_{\tau}^{-1}\geq (k_L+o(1))\z_\t\gg \hat{r}_{i}(\alpha,\tau)$ holds, finishing the proof of assertion (ii).

Next we deal with statements (i) and (iii) jointly. Without loss of generality let us assume that $Y_i\geq 0$. First note that $\Pi(\theta_i<0 \given Y_i,\tau)>\alpha$ implies that $0\in \hat{C}_{ni}(1,\alpha)$. Then, with $M$ sufficiently large as in the proof of \eqref{eq: margcovM},
\begin{align*}
\Pi(\theta_i<0| Y_i,\tau)
&\geq\int_{0}^{M}\Pi(\theta_i<0\given Y_i,\tau,\lambda_i)\pi(\lambda_i\given Y_i,\tau)d\lambda_i\\
&= \int_{0}^{M}\Pi\Big(\frac{\theta_i-\hat\theta_i(\t)}{\hat{r}_i(\tau,\lambda_i)}<-\frac{\hat\theta_i(\tau)}{\hat{r}_i(\tau,\lambda_i)}\Big| Y_i,\tau,\lambda_i\Big)\pi(\lambda_i\given Y_i,\tau)d\lambda_i\\
&=\Pi(\lambda_i\leq M\given Y_i,\tau ) \Phi\Big(-\frac{M\tau}{\sqrt{1+M^2\tau^2}}Y_i\Big).
\end{align*}
Next note that similarly to the proof of statement (ii) we have that with probability tending to one for at least $1-\gamma_n$ fraction of the zero parameters and for at least $1-\gamma_n$ fraction of the nonzero parameters belonging to $\cS\cup\cM$ the corresponding observations satisfy $|Y_i|\leq k_M \zeta_{\tau}+ |\eps_i|\leq k \zeta_{\tau}$, for some $k<1$.  Then in the view of the proof of \eqref{eq: margcovM} we know that if $|Y_i|\leq k \zeta_{\tau}$ the first term of the right hand side of the preceding display tends to one, while the second term tends to $1/2$, since $Y_iM\tau/(\sqrt{1+M^2\tau^2})\rightarrow 0$. This readily gives us assertion (iii). To prove the precise statement (i) about the FDR we note that the probability that $|Y_i|>k \zeta_{\tau}$ for some $k_M<k<1$ is bounded from above by $\gamma_n$ and therefore following the above argument the probability of $0\notin C_{ni}(1,\alpha)$ is also upper bounded by $\gamma_n$, finishing the proof of our statements.

\end{proof}

\begin{theorem}\label{thm: modelselect:threshold}
With the thresholding model selection method with probability tending to one at least a $1-2\gamma_n$ fraction of the zero signals will not be selected. Furthermore at least a $1-\gamma_n$ fraction of signals belonging to the set $\cL$ will be selected. Finally, any nonzero signal belonging to the sets $\cS$ or $\cM$ will not be selected with probability tending to one.
\end{theorem}

\begin{proof}
First we deal with the zeroes. In view of the proof of \eqref{eq: margcovS} we have that with probability tending to one at least a $1-\gamma_n$ fraction of the zero coefficients satisfy
$|Y_i|\leq k_S\tau+\zeta_{\gamma_n/2}+\delta=o(\zeta_{\tau})$ (for some arbitrary $\delta>0$). Therefore on the same index set we have $|\hat\theta_i(\tau)|\leq \tau(1\vee|Y_i|e^{Y_i^2/2})=o( \tau^{1-\delta})$, for arbitrary $\delta>0$. We finish the proof by noting that $P(|Y_i|\leq \tau^{1-\delta})\leq \tau^{1-\delta}$ and therefore at least $1-\gamma_n-O(\tau^{1-\delta})\geq1-2\gamma_n$ fraction of the zero coordinates have $\kappa_i\leq 1/2$ and won't be selected as nonzero.

Next consider the large signals belonging to the set $\cL$. In view of \eqref{eq: hulp01} and the proof of \eqref{eq: margcovS} we have that with probability tending to one at least a $1-\gamma_n$ fraction of the signal coefficients belonging to the set $\cL$ satisfy
\begin{align*}
\Big|\frac{\hat\theta_i(\tau)}{Y_i}\Big|\geq \frac{|\theta_{0,i}|-|\eps_i|-2\zeta_{\tau}^{-1}}{|\theta_{0,i}|+|\eps_i|}\geq \frac{k_L\zeta_{\tau}-\zeta_{\gamma_n/2}-\delta-2\zeta_{\tau}^{-1}}{k_L\zeta_{\tau}+\zeta_{\gamma_n/2}+\delta}\geq 1-2\frac{\zeta_{\gamma_n/2}+\delta+\zeta_{\tau}^{-1}}{k_L\zeta_{\tau}}.
\end{align*}
We conclude the proof by noting that by the assumption $\zeta_{\gamma_n/2}=o(\zeta_\tau)$ the left hand side of the preceding display takes the form $1+o(1)$.

Finally we show that for the nonzero signals belonging to the sets $\cS$ or $\cM$ with probability tending to one $\kappa_i(\tau)=o(1)$ holds, hence they will not be discovered. To show this note that in view of the proof of assertion \eqref{eq: margcovM} we have that with probability tending to one $|Y_i|\leq k\zeta_\tau$ for $|\theta_{0,i}|\leq k_M\zeta_\tau$ with $k_M<k<1$. Therefore following from Lemma \ref{LemmaBoundsPostMeanVariance} (iii) we have with probability tending to one that
\begin{align*}
\frac{|\hat\theta_i(\tau)|}{|Y_i|}\leq \frac{\tau|Y_i| e^{Y_i^2/2}}{|Y_i|}\leq \tau^{1-k^2}=o(1),
\end{align*}
concluding the proof.
\end{proof}

\textbf{Proof of Theorem \ref{thm: modelselect:adapt}}
\begin{proof} 
The proof for the MMLE empirical Bayes method closely resembles the proof of Theorem \ref{thm: modelselect:credible} with plugging in $\hat\tau_n$ for $\tau$. 

Hence it remains to deal with the hierarchical Bayes method. For the first and third statement note that in view of the proof of Theorem \ref{thm: modelselect:credible} we have that
\begin{align*}
\Pi(\theta_i<0 \mid Y_i)
&\geq \Pi(\tau\in[1/n,5t_n] \given Y_i) \sup_{\tau\in[1/n,5t_n]}\Pi(\lambda_i\leq M \given Y_i,\tau ) \Phi\Big(-\frac{M\tau}{\sqrt{1+M^2\tau^2}}Y_i\Big),
\end{align*}
with $M$ as specified in the proof of Theorem \ref{thm: modelselect:credible}.  On the right hand side the first two terms tend to one, while the third to $1/2$, and hence the left hand side is larger than $\alpha$ (for any $\alpha<1/2$), concluding the proof of the statement as in  the proof of Theorem \ref{thm: modelselect:credible}.

Finally we show that the model selection with the  hierarchical Bayes method selects with high probability a large fraction of the signals belonging to the set $\cL_a$. Again without loss of generality let us assume that $Y_i\geq 0$. First in view of the proof of  the hierarchical Bayes version of \eqref{eq: margcovL_a} we have that with probability tending to one, $|Y_i-\hat\theta_i|\rightarrow 0$. Therefore note that for arbitrary $L\geq 1$ with probability tending to one at least $1-\gamma_n$ fraction of the index set $\cL$ satisfies that the interval $[\hat\theta_i-Y_i/(2L),\hat\theta_i+Y_i/(2L)]$ contains every interval $[\hat\theta_i(\tau,\lambda_i)-z_{\alpha/4}r_i(\lambda_i,\tau),\hat\theta_i(\tau,\lambda_i)+z_{\alpha/4}r_i(\lambda_i,\tau)]$ for $\tau\in[1/n,5t_n]$ and $\lambda_i\geq g_{\tau}/\tau$ (for any $g_\tau \to \infty$ as $\t \to 0$), since similarly to the non-adaptive case $Y_i\geq (k_L+o(1))\zeta_{1/n}$, $r_i(\lambda_i,\tau)=O(1)$ and $Y_i(1+o(1))\leq\hat\theta_i(\tau,\lambda_i)\leq Y_i$. Therefore,
\begin{align*}
&\Pi(\theta_i:\, \theta_i\in [\hat\theta_i-Y_i/(2L),\hat\theta_i+Y_i/(2L)] \given Y_i)\\
&\quad\geq \int_{1/n}^{5t_n} \int_{g_{\tau}/\tau}^{\infty}  \Pi(\theta_i:\, \theta_i\in [\hat\theta_i(\tau,\lambda_i)-z_{\alpha/4}r_i(\lambda_i,\tau),\hat\theta_i(\tau,\lambda_i)+z_{\alpha/4}r_i(\lambda_i,\tau)] \given Y_i, \lambda_i, \t)\\
&\qquad\qquad\qquad\qquad\qquad\times \pi(\lambda_i \given Y_i,\tau)\pi(\tau \given Y_i)d\lambda_id\tau\\
&\quad\geq \int_{1/n}^{5t_n} \int_{g_{\tau}/\tau}^{\infty}  (1-\alpha/2) \pi(\lambda_i \given Y_i,\tau)\pi(\tau \given Y_i)d\lambda_id\tau=1-\alpha/2+o(1)>1-\alpha.
\end{align*}
So we can conclude that $0\notin \hat{C}_{ni}(L)$.

\end{proof}

\section{Proofs for the coverage of the credible balls in the deterministic case}

\subsection{Proof of Lemma \ref{lem: LBradius}}\label{Sec: lem: LBradius} 
\begin{proof}
The square radius $\hat r^2(\a,\t)$ is defined as the upper $\a$-quantile of the
variable $W=\|\th-\hat\th(\t)\|_2^2$ relative to its posterior distribution given $(Y^n,\t)$,
where $\hat\th(\t)=\E(\th\given Y^n,\t)$. By Chebyshev's inequality the variable
$W$ falls below $\E(W\given Y^n,\t)-c \sd(W\given Y^n,\t)$ with conditional probability given $(Y^n,\t)$ smaller than $1/c^2$  
for any given $c>0$. 
This implies that $\hat r^2(\a,\t)\ge \E(W\given Y^n,\t)-c \sd(W\given Y^n,\t)$ for  $c>0$ such that $1/c^2\le 1-\a$.
Thus it suffices to show that $\E(W\given Y^n,\t)\ge 0.501 n\t\z_\t$ and $\sd(W\given Y^n,\t)\ll n\t\z_\t$, with probability
tending to 1. Here the conditional expectations $\E(W\given Y^n,\t)$ and $\sd(W\given Y^n,\t)$ refer to the
posterior distribution of $\th$ given $(Y^n,\t)$ (where $W$ is a function of $\th$), which are 
functions of $Y^n$ that will be considered under the law of $Y^n$ following the true parameter. The variable 
$W=\sumin \bigl(\th_i-\hat\th_i(\t)\bigr)^2$ is lower bounded by 
the sum of squares $W_0$ of the variables $\th_i-\hat\th_i(\t)$ corresponding to the
indices with $\th_{0,i}=0$, which are $(n-p_n)\sim n$ of the coordinates.  The upper $\a$-quantile of $W$ is 
bigger than the upper $\a$-quantile of $W_0$, and hence
it suffices to derive a lower bound for the latter. For simplicity of notation we assume that all $n$ parameters $\th_{0,i}$ are zero
and write $W$ for $W_0$.

Because given $\t$ the coordinates are independent under the posterior distribution,
\begin{align*}
\E(W\given Y^n,\t)&=\sumin \E\bigl[\bigl(\th_i-\hat\th_i(\t)\bigr)^2\given Y_i,\t\bigr]=\sumin\var(\th_i\given Y_i,\t),\\
\var(W\given Y^n,\t)&=\sumin \var\bigl[\bigl(\th_i-\hat\th_i(\t)\bigr)^2\given Y_i,\t\bigr]
\le \sumin \E\bigl[\bigl(\th_i-\hat\th_i(\t)\bigr)^4\given Y_i,\t\bigr].
\end{align*}
Because the variables $Y_i$ are i.i.d.\ under the true distribution, Lemma~\ref{LemmaMomentsOfThetaGivenY} below gives that 
\begin{align*}
\E_0\E(W\given Y^n,\t)&\sim (2/\pi)^{3/2} n\t\z_\t,\\
\var_0\E(W\given Y^n,\t)&\lesssim n\t\z_\t,\\
\E_0\var(W\given Y^n,\t)&\lesssim n\t\z_\t^3.
\end{align*}
From the first two assertions and another application of Chebyshev's inequality, now with respect to the true law of 
$Y^n$, it follows that for any $c_n\ra\infty$ the probability of the event 
 $\E(W\given Y^n,\t)\le (2/\pi)^{3/2}\, n\t\z_\t-c_n\sqrt{n\t\z_\t}$
tends to zero. Since $\sqrt{n\t\z_\t}\ll n\t\z_\t$ (easily) under the assumption that $n\t/\z_\t\ra\infty$ and $(2/\pi)^{3/2}\approx 0.507$,
it follows that $\E(W\given Y^n,\t)$ is lower bounded by $0.5 n\t\z_\t$ with probability tending to one.
By Markov's inequality the probability of the event $\sd(W\given Y^n,\t)\ge c_nn\t\z_\t$ is bounded
{above} by $(c_nn\t\z_\t)^{-2}\E_0\var(W\given Y^n,\t)$, which is {further} bounded above by
$(c_nn\t\z_\t)^{-2}n\t\z_\t^3$, by the third assertion in the display. This tends to zero for some $c_n\ra 0$,
again by the assumption that $n\t/\z_\t\ra\infty$ (tightly this time).
\end{proof}

For the proof of  Lemma~\ref{lem: LBradius}, we have employed the lemma below,
which is based on the following observations.
The posterior density of $\th_i$ given $(Y_i=y,\t)$ is (for fixed $\t$) an exponential family with density
$$\th\mapsto \frac{\phi(y-\th)g_\t(\th)}{\psi_\t(y)}=c_\t(y) e^{\th y} g_\t(\th) e^{-\th^2/2},$$
where $g_\t$ is the prior density of $\th$, and $\psi_\t$ is the
Bayesian marginal density of $Y_i$, given in \eqref{Eqpsitau}, and the norming constant is given by
$$c_\t(y)=\frac{\phi(y)}{\psi_\t(y)}=\frac{\pi}{\t I_{-1/2}(y)},$$
for the function $I_{-1/2}(y)$ defined in \eqref{eq:def.Ik}. The cumulant moment generating function
$z\mapsto \log \E (e^{z\th_i}\given Y_i=y,\t)$ of the family is given by 
$z\mapsto\log \bigl(c_\t(y)/c_\t(y+z)\bigr)$, which is $z\mapsto \log I_{-1/2}(y+z)$ plus an additive constant
independent of $z$. We conclude that the first, second and fourth cumulants are given by
\begin{align}
\hat\th_i(\t)=\E(\th_i\given Y_i=y,\t)&=\frac{d}{dy}\log I_{-1/2}(y),\nonumber\\
\var(\th_i\given Y_i=y,\t)&=\frac{d^2}{dy^2}\log I_{-1/2}(y),
\label{EqPosteriorCumulants}\\
\E\bigl[\bigl(\th_i-\hat\th_i(\t)\bigr)^4\given Y_i=y,\t\bigr]-3 \var(\th_i\given Y_i=y,\t)^2
&=\frac{d^4}{dy^4}\log I_{-1/2}(y).\nonumber
\end{align}
The derivatives at the right side can be computed by repeatedly using the product and sum rule together with the
identity $I_k'(y)=y I_{k+1}(y)$, for $I_k$ as in \eqref{eq:def.Ik}.

\begin{lemma}
\label{LemmaMomentsOfThetaGivenY}
For $\E_0$ referring to the distribution of $Y_i\sim N(0,1)$, as $\t\ra 0$,
\begin{align*}
\frac{4C^{-1}\t\z_\t}{\pi\sqrt{2\pi}}\lesssim\E_0 \inf_{t\in[C^{-1}\t,C\t]}\var(\th_i\given Y_i,t)&\lesssim \E_0 \sup_{t\in[C^{-1}\t,C\t]}\var(\th_i\given Y_i,t)\lesssim  \frac{4C\t\z_\t}{\pi\sqrt{2\pi}},\\
\E_0 \sup_{t\in[C^{-1}\t,C\t]}\var(\th_i\given Y_i,t)^2&\lesssim\t\z_\t,\\
\E_0 \sup_{t\in[C^{-1}\t,C\t]}\E\bigl[\bigl(\th_i-\hat\th_i(t)\bigr)^4\given Y_i,t\bigr]&\lesssim \t\z_\t^3.
\end{align*}
\end{lemma}

\begin{proof}
The first assertion is already contained in \cite{vdPas}, but we give a new proof, which also
prepares for the proofs of the other assertions.

Since $(\log h)''=h''/h-(h'/h)^2$, for any function $h$,  and $I_{-1/2}'(y)=y I_{1/2}(y)$ and $I_{-1/2}''(y)=y^2I_{3/2}(y)+I_{1/2}(y)$, we
have by the formulas preceding the lemma,
\begin{equation}
\label{EqPosteriorVarianceExpression}
\var(\th_i\given Y_i=y,\t)=y^2\Bigl[\frac{I_{3/2}}{I_{-1/2}}-\Bigl(\frac{I_{1/2}}{I_{-1/2}}\Bigr)^2\Bigr](y)+\frac{I_{1/2}}{I_{-1/2}}(y).
\end{equation}
By Lemmas~\ref{LemmaI-1/2} and~\ref{LemmaI1/23/2} the right side is equivalent, uniformly in $y$, to
\begin{align*}
&y^2\Bigl[\frac{H_{3/2}(y)}{\pi/\t+H_{-1/2}(y)}\bigl(1+O(\sqrt\t)\bigr)
-\frac{H_{1/2}^2(y)}{\bigl(\pi/\t+H_{-1/2}(y)\bigr)^2}\bigl(1+O(\sqrt\t)\bigr)\Bigr]\\
&\qquad\qquad\qquad\qquad\qquad\qquad+\frac{H_{1/2}(y)}{\pi/\t+H_{-1/2}(y)}\bigl(1+O(\sqrt\t)\bigr),
\end{align*}
where $H_k(y)=(y^2/2)^{-k}\int_c^{y^2/2}v^{k-1}e^v\,dv$, with $c=0$ if $k>0$ and $c=1$ otherwise. Uniformly in $y\ge 1/\e_\t\ra\infty$, all functions
$H_k$ can be expanded as $H_k(y) =e^{y^2/2}/(y^2/2)(1+O(1/y^2))$, by Lemma~\ref{LemmaIncompleteGamma}.

Let $\k_\t$ be the solution to $e^{\k_\t^2/2}/(\k_\t^2/2)=1/\t$. For $y\ll \k_\t$ the factor $\pi/\t$ dominates the factor
$H_{-1/2}(y)$ and the preceding display can be approximated by
\begin{equation}
\label{EqHulpLemma41}
\frac{\t}{\pi}y^2 H_{3/2}(y)-\frac{\t^2}{\pi^2}y^2 H_{1/2}^2(y)+\frac{\t}{\pi}H_{1/2}(y).
\end{equation}
For instance, we can use this approximation on $[0,\z_\t]$, up to a uniform $1+o(1)$-term,
since $e^{-\z_\t^2/2}/\z_\t^2\ll 1/\t$.
A multiple of the preceding display, with the negative term removed,
is an upper bound for $\var(\th_i\given Y_i,\t)$ for any $y$; we use this for  $y\in [\z_\t,\k_\t]$. 
For $y{\geq}\k_\t$ the factor $H_{-1/2}(y)$ dominates $\pi/\t$ and the second to last display can be rewritten as,
for $\d_\t(y)=(\pi/\t )/H_{-1/2}(y)$,
\begin{align}
&y^2\Bigl[\frac{1+O(y^{-2})}{1+\d_\t(y)}(1+o(1))
-\frac{1+O(y^{-2})}{(1+\d_\t(y))^2}(1+o(1))\Bigr]
+\frac{1+O(y^{-2})}{1+\d_\t(y)}(1+o(1))\nonumber\\
&\qquad\qquad\qquad \qquad\qquad
=\frac{y^2\d_\t(y)}{(1+\d_\t(y))^2}+r_\t(y),\label{EqHulpLemma41Two}
\end{align}
where $r_\t(y)$ is uniformly bounded in $y\ge\k_\t$ as $\t\ra0$.

We can choose  $\e_{\t/C}\ra 0$ slow enough that 
$$\E_0 \sup_{t\in[C^{-1}\t,C\t]} \var(\th_i\given Y_i,t)\1_{0\le |Y_i|\le 1/\e_t}\lesssim C\t\int_ 0^{1/\e_{\t/C}} \bigl[y^2  H_{3/2}(y)+H_{1/2}(y)\bigr]\phi(y)\,dy$$
is of smaller order than $\t\z_\t$. Then this part of the expectation is negligible. For $1/\e_{t}\le |y| \le \z_{t}$,
we expand the functions $H_k$ in \eqref{EqHulpLemma41} and find that
\begin{align}
&\E_0 \sup_{t\in[C^{-1}\t,C\t]}  \var(\th_i\given Y_i, t)\1_{1/\e_t\le |Y_i|\le \z_\t}\nonumber\\
&\qquad\lesssim 2 \int_ {0}^{\infty}\sup_{t\in[C^{-1}\t,C\t]}\1_{1/\e_t\le |y|\le \z_\t}\bigl[(2\t/\pi) e^{y^2/2}\nonumber\\
&\qquad\qquad\qquad\qquad\qquad-(2\t/\pi)^2e^{y^2}/y^2+(2\t/\pi) e^{y^2/2}/y^2\bigr]\phi(y)\,dy.\label{eq: helpUBvar1}
\end{align}
We note that the integrand is non-negative and its derivative with respect to $t$ is also non-negative for every $1/\eps_{C\t}\leq y\leq \z_{\t/C}$ and $t\leq C\t$, i.e.
\begin{align*}
(2/\pi) e^{y^2/2}-(8\t/\pi^2)e^{y^2}/y^2+(2/\pi) e^{y^2/2}/y^2>0,
\end{align*}
since $e^{y^2/2}\leq C/\t$ and $y^2\rightarrow\infty$.
Therefore, we can further bound the right hand side of $\eqref{eq: helpUBvar1}$ as
\begin{align*}
2 \int_{\e_{C\t}^{-1}}^{\z_{\t/C}}\bigl[(2C\t/\pi) e^{y^2/2}
-(2C\t/\pi)^2e^{y^2}/y^2+(2C\t/\pi) e^{y^2/2}/y^2\bigr]\phi(y)\asymp{\frac{2C\sqrt{2}}{\pi\sqrt{\pi}}}\t\z_\t.
\end{align*}
Similar computations also lead to
\begin{align*}
\E_0 \inf_{t\in[C^{-1}\t,C\t]}  \var(\th_i\given Y_i, t)\1_{1/\e_t\le |Y_i|\le \z_\t}\gtrsim {\frac{2C\sqrt{2}}{\pi\sqrt{\pi}}}\t\z_\t.
\end{align*}

For $y\in [\z_\t,\k_\t]$ we again use \eqref{EqHulpLemma41}, but as an upper bound (without the negative term), and obtain
$$\E_0 \sup_{t\in[C^{-1}\t,C\t]} \var(\th_i\given Y_i,t)\1_{\z_\t\le |Y_i|\le \k_\t}\lesssim C\t \int_{\z_{C\t}}^{\k_{C^{-1}\t}} e^{y^2/2}\phi(y)\,dy\lesssim \t(\k_\t-\z_\t),$$
which is of lower order than the preceding display.
By \eqref{EqHulpLemma41Two} the contribution of $y\ge \k_\t$ is bounded by
\begin{align*}
\E_0 \sup_{t\in[C^{-1}\t,C\t]}\var(\th_i\given Y_i,t)\1_{\k_\t\le |Y_i|}&\lesssim \int_{\k_{\t/C}}^\infty [y^2\d_{\t/C}(y)+1]\phi(y)\,dy\\
&\lesssim \int_{\k_{\t/C}}^\infty[C\t^{-1}y^4e^{-y^2}+e^{-y^2/2}]\,dy\\
&\lesssim\t^{-1}\k_\t^3 e^{-\k_\t^2}+\k_\t^{-1}e^{-\k_\t^2/2}=O(\t/\k_\t).
\end{align*}
This concludes the proof of the first assertion.

For the proof of the second assertion we follow the same approach. We simply square the integrands in the
preceding bounds and obtain a negligible contribution from the interval $[0,1/\e_\t]$,
a contribution  bounded by $C^2\t^2\int_0^{\k_{\t/C}}e^{y^2}\,\phi(y)\,dy\lesssim \t^2e^{\k_{\t/C}^2/2}/\k_\t\lesssim \t\z_\t$ 
from the interval $[1/\e_\t,\k_\t]$ and a contribution no bigger than a multiple of 
$$\int_{\k_{C\t}}^\infty [y^4\d_{\t/C}^2(y)+1]\phi(y)\,dy\lesssim \int_{\k_{C\t}}^\infty [C^2\t^{-2}y^8e^{-3y^2/2}+e^{-y^2/2}]\,dy\lesssim \t\k_\t$$
from the interval $[\k_\t,\infty)$.

For the proof of the third assertion it suffices to bound the
fourth cumulant of $\th_i$ given $(Y_i,\t)$, in view of the second assertion. For any function $h$ we have
$$(\log h)''''=\frac{h''''}{h}-4\frac{h'''}{h}\frac{h'}h+12\frac{h''}h\Bigl(\frac{h'}h\Bigr)^2
-3\Bigl(\frac{h''}h\Bigr)^2-6\Bigl(\frac{h'}h\Bigr)^4.$$
Combined with the formulas for $I_{-1/2}'$ and $I_{-1/2}''$ given before as well as
$I_{-1/2}'''(y)=y^3I_{5/2}(y)+3yI_{3/2}(y)$ and $I_{-1/2}''''(y)=y^4I_{7/2}(y)+6y^2I_{5/2}(y)+3I_{3/2}(y)$, we find that
the fourth cumulant can be written in the form
\begin{align*}
&\frac{y^4I_{7/2}(y)+6y^2I_{5/2}(y)+3I_{3/2}(y)}{I_{-1/2}(y)}
-4\frac{y^3I_{5/2}(y)+3yI_{3/2}(y)}{I_{-1/2}(y)}\frac{y I_{1/2}(y)}{I_{-1/2}(y)}\\
&\qquad+12\frac{y^2I_{3/2}(y)+I_{1/2}(y)}{I_{-1/2}(y)}\Bigl(\frac {yI_{1/2}(y)}{I_{-1/2}(y)}\Bigr)^2
-3\Bigl(\frac{y^2I_{3/2}(y)+I_{1/2}(y)}{I_{-1/2}(y)}\Bigr)^2
-6 \Bigl(\frac {yI_{1/2}(y)}{I_{-1/2}(y)}\Bigr)^4.
\end{align*}
As before we expand these expressions with the help of Lemmas~\ref{LemmaI-1/2} and~\ref{LemmaI1/23/2},
and next integrate separately over $[0,1/\e_{C^{-1}\t}]$, $[1/\e_{C\t},{2}\k_{\t/C}]$, and $[{2}\k_{C\t},\infty)$. 
The first interval gives a negligible contribution. Following from the inequality  $I_{-1/2}(y)\geq I_k(y)$ for $k\geq -1/2$ and Lemma \ref{LemmaIncompleteGamma} one can obtain that the dominating term in the second interval is  $C\t y^2e^{y^/2}$. This leads to 
$$\int_0^{\k_{\t/C}}\sup_{t\in [C^{-1}\t_n,C\t_n]} t y^2e^{y^2/2}\phi(y)\,dy\lesssim C\t\int_0^{\k_{\t/C}} y^2\,dy\lesssim\t\z_\t^3$$
On the last interval
\begin{align*}
\int_{y\geq 2\k_{C\t}}y^4 e^{-y^2/2}\lesssim \k_{\t}^{11}\t^{4}=o(\t\z_\t^3).
\end{align*}

\end{proof}

\section{Proofs for the MMLE}

\subsection{Proof of Lemma \ref{thm: LB_tau}}\label{sec: thm: LB_tau}
Given $\th_0$ that satisfies the excessive-bias restriction, let $\tilde \z=A\sqrt{2\log (n/q)}$ and 
$\tilde p=\#\bigl(i: |\th_{0,i}|\ge\tilde \z\bigr)$, for $q$ as in \eqref{condition: EB}. Then 
$q/C_s\le \tilde p\le p =\#\bigl(i: \th_{0,i}\neq 0\bigr)\le p_n$, which is $o(n)$ by assumption, so
that $\tilde\z\ra\infty$, uniformly in $\th_0$.

Take any $\d_n\da 0$ and $A_1\in (A^{-1},1)$ and for given $\t$ split the set of indices $1,\ldots,n$ into
$I_{2}:=\{i: |Y_i|\ge A_1\tilde \z\}$,  $I_{0}=\{i\notin I_{2}: |\th_{0,i}|\le \d_n\z_\t^{-2}\}$,
 and $I_{1}=I_2^c\cap I_0^c$ the remaining indices. Since $|Y_i|\ge |\th_{0,i}|-|\varepsilon_i|$, we have that
$i\in I_{2}$ as soon as $|\th_{0,i}|\ge\tilde\z$ and $|\varepsilon_i|<(1-A_1)\tilde\z$. 
By definition there exist $\tilde p$ coordinates with $|\th_{0,i}|\ge\tilde\z$, and the
number of the corresponding variables $|\varepsilon_i|$ that fall below $(1-A_1)\tilde\z$ is a
binomial variable on $\tilde p$ trials and success probability tending to one, as $(1-A_1)\tilde\z\ra\infty$.
By Chebyshev's inequality it follows that with probability tending to one the cardinality 
of $I_{2}$ is at least $\tilde p/2$ (easily). By the excessive-bias restriction
$$\d_n^2\z_\t^{-4}\#\bigl(i: \d_n\z_\t^{-2}<|\th_{0,i}|< \tilde\z\bigr)
\le \sum_{i: |\th_{0,i}|<\tilde\z}\th_{0,i}^2 \lesssim q\log (n/q) \le C_s\tilde p\log (ne/(C_s\tilde p)).$$
This shows that the number of elements of $I_1$ with $|\th_{0,i}|<\tilde\z$ is bounded above by a multiple of
$\d_n^{-2}\z_\t^4\tilde p\log (ne/(C_s\tilde p))$. The number of $\th_{0,i}$ with 
$|\th_{0,i}|\ge\tilde\z$ is $\tilde p$ by definition, which is smaller than the preceding number
if $\d_n$ tends to zero sufficiently slowly and $\z_\t$ is bounded away from 0. In that case the
cardinality of $I_1$ is bounded above by $\d_n^{-2}\z_\t^4\tilde p\log (ne/(C_s\tilde p))$.
Since the indices of all zero coordinates are contained in $I_0$, 
the cardinality of $I_{1}$ is also trivially bounded from above by $p$.

By Lemma~\ref{lem: diff_like} the derivative of the log-likelihood can be written in the form
\begin{align}
\label{EqDerivativeM}
\frac{d}{d\t}M_\t(Y^n)&=\frac1\t\sum_{i\in I_{0}}m_\t(Y_i)+\frac1\t\sum_{i\in I_{1}}m_\t(Y_i)+\frac1\t\sum_{i \in I_{2}}m_\t(Y_i)\\
&\ge-\frac{C_e}{\z_\t}n-|I_1|+|I_2| C\Bigl(\frac1\t\wedge \frac{e^{A_1^2\tilde\z^2/2}}{A_1^2\tilde\z^2}\Bigr),
\nonumber
\end{align}
with probability tending to 1, uniformly in $\t\in [1/n,\eta_n]$ and any $\eta_n\da 0$,
for constants $C_e,C>0$. This follows 
by applying Proposition~\ref{prop:Em(Y)} together with Lemma~\ref{lem: uniform0} to the first sum,
Lemma~\ref{Lemma: tech1}(ii) and the monotonicity of $y\mapsto m_\t(y)$ to the second,
and Lemma~\ref{Lemma: tech1}(vi) to the third sum.
The right side is certainly nonnegative for $\t$ such that the third term dominates twice the absolute
values of both the first and second terms. Since $|I_2|\ge \tilde p/2$ and $\tilde p\gtrsim q=n e^{-\tilde\z^2A^{-2}/2}$,
it follows that the right side is nonnegative if
$$\frac{n}{\z_\t}\lesssim \frac{\tilde p}\t,\qquad \frac{n}{\z_\t}\lesssim \frac{n e^{\tilde\z^2(A_1^2-A^{-2})/2}}{\tilde\z^2},
\qquad |I_1|\lesssim \frac{\tilde p}\t,\qquad |I_1|\lesssim \frac{n e^{\tilde\z^2(A_1^2-A^{-2})/2}}{\tilde\z^2},$$
where the multiplicative constants must be sufficiently small.
The first inequality is satisfied for $\t\lesssim \t_n(\tilde p)$; the second is trivial 
since $A_1>A^{-1}$ and $\tilde\z\ra\infty$, and $\z_\t^{-1}\ra0$; 
the third can be reduced to
$\t\z_\t^4\lesssim \d_n^{2}/\log (ne/(C_s\tilde p))$, which is (easily) verified if 
$\t\lesssim \t_n(\tilde p)$ and $\d_n$ tends to zero sufficiently slowly;
the fourth is trivial since $|I_1|\le p\ll n$ and $A_1>A^{-1}$  and $\tilde\z\ra\infty$.
It follows that $\t\mapsto M_\t(Y^n)$ is increasing for $\t\lesssim \t_n(\tilde p)$ and hence $\mmle\gtrsim\t_n(\tilde p)$.

For the proof of the upper bound we use the same decomposition \eqref{EqDerivativeM},
but redefine the sets $I_k$ slightly, to 
$I_0=\{i: |\th_{0,i}|\leq \d_n/\z_\t^2\}$, $I_1=\{\d_n/\z_\t^2\leq |\th_{0,i}|\leq \z_\t/4\}$ 
and $I_2=I_0^c\cap I_1^c$. Reasoning as before, using the excessive-bias restriction, we see that the cardinalities
of the sets $I_1$ and $I_2$ are bounded by multiples of $\d_n^{-2}\z_{\t}^4 \tilde{p}\log (ne/(C_s\tilde p))$ 
and $\z_{\t}^{-2}\tilde{p}\log (ne/(C_s\tilde p))+ \tilde p$, respectively. 
By the decomposition \eqref{EqDerivativeM} we obtain, 
\begin{equation}\frac{d}{d\t}M_\t(Y^n)\lesssim 
-\frac{C_e}{\z_\t}{(n-p)}+\frac{|I_1|\t^{1/16}}{\t\z_\t}+o\Bigl(\frac{|I_1|\t^{1/32}}{\t\z_\t}\Bigr)
+\frac1\t |I_2| C_u,
\label{eq: UB_eb}
\end{equation}
with probability tending to 1, uniformly in $\t\in [1/n,\eta_n]$ and any $\eta_n\da 0$.
Here the upper bounds on the sums over the coordinates in $I_0$ and $I_1$ follow
with the help of the first and second parts of Proposition~\ref{prop:Em(Y)} and Lemma~\ref{lem: uniform0},
and the bound on the sum over the coordinates in $I_2$ follows from Lemma~\ref{Lemma: tech1}(i).
The right side is certainly negative for $\t$ such that $2\t^{-1}|I_2|C_u\le C_e(n-p)/\z_\t$ and 
$|I_1|\t^{1/32}/{\z_\t}\le C_u |I_2|$. The first reduces to
$\t\z_\t\gtrsim(\tilde{p}/n)\log (ne/(C_s\tilde p))$ and $\t/\z_\t\gtrsim \tilde p/n$ 
and hence is true for $\t\gtrsim \t_n(\tilde p)$;
the second reduces to $\t^{1/32}\z_\t^5\lesssim \d_n^2$ and is true as well provided $\d_n\da0$ slowly. 
Since we may assume that
$\mmle\in [1/n,\eta_n]$ for some $\eta_n\da0$ by 
Theorem 3.1 in  \cite{contractionpaper}, it follows in that case that $\mmle\lesssim \t_n(\tilde p)$.


\section{Proof for the adaptive credible sets}

\subsection{Proof of Theorem \ref{thm:coverage}}\label{Sec:coverage_adaptive}
To simplify notation set $T_n=[C^{-1}\tilde\t_n,C\tilde\t_n]$, where $\tilde\t_n = \t_n(\tilde p_n)$. 

First we deal with the empirical Bayes credible sets.
Since $\hat\t_n\in T_n$ with probability tending to one by Condition~\ref{cond.adapt.coverage.eb},
\begin{align*}
P_{\th_0}\big(\th_0\notin\hat{C}_n(\hat\t_n,L)\big)
&= P_{\th_0}\big( \|\th_0-\hat\th(\hat\t_n)\|_2> L\hat r(\a,\hat\t_n) \big)\\
&\leq P_{\th_0}\big( \sup_{\t\in T_n}\|\th_0-\hat\th(\t)\|_2> L\inf_{\t\in T_n} \hat r(\a,\t) \big)+o(1).
\end{align*}
By Lemma~\ref{lem: LBradiusAdapt} $\inf_{\t\in T_n} \hat r(\a,\t)\gtrsim\sqrt{n\tilde\t_n\z_{\tilde\t_n}}$, with probability
tending to one.
Therefore it suffices  to show that $\sup_{\t\in T_n}\|\th_0-\hat\th(\t)\|_2=O_P(\sqrt{n\tilde\t_n\z_{\tilde\t_n}})$.
We show this by bounding the second moment of this variable.

We split the sum in $\|\hat\th(\t)-\th_0\|_2^2=\sum_i(\hat\th_i(\t)-\th_{0,i})^2$ in two parts,
according to the values of $\th_{0,i}$. Set $\tilde\z=A\sqrt{2\log(n/q)}$, for $q$ as in \eqref{condition: EB}.

If $|\th_{0,i}|\geq \z_{\tilde\t_n}/5$, then
we first use Lemma~\ref{LemmaBoundsPostMeanVariance}(ii) together
with the triangle inequality to see that $|\hat\th_i(\t)-\th_{0,i}|\lesssim \z_\t+ |Y_i-\th_{0,i}|$, as $\t\ra0$, whence 
$$\E_{\th_{0,i}}\sup_{\t\in T_n}(\th_{0,i}-\hat\th_{i}(\t))^2\lesssim \sup_{\t\in T_n}\z_\t^2+\var_{\th_{0,i}}Y_i\lesssim \z_{\tilde\t_n}^2.$$
By the excessive-bias restriction
$$\frac{\z_{\tilde\t_n}^2}{25}\,\Bigl|\Bigl\{i: \frac{\z_{\tilde\t_n}}5<|\th_{0,i}|< \tilde\z\Bigr\}\Bigr|
\le \sum_{i: |\th_{0,i}|\le \tilde\z}\th_{0,i}^2 \lesssim q\log (n/q)\lesssim\tilde p\log (ne/(C_s\tilde p)).$$
Since $\log(ne/(C_s\tilde p))/\z_{\tilde\t_n}^2\ra 1$, it follows that there are fewer than a constant times $\tilde p$ 
parameters with $|\th_{0,i}|\geq \z_{\tilde\t_n}/5$ and hence their total contribution to the sum is
bounded by $\tilde p \z_{\tilde\t_n}^2$.

For parameters such that $|\th_{0,i}|\leq \z_{\tilde\t_n}/5$ we use the triangle inequality
$|\hat\th_i(\t)-\th_{0,i}|\le |\hat\th_i(\t)|+|\th_{0,i}|$,  and next further bound $|\hat\th_i(\t)|$ by $\t |Y_i|e^{Y_i^2/2}$
in case $|Y_i-\th_{0,i}|\le \z_{\tilde\t_n}$, which is valid in view of Lemma~\ref{LemmaBoundsPostMeanVariance} (iii),
and further bound $|\hat\th_i(\t)|\le |Y_i|$ by $|Y_i-\th_{0,i}|+|\th_{0,i}|$, otherwise. This gives
\begin{align*}
\E_{\th_{0,i}}\!\sup_{\t\in T_n}|\th_{0,i}-\hat\th_i(\t)|^2
&\lesssim \E_{\th_{0,i}}\tilde\t_n^2|Y_i|^2e^{Y_i^2}\1_{|Y_i-\th_{0,i}|\le \z_{\tilde\t_n}}\\
&\qquad\qquad\qquad+\E_{\th_{0,i}}|Y_i-\th_{0,i}|^2\1_{|Y_i-\th_{0,i}|> \z_{\tilde\t_n}}+\th_{0,i}^2.
\end{align*}
The second expectation on the right is bounded above by $\tilde\t_n\z_{\tilde\t_n} $. 
The first expectation on the right is equal to 
$\t^2\int_{-\z_\t}^{\z_\t}(y+\th)^2 e^{(y+\th)^2}\phi(y)\,dy\lesssim \t^2\z_\t^2 \int_0^{\z_\t}e^{\th^2+2y|\th|} e^{y^2/2}\,dy$, 
for $\t=\tilde\t_n$ and $\th=\th_{0,i}$.
For $|\th|\lesssim \z_\t^{-1}$, the exponential factor $e^{\th^2+2y|\th|}$ is uniformly bounded, and the
whole expression is bounded by a multiple of $\t^2\z_\t^2\int_0^{\z_\t}e^{y^2/2}\,dy\lesssim\t\z_\t$.
For $|\th|\gtrsim \z_\t^{-1}$, but $|\th|\le \z_\t/5$, the exponential factor is bounded above by $e^{\z_\t^2/25+2\z_\t^2/5}=\t^{-22/25}$
and the whole expression is bounded above by $\t^{3/22}\z_\t\lesssim \th^2$. Thus in both cases
the first equation is bounded above by a multiple of  $\tilde\t_n\z_{\tilde\t_n}+\th_{0,i}^2$.

Combining the above two cases we find 
\begin{align*}
\E_{\th_0}\sup_{\t\in T_n}\|\hat\th(\t)-\th_{0}\|_2^2\lesssim 
\tilde{p}\z_{\tilde\t_n}^2 + n\tilde\t_n \z_{\tilde\t_n}+\sum_{i: |\th_{0,i}|<\z_{\tilde\t_n}/5}\th_{0,i}^2.
\end{align*}
Since $\z_{\tilde\t_n}^2\sim \log (n/\tilde p)\le \log (ne/(C_sq))\sim \log (n/q)$ and $1/5<1$,
the last term is bounded above by a multiple of $q\log (n/q)\lesssim \tilde p\log (n/\tilde p)$ 
by the excessive-bias restriction, whence the
whole expression is bounded above $n\tilde\t_n \z_{\tilde\t_n}\asymp \tilde p\log(n/\tilde p)$.
This concludes the proof of the coverage of the empirical Bayes credible balls.

The proof of their rate-adaptive size follows along the same lines.

Next we deal with the hierarchical Bayes credible sets.  By Lemma~\ref{Lem:hyperLB} and the triangle inequality
\begin{align*}
P_{\th_0}\bigl(\th_0\notin\hat{C}_n(L)\bigr)
&\le P_{\th_0} \bigl(\|\th_0-\hat\th\|_2>L \hat r(\a)\bigr)\\
&\leq P_{\th_0} \big(\|\th_0-\hat\th(\tilde\t_n)\|_2+ \|\hat\th-\hat\th(\tilde\t_n)\|_2> L A \sqrt{n\z_{\tilde\t_n}\tilde\t_n} \big)+o(1).
\end{align*}
The proof for the empirical Bayes set as just given shows that
$\|\th_0-\hat\th(\tilde\t_n)\|_2=O_P(\sqrt{n\tilde\t_n\z_{\tilde\t_n}})$. Therefore,  it is sufficient to show that
$\|\hat\th-\hat\th(\t_n)\|_2=O_P(\sqrt{n\tilde\t_n\z_{\tilde\t_n}})$.
Since $\hat\th=\int \hat\th(\t)\,\pi(\t\given Y_i)\,d\t$, Jensen's inequality gives
\begin{align}
\|\hat\th-\hat\th(\tilde\t_n)\|_2^2&\leq \int_{1/n}^{1} \|\hat\th(\t)-\hat\th(\tilde\t_n)\|_2^2\pi(\t \given Y^n)\,d\t\nonumber\\
&\leq \sup_{\t\in T_n}\|\hat\th(\t)-\hat\th(\tilde\t_n)\|_2^2+
\sup_{\t\in[1/n,1]}\|\hat\th(\t)-\hat\th(\tilde\t_n)\|_2^2\,\Pi(\t\notin T_n \given Y^n).\label{eq: hulp01b}
\end{align}
The first term on the right hand side is bounded from above by
$4\sup_{\t\in T_n}\|\hat\th(\t)-\th_0\|_2^2$, and was already seen to be $O_P(n\tilde\t_n\z_{\tilde\t_n})$.
By the triangle inequality and Lemma~\ref{LemmaBoundsPostMeanVariance} (i)+(ii)
the second supremum on the right hand side is bounded by 
\begin{align*}
4\sup_{\t\in[1/n,1]}\|\hat\th(\t)-Y^{n}\|_2^2
&\leq 4n\sup_{\t\in[1/n,1]} \z_{\t}^2\lesssim n\log n.
\end{align*}
By Lemma \ref{lem:hyperLB} 
we can choose the constant $C$ in the definition of $T_n$ such that $\Pi(\t\notin T_n \given Y^n )\leq e^{-c_3 \tilde{p}}$,
for a constant $c_3>0$. For $\tilde{p}\ge (2/c_3)\log n$ the probability $\Pi(\t\notin T_n \given Y^n )$
is of the order $n^{-2}$, and the second term on the right hand side of \eqref{eq: hulp01b} is negligible.

\section{Lemmas supporting the coverage results}

\begin{lemma}\label{lem: uniform1}
For $\t\geq 1/n$ and $Y^n\sim N_n(0,I_n)$, set $H_n(\t)=\E(\|\th-\hat\th(\t)\|_2^2 \given \t,Y^n )=\sum_{i=1}^n\var(\th\given \t,Y_i)$.
Then for any $C>0$, as $\t\ra 0$,
\begin{align*}
\sup_{t\in[C^{-1}\t,C\t]}\frac{1}{n\t\z_{\t}}\Big|H_n(t)-\E_{0}H_n(t) \Big|\stackrel{P}{\rightarrow}0.
\end{align*}
\end{lemma}

\begin{proof}
Set $T=[C^{-1}\t,C\t]$. In view of Corollary 2.2.5 of \cite{vdVW}
(applied with $\psi(x)=x^2$) it is sufficient to show that $\Var_{0} \big(H_n(t)/n\t\z_{\t}\big)\rightarrow 0$ for some
$t\in T$, and
\begin{align}\label{eq: entropy2}
\int_0^{\diamm}\sqrt{N(\eps, T,d_n)}\,d\eps=o(1),
\end{align}
where $d_n$ is the intrinsic metric defined by its square
$d_n^2(\t_1,\t_2)=(n\t\z_{\t})^{-2}\Var_{0}\bigl(H_n(\t_1)-H_n(\t_2)\bigr)$, 
$\diamm$ is the diameter of the interval $T$ with respect to the metric $d_n$, and
$N(\eps,A,d_n)$ is the covering number of the set $A$ with $\eps$ radius balls with respect to the
metric $d_n$. 

In view of Lemma~\ref{lem: LBradius},  
$$\Var_{0} \big(H_n(\t)/(n\t\z_{\t})\big) \lesssim (n\t\z_\t)^{-1}\ra0.$$
Combining this with the triangle inequality and the fact that $\t\z_\t\asymp t\z_t$ for every $t\in T$,
we also see that the diameter $\diamm$ is bounded from above by a multiple of $1/(n\t\z_\t)^{1/2}$.

Since $d_n(\t_1,\t_2)\lesssim  |\t_2-\t_1|\t^{-3/2}n^{-1/2}$, by Lemma~\ref{lem: metric2},
the covering number of the interval
$T$ with balls of radius $\eps$ is bounded by a multiple of $\eps^{-1}/(n\t)^{1/2} $ . 
Hence the integral of the entropy is bounded by
$$\int_0^{\diamm}\sqrt{N(\eps, T,d_n)}\,d\eps\lesssim (n\t)^{-1/4} \int_0^{1/(n\t\z_\t)^{1/2}}\eps^{-1/2}\,d\eps
\lesssim (n\t)^{-1/2}\z_{\t}^{-1/4}\rightarrow 0. $$
\end{proof}

\begin{lemma}\label{lem: metric2}
For  $Y_i\sim N(0,1)$, and $1/n\leq \t_1<\t_2\leq 1/2$, 
\begin{align*}
\E_0 \left(\Var(\theta_i\given Y_i,\t_1)-\Var(\theta_i\given Y_i,\t_2)\right)^2
\lesssim (\t_2-\t_1)^2 \t_1^{-1}\z_{\t_1}^2.
\end{align*}
\end{lemma}

\begin{proof}
Differentiating the left side of \eqref{EqPosteriorVarianceExpression} with respect to $\t$ and
applying Lemma~\ref{Lem: Diff} we see that the left side of the lemma is bounded above
by $|\t_1-\t_2|^2$ times 
\begin{align*}
& \sup_{\t\in[\t_1,\t_2]}\!\!\!\E_0\Big[Y_i^2\frac{\dot I_{3/2}}{I_{-1/2}}-2Y_i^2 \frac{\dot I_{1/2} I_{1/2}}{I_{-1/2}^2}\! + \frac{\dot I_{1/2}}{I_{-1/2}} 
\!-  \frac{\dot I_{-1/2}}{I_{-1/2}} \Bigl[ Y_i^2  \frac{I_{3/2}}{I_{-1/2}}-2Y_i^2  \frac{I_{1/2}^2}{I_{-1/2}^2} +  \frac{I_{1/2}}{I_{-1/2}}   \Bigr]  \Bigr]^2\\
&=\!\sup_{\t\in[\t_1,\t_2]}\!\!\!\E_0\Big[Y_i^2\frac{\dot I_{3/2}}{I_{-1/2}}\!-2Y_i^2 \frac{\dot I_{1/2} I_{1/2}}{I_{-1/2}^2} \!+ \frac{\dot I_{1/2}}{I_{-1/2}} 
\!+ \frac{\dot I_{-1/2}}{I_{-1/2}}  \Bigl[m_\t(Y_i)+Y_i^2 \bigl[\frac{2I_{1/2}^2}{I_{-1/2}^2}-\frac{I_{1/2}}{I_{-1/2}} \bigr] \Bigr]\Bigr]^2,
 \end{align*}
in view of \eqref{def: m}. Here $\dot I_k$ denotes the partial derivative of $I_k$ with respect to $\t$,
and the argument $Y_i$ of $I_k$ and $\dot I_k$ has been omitted.
In view of Lemma~\ref{lem: metric} and the following properties shown in its proof:
\begin{align}
0&\leq\frac{J_{3/2}-J_{5/2}}{I_{-1/2}}\le \frac{J_{1/2}-J_{3/2}}{I_{-1/2}}\le \frac{J_{1/2}}{I_{-1/2}}\leq 4,\nonumber\\
0&\leq\frac{J_{-1/2}-J_{1/2}}{I_{-1/2}}\leq \frac{J_{-1/2}}{I_{-1/2}}\leq \frac1{\t^{2}},
\end{align} 
the right hand side of the preceding display is further bounded above by a multiple of
\begin{align*}
& \sup_{\t\in[\t_1,\t_2]}\Bigl[\t^2+ \t^2\E_0Y_i^4 + \t^{-2}\E_0m_\t(Y_i)^2+  \t^{-2}\E_0Y_i^4\frac{I_{1/2}^2}{I_{-1/2}^2} \Bigr].
 \end{align*}
The first two terms inside the square brackets are uniformly bounded, 
the third one is of order $o(\t^{-1}\z_{\t}^{-2})$ as $\t\ra0$ in view of Lemma~\ref{lem: m2}, 
and is uniformly bounded, by Lemma~\ref{Lem: asymp_m}.
It remains to deal with the last term. By Lemmas~\ref{LemmaI-1/2} and~\ref{LemmaI1/23/2} 
the quotient $I_{1/2}/I_{-1/2}$ is bounded by a constant for $|y|\ge \k_\t$,
and by a multiple of $\t e^{y^2/2}/y^2$, otherwise. Therefore, 
\begin{align*}
\int_{|y|\ge \k_\t}  y^4\frac{I_{1/2}^2}{I_{-1/2}^2}\phi(y)\,dy&\lesssim\int_{\k_\t}^\infty y^4 e^{-y^2/2}\,dy\lesssim e^{-\k_\t^2/2}\k_\t^{3}
\lesssim \t\z_\t,\\
\int_{|y| \leq \k_\t}  y^4\frac{I_{1/2}^2}{I_{-1/2}^2}\phi(y)\,dy&\lesssim 
\t^2 \int_0^{\k_\t}e^{y^2/2}\,dy\lesssim \t^2 \k_\t^{-1}e^{\k_\t^2/2}\lesssim \t\z_\t.
\end{align*}
This concludes the proof.
\end{proof}

\section{Results proven in the companion paper}
Some technical lemmas from \cite{contractionpaper} reappear in the proofs of the coverage results. We state the lemmas here for completeness. All proofs are given in \cite{contractionpaper}.

\begin{lemma}\label{lem:HBhyper}
If Conditions~\ref{cond.hyper.3} and~\ref{cond.hyper.1} hold, then 
$\inf_{\th_0\in\ell_0[p_n]}\E_{\th_0}\Pi(\t: \t\le 5 t_n \given Y^n)\ra 1.$ Furthermore, if only Conditions~\ref{cond.hyper.3} and~\ref{cond.hyper.1a} hold, then 
 the similar assertion  is true but with $5t_n$ replaced by $(\log n)t_n$.
\end{lemma}

\begin{lemma}
\label{LemmaBoundsPostMeanVariance}
For $A>1$ and every $y\in \RR$,
\begin{itemize}
\item[(i)] $|\E(\th_i\given Y_i=y,\t)-y|\le 2 \z_\t^{-1}$, for $|y|\ge A\z_\t$, as $\t\ra0$.
\item[(ii)] $|\E(\th_i\given Y_i=y,\t)|\le |y|$.
\item[(iii)] $|\E(\th_i\given Y_i=y,\t)|\le \t |y| e^{y^2/2}$, as $\t\ra0$.
\item[(iv)] $|\var(\th_i\given Y_i=y,\t)- 1|\le \z_\t^{-2}$, for $|y|\ge A\z_\t$, as $\t\ra0$.
\item[(v)] $\var(\th_i\given Y_i=y,\t)\le 1+y^2$,
\item[(vi)] $\var (\th_i\given Y_i=y,\t)\lesssim \t e^{y^2/2}(y^{-2}\wedge 1)$, as $\t\ra0$.
\item[(vii)] $|\E(\th_i\given Y_i=y,\t)-y|\lesssim (\log |y|)/|y|$, uniformly in $\t\ge \t_0>0$ and $|y|\ra\infty$.
\end{itemize}
\end{lemma}

\begin{lemma}\label{Lem: DominatingDensity}
If $f_1,f_2: [0,\infty)\ra[0,\infty)$ are probability densities such that  $f_2/f_1$ is monotonely increasing,
then, for any monotonely increasing function $h$,
$\E_{f_1} h(X)\leq \E_{f_2} h(X).$
\end{lemma}

\begin{lemma}\label{lem: diff_like}
The derivative of the log-likelihood function takes the form 
$\frac{d}{d\t}M_\t(y^n)=\frac1\t\sum_{j=1}^{n}m_\t(y_j).$
\end{lemma}

\begin{proposition}\label{prop:Em(Y)}
Let $Y\sim N(\th,1)$. Then  $\sup_{\t\in [\eps,1]}\E_0 m_\t(Y)<0$ for every $\eps>0$, and as $\t\ra0$, $\E_\th m_\t(Y)= -\frac{2^{3/2}}{\pi^{3/2}}\,\frac{\t}{\z_\t}\bigl(1+o(1)\bigr)$ if $  |\th|=o(\z_{\t}^{-2})$, and $\E_\th m_\t(Y)= o (\t^{1/16}\z_{\t}^{-1})$ if $|\th|\leq \z_\t/4$.
\end{proposition}

\begin{lemma}\label{lem: uniform0}
For any $\e_\t\downarrow0$ and uniformly in $I_0\subseteq\{i: |\th_{0,i}|\leq \z_\t^{-1}\}$ with $|I_0|\gtrsim n$,
$
\sup_{1/n\leq \t\leq \e_\t}\frac{1}{|I_0|}\Big|\sum_{i\in I_0}m_\t(Y_i)\frac{\z_\t}{\t}
-\sum_{i\in I_0}\E_{\th_0} m_\t(Y_i)\frac{\z_\t}{\t}\Big|\stackrel{P_{\th_0}}{\rightarrow}0.
$
Similarly, uniformly in $I_1\subseteq\{i: |\th_{0,i}|\leq \z_\t/4\}$,
$
\sup_{1/n\leq \t\leq \e_\t}\frac{1}{|I_1|}\Big|\sum_{i\in I_1}m_\t(Y_i)\frac{\z_\t}{\t^{1/32}}
-\sum_{i\in I_1}\E_{\th_0} m_\t(Y_i)\frac{\z_\t}{\t^{1/32}}\Big|\stackrel{P_{\th_0}}{\rightarrow}0.
$
\end{lemma}

\begin{lemma}\label{lem: metric}
Let $Y\sim N(\th,1)$. For $|\th|\lesssim \z_\t^{-1}$ and $0<\t_1<\t_2\le 1/2$. Then 

$\E_\th\left(\frac{\z_{\t_1}}{\t_1}m_{\t_1}(Y)-\frac{\z_{\t_2}}{\t_2}m_{\t_2}(Y)\right)^2
\lesssim (\t_2-\t_1)^2 \t_1^{-3}.
$
Furthermore, for $|\th|\leq \z_\t/4$, and $\eps=1/16$ and $0<\t_1<\t_2\le 1/2$,
$
\E_\th\left(\frac{\z_{\t_1}}{\t_1^{\eps}}m_{\t_1}(Y)-\frac{\z_{\t_2}}{\t_2^{\eps}}m_{\t_2}(Y)\right)^2
\lesssim (\t_2	-\t_1)^2 \t_1^{-2-\eps}.
$
\end{lemma}

\begin{lemma}\label{lem: m2}
Let $Y\sim N(\th,1)$. Then,  as $\t\rightarrow 0$, $\E_\th m_\t^2(Y)=o(\t \z_\t^{-2})$ if $|\th|\lesssim \z_\t^{-1}$, and $\E_\th m_\t^2(Y)=o (\t^{1/16} \z_{\t}^{-2})$ if $|\th|\leq \z_\t/4.$
\end{lemma}

\begin{lemma}
\label{lem: E_nonzero_m(Y)}
\label{Lemma: tech1}
\label{Lem: asymp_m}
The function $y\mapsto m_\t(y)$ is symmetric about 0 and nondecreasing on $[0,\infty)$ with 
\begin{itemize}
\item[(i)] $-1\le  m_\t(y)\le C_u$, for all $y\in\RR$ and all $\t\in[0,1]$, and some $C_u<\infty$.
\item[(ii)] $m_\t(0)=-(2\t/\pi)(1+o(1))$, as $\t\ra0$.
\item[(iii)] $m_\t(\z_\t)=2/(\pi\z_\t^2)(1+o(1))$, as $\t\ra0$.
\item[(iv)] $m_\t(\k_\t)=1/(\pi+1)/(1+o(1))$, as $\t\ra0$.
\item[(v)] $\sup_{y\ge A\z_\t}|m_\t(y)-1|=O(\z_\t^{-2})$, as $\t\ra0$, for every $A>1$.
\item[(vi)] $m_\t(y)\sim\t e^{y^2/2}/(\pi y^2/2+\t e^{y^2/2})$, as $\t\ra0$, uniformly in $|y|\ge1/\e_\t$, for any $\e_\t\da0$.
\item[(vii)] $|m_\t(y)|\lesssim\t e^{y^2/2}(y^{-2}\wedge 1)$, as $\t\ra0$, for every $y$.
\end{itemize}
\end{lemma}

\begin{lemma}
\label{LemmaIncompleteGamma}
For any $k$, as $y\ra\infty$,
$\int_1^y u^ke^u\,du=y^ke^y\bigl(1-k/y+O(1/y^2)\bigr).$
Consequently, as $y\ra\infty$,
$\int_1^y{u^{k}}{e^u}\,du-\frac 1y\int_1^y u^{k+1} e^u\,du=y^{k-1}e^y\bigl(1+O(1/y)\bigr).$
\end{lemma}

\begin{lemma}
\label{LemmaI-1/2}
There exist functions $R_\t$ with $\sup_{y}|R_\t(y)|=O(\sqrt{\t})$ as $\t\downarrow0$,  such that
$
I_{-1/2}(y)=\Bigl(\frac{\pi}\t+\sqrt{y^2/2}\int_1^{y^2/2} \frac1{v^{3/2}}e^v\,dv\Bigr)\bigl(1+R_\t(y)\bigr).
$
Furthermore, given $\e_\t\ra 0$ there exist functions $S_\t$ with $\sup_{y\ge 1/\e_\t}|S_\t(y)|=O(\sqrt \t+\e_\t^2)$, 
such that, as $\t\da0$, 
$
I_{-1/2}(y)=\Bigl(\frac{\pi}\t+\frac{e^{y^2/2}}{y^2/2}\Bigr)\bigl(1+S_\t(y)\bigr).
$
\end{lemma}

\begin{lemma}
\label{LemmaI1/23/2}
For $k>0$, there exist functions $R_{\t,k}$ with $\sup_{y}|R_{\t,k}(y)|=O(\t^{2k/(k+1)})$,
and for given $\e_\t\ra 0$ functions $S_{\t,k}$ with $\sup_{y\ge 1/\e_\t}|S_{\t,k}(y)|=O(\t^{2k/(2k+1)}+\e_\t^2)$, such that,
as $\t\da0$, $I_{k}(y)=\frac1{(y^2/2)^k}\int_0^{y^2/2}v^{k-1}e^v\,dv\bigl(1+R_{\t,k}(y)\bigr)\lesssim\bigl(1\wedge y^{-2}\bigr)e^{y^2/2},$ and 
$I_{k}(y)=\frac{e^{y^2/2}}{y^2/2}\bigl(1+S_{\t,k}(y)\bigr).$ 

There also exist functions $\bar R_{\t}$ with $\sup_{y}|\bar R_{\t}(y)|=O(\t^{1/2})$ and  
$\bar S_{\t}$ with $\sup_{y\ge 1/\e_\t} \allowbreak |\bar S_{\t}(y)|=O(\sqrt\t+\e_\t^2)$, such that, as $\t\da0$ and $\e_\t\ra0$, $I_{1/2}(y)-I_{3/2}(y)=\frac1{\sqrt{y^2/2}}\int_0^{y^2/2}\frac{1-2v/y^2}{\sqrt v}e^v\,dv\bigl(1+\bar R_{\t}(y)\bigr)
\lesssim (1\wedge y^{-4})e^{y^2/2},$ and
$I_{1/2}(y)-I_{3/2}(y)=\frac{e^{y^2/2}}{(y^2/2)^2}\bigl(1+\bar S_{\t}(y)\bigr).$
\end{lemma}



\begin{lemma}\label{Lem: Diff}
For any stochastic process $(V_{\t}: \t>0)$ with continuously differentiable sample paths $\t\mapsto V_\t$, 
with derivative written as $\dot V_\t$,
$\E( V_{\t_2}-V_{\t_1})^2\leq (\t_2-\t_1)^2\sup_{\t\in[\t_1,\t_2]}\E \dot V_{\t}^2.
$\end{lemma}

\bibliographystyle{acm}
\bibliography{references}
\end{document}